\documentclass[11pt]{amsart}
\usepackage[utf8]{inputenc}
\usepackage{amssymb,amsmath, amsthm, mathabx}
\usepackage{color, enumerate}
\usepackage{hyperref}
\usepackage{caption}
\usepackage{subcaption}
\usepackage{a4wide}
\usepackage{accents}
\usepackage{tikz}
\usepackage{graphicx}
\usepackage{chngcntr}
\usepackage{datetime2}
\usepackage{scalerel}    

\usepackage{xifthen}
\def \w {\omega}

\def \wb{\wh{\omega}}

\def \wne{\wc{\omega}}

\def \G{\mathrm G}

\def \Gb{\wh{\G}}

\def \Gne{\wc{\G}}

\def \B{\mathrm B}

\def \Bcdf{\mathrm F}

\def \bcdf{\mathrm f}

\def \b{\mathrm b}

\def \cif{\varphi}

\def \tree{\sT}

\def\L{\mathcal{L}}

\def\Lb{\mathrm{L}}

\def\I{\mathcal{I}}

\def\Ib{\mathrm{I}}

\def\hor{\mathrm{hor}}
\def\ver{\mathrm{ver}}

\def \P{\mathbf P} 
\def\E{\bfE}

\def\Exp{\mathrm{Exp}}

\def\curv{\sigma}

\def\dd{\mathrm d}

\def\Eh{\mathrm{Z}}
\def\Ev{\mathrm{Z}}

\def\M{\mathrm M} 
\def\Shp{\gamma}

\def\Min{\zeta}
\def\Minx{\xi}

\newcommand{\bbR}{\mathbb{R}}

\newcommand{\bbZ}{\mathbb{Z}}

\newcommand{\bfE}{\mathbf{E}}

\newcommand{\sT}{\mathcal{T}}

\newcommand{\Fabs}[4]
{
\ifthenelse{\isempty{#2}}
{
{#1}_{#3}^{#4}
}
{
{#1}_{#3}^{#4}(#2)
}
}

\newcommand{\one}{\mathbf{1}}

\newcommand{\lc}{\lceil}
\newcommand{\rc}{\rceil}

\newcommand\var{\rm Var}

\def\dd{\mathrm d}

\definecolor{darkblue}{rgb}{.1,.1,.6}

\newcommand{\wt}[1]{\widetilde{#1}}
\newcommand{\wh}[1]{\widehat{#1}}
\newcommand{\wc}[1]{\widecheck{#1}}

\newtheorem{thm}{Theorem}[section]
\newtheorem{prop}[thm]{Proposition}
\newtheorem{lem}[thm]{Lemma}
\newtheorem{cor}[thm]{Corollary}

\theoremstyle{remark}
\newtheorem{rem}[thm]{Remark}

\newcommand{\be}{\begin{equation}}
\newcommand{\ee}{\end{equation}}

\usepackage[percent]{overpic}

\counterwithin{figure}{section}
\counterwithin{equation}{section}

\usepackage{todonotes, xargs}
\newcommandx{\note}[2][1=]{\todo[linecolor=yellow,backgroundcolor=yellow!25,bordercolor=yellow,#1]{#2}}

\allowdisplaybreaks

\title[Moderate deviations in the exponential LPP]
{Right-tail moderate deviations in the exponential last-passage percolation} 

\author[E.~Emrah]{Elnur Emrah}
\address{Elnur Emrah\\ KTH Royal Institute of Technology \\ Department of Mathematics \\ SE-100 44 Stockholm\\ Sweden.}
\email{elnur@kth.se}
\urladdr{https://sites.google.com/view/elnur-emrah}
\thanks{E.\ Emrah was partially supported by the grant KAW 2015.0270 from the Knut and Alice Wallenberg Foundation and by the Mathematical Sciences Department at Carnegie Mellon University through a postdoctoral position.} 

\author[C.~Janjigian]{Christopher Janjigian}
\address{Christopher Janjigian\\ University of Utah\\  Mathematics Department\\ 155S 1400E\\   Salt Lake City, UT 84112\\ USA.}
\email{janjigia@math.utah.edu}
\urladdr{http://www.math.utah.edu/~janjigia}
\thanks{C.\ Janjigian was partially supported by a postdoctoral grant from the Fondation Sciences Math\'ematiques de Paris while working at Universit\'e Paris Diderot.}

\author[T.~Sepp\"al\"ainen]{Timo Sepp\"al\"ainen}
\address{Timo Sepp\"al\"ainen\\ University of Wisconsin-Madison\\  Mathematics Department\\ Van Vleck Hall\\ 480 Lincoln Dr.\\   Madison WI 53706-1388\\ USA.}
\email{seppalai@math.wisc.edu}
\urladdr{http://www.math.wisc.edu/~seppalai}
\thanks{T.\ Sepp\"al\"ainen was partially supported by  National Science Foundation grant   DMS-1854619  and by the Wisconsin Alumni Research Foundation.} 

\thanks{}

\date{Last modified on \today\ at\ \filemodprinttime{\jobname}}
\date{\DTMnow. } 

\begin{document}

\begin{abstract} We study moderate deviations in the exponential corner growth model, both in the bulk setting and the increment-stationary setting. The main results are sharp right-tail bounds on the last-passage time and the exit point of the increment-stationary process.  The arguments utilize calculations with the stationary version and a moment generating function identity due to E. Rains, for which we give a short probabilistic proof. As applications of the deviation bounds, we derive upper bounds on the speed of distributional convergence in the Busemann function and competition interface limits. 
\end{abstract}


\maketitle

\tableofcontents

\section{Introduction}

This work derives sharp quantitative fluctuation bounds in the exponential corner  growth model (CGM).   
Our main results are new, but we also rederive an upper bound for the bulk  last-passage percolation (LPP) process that has been proved previously both with  probabilistic arguments through a coupling with  the exclusion process and with integrable probability.  
The proofs in the present paper utilize  the increment-stationary LPP  process (introduced below in Section \ref{SstatLPP}) and planarity.  

A difficult  technical hurdle in the probabilistic approach to LPP  has been  accessing the left tail of the last-passage time. The present paper demonstrates how this difficulty can be overcome in a number of interrelated problems in one of the most studied representatives of the KPZ class, namely, the exponential CGM. 

The starting point of our development is a moment generating function  identity from a preprint of E.\ Rains \cite{rain-00}, recorded as Proposition \ref{PLMId} below. We give a short probabilistic proof of this identity utilizing  the increment-stationary LPP process. Unlike the two proofs  in \cite{rain-00}, our argument does not require explicit formulas for the distribution of the last-passage times. 

  The contributions of this paper are listed below. The main results are items (i) and (ii).  Items (i)--(iv) are new results. \\[-10pt]   
\begin{enumerate}[\normalfont (\romannumeral1)] \itemsep=1pt 
\item An upper moderate deviation bound for the exit point in the increment-stationary LPP process (Theorem \ref{TExitPr}). This bound is sharp and matches a recently obtained lower bound (Remark \ref{RExitLB}). 
\item Sharp   
moderate  deviation bounds for the right tail   of the  increment-stationary  LPP process  (Theorem \ref{TisLppMDRate}).   

\item An upper bound on the speed of distributional convergence in  Busemann limits (Theorem \ref{TBuse}). 
\item An upper bound on the speed of distributional convergence of the  competition interface (Theorem \ref{TCif}). 
\item A rederivation of the sharp right tail moderate deviation bound 
of  the bulk LPP process without coarse graining or integrable probability (Theorem \ref{TLppMDBnd}). 
\end{enumerate}


\subsection*{Notation and conventions}  $X \sim \Exp(\lambda)$ for  $\lambda>0$ means that $X$ is a rate $\lambda$ exponential random variable, with mean, variance and moment generating function  $E(X)=\lambda^{-1}$, $\var(X)=\lambda^{-2}$ and    $E(e^{tX})=\lambda(t-\lambda)^{-1} \one_{\{t < \lambda\}}+ \infty \one_{\{t \ge \lambda\}}$ for $t \in \bbR$. 


$[n] = \{1, 2, \dotsc, n\}$ for $n \in \bbZ_{>0}$, and $[0] = \emptyset$.  
$\inf \emptyset = \infty$ and $\sup \emptyset = -\infty$. 



The value of the constants may change between subsequent steps of a derivation.

\section{Model and main results}

\subsection{Last-passage times and exit points}

An up-right path of length $l \in \bbZ_{>0}$ is a finite sequence $\pi = (\pi_i)_{i \in [l]}$ on $\bbZ^2$ such that $\pi_{i+1}-\pi_{i} \in \{(1, 0), (0, 1)\}$ for $i \in [l-1]$. Let $\Pi_{p, q}^{m, n}$ denote the set of all up-right paths $\pi = (\pi_i)_{i \in [l]}$ from $\pi_1 = (p, q) \in \bbZ^2$ to $\pi_l = (m, n) \in \bbZ^2$. 

Let $\{\w(i, j): i, j \in \bbZ_{>0}\}$ be independent, $\Exp(1)$-distributed random weights. Define the bulk last-passage time from $(p, q) \in \bbZ_{>0}^2$ to $(m, n) \in \bbZ_{>0}^2$ by 
\begin{align}
\G_{p, q}(m, n) = \max_{\pi \in \Pi_{p, q}^{m, n}} \sum_{(i, j) \in \pi} \w(i, j). \label{Elpp}
\end{align}
We  omit the subscript on the left-hand side when $p = q = 1$. For brevity, we also suppress the vertex $(m, n)$ from the notation in the computations below if there is no risk of confusion. 

We also consider last-passage times defined on $\bbZ_{\ge 0}^2$ with boundary weights: For $w > 0$ and $z < 1$, let $\{\wb^{w, z}(i, j): i, j \in \bbZ_{\ge 0}\}$ be independent random weights with marginal distributions given by $\wb^{w, z}(0, 0) = 0$, and 
\begin{align}
\wb^{w, z}(i, j) \sim \begin{cases} \Exp(1) \quad &\text{ if } i, j > 0 \\ \Exp(w) \quad &\text{ if } i > 0, j = 0\\ \Exp(1-z) \quad &\text{ if } i = 0, j > 0, \end{cases} \label{Ewbd}
\end{align} 
Then define the last-passage time from $(p, q) \in \bbZ_{\ge 0}^2$ to $(m, n) \in \bbZ_{\ge 0}^2$ by 
\begin{align}
\Gb^{w, z}_{p, q}(m, n) = \max_{\pi \in \Pi_{p, q}^{m, n}} \sum_{(i, j) \in \pi} \wb^{w, z}(i, j). \label{Elppbd}
\end{align}
Several notational simplifications will be employed. 
We drop one $z$ from the superscript when $w = z$ ($\Gb^{z}=\Gb^{z, z}$) and drop the subscript when $p=q=0$ ($\Gb^{w, z}=\Gb^{w, z}_{0,0}$).   When $p > 0$  the $\Exp(1-z)$  boundary weights  $\wb^{w, z}(0, j)$ do not   enter definition \eqref{Elppbd}, and hence  we  omit $z$  (for example,  $\Gb^{w, z}_{1,0} =\Gb^{w}_{1,0}$). Similarly  $w$ is omitted when $q > 0$.   The convention in \eqref{Ewbd} is chosen so that $\Gb^{z}$ is the increment-stationary LPP process.  
 
Define the (maximal) horizontal and vertical exit points by 
\begin{align}
\Eh^{w, z, \hor}(m, n) &= \max \bigl( 0, \max \{k \in [m]: \Gb^{w, z}(m, n) = \Gb^{w}(k, 0) + \G_{k, 1}(m, n)\}\bigr)  \label{EEh}\\
\Ev^{w, z, \ver}(m, n) &= \max \bigl( 0, \max \{l \in [n]: \Gb^{w, z}(m, n) = \Gb^{z}(0, l) + \G_{1, l}(m, n)\}\bigr). \label{EEv} 
\end{align}
Since the weights in \eqref{Ewbd} have continuous distributions, a.s.\ there exists a unique path (the \emph{geodesic}) $\pi_{p, q}^{m, n} \in \Pi_{p, q}^{m, n}$ that maximizes the right-hand side of \eqref{Elppbd}. $\Eh^{w, z, \hor}(m, n)$ equals the horizontal coordinate of the last vertex $\pi_{p, q}^{m, n}$ visits on the horizontal axis, and similary for $\Ev^{w, z, \ver}(m, n)$. A.s., exactly one of $\Eh^{w, z, \hor}(m, n)$ and $\Ev^{w, z, \ver}(m, n)$ is nonzero. 

All the weights can be coupled  through a single collection  $\{\eta(i, j): i, j \in \bbZ_{\ge 0}\}$ of  i.i.d.\ $\Exp(1)$-distributed random real numbers by setting 
\begin{align}
\w(i, j) &= \eta(i, j) \quad \text{ for } i, j \in \bbZ_{>0} \text{ and } \label{ECplBlk}\\
\w^{w, z}(i, j) &= \eta(i, j)\bigg(\one_{\{i, j > 0\}} + \frac{\one_{\{i > 0, \hspace{0.5pt} j = 0\}}}{w} + \frac{\one_{\{i=0, \hspace{0.5pt} j> 0\}}}{1-z}\bigg) \quad \text{ for } i,j \in \bbZ_{\ge 0} . \label{ECplBlkBd}
\end{align}

\subsection{Increment-stationary last-passage percolation}\label{SstatLPP}

A down-right path of length $l \in \bbZ_{>0}$ is a finite sequence $\pi = (\pi_{k})_{k \in [l]}$ in $\bbZ^2$ such that $\pi_{k+1}-\pi_{k} \in \{(1, 0), (0, -1)\}$ for $k \in [l-1]$. By virtue of a version of Burke's theorem for LPP \cite{bala-cato-sepp}, for each $z \in (0, 1)$, the increments of the $\G^z$-process along any down-right path $\pi = (\pi_k)_{k \in [l]}$ in $\bbZ_{\ge 0}^2$ have this property: 
\begin{align}
\label{EBurke}
\begin{split}
&\{\G^z(\pi_{k+1})-\G^{z}(\pi_{k}): k \in [l-1]\} \text{ are jointly independent with marginals given by } \\
&\G^z(i, n)-\G^z(i-1, n) \sim \Exp\{z\} \quad \text{and} \quad \G^{z}(m, j)-\G^{z}(m, j-1) \sim \Exp\{1-z\}
\end{split}
\end{align}
for $m, n \in \bbZ_{\ge 0}$ and $i, j \in \bbZ_{>0}$. As a consequence, the $\Gb^z$-process is increment-stationary in the sense that
\begin{align}
\Gb^z(m+p, n+q) - \Gb^z(p, q) \stackrel{\text{dist}}{=} \Gb^{z}(m, n) \quad \text{ for } m, n, p, q \in \bbZ_{\ge 0}. \label{Eincst}
\end{align}
Define 
\begin{align}
\M^z(x, y) = \frac{x}{z} + \frac{y}{1-z} \quad \text{ for } x, y \in \bbR_{\ge0} \text{ and } z \in (0, 1). \label{EM}
\end{align}
It follows from \eqref{Eincst} that 
\begin{align*}
\E[\Gb^{z}(m, n)] = \M^z(m, n) \quad \text{ for } m, n \in \bbZ_{\ge 0}.
\end{align*} 
The curve $z \mapsto \M^z(x, y)$ for $z \in (0, 1)$ and some fixed $x, y \in \bbR_{>0}$ is plotted in Figure \ref{FMplot}. 

 The \emph{shape function} of  the $\G$-process is given   by 
\begin{align}
\Shp(x, y) = \inf_{z \in (0, 1)} \M^z(x, y) = (\sqrt{x}+\sqrt{y})^2 \quad \text{ for } x, y \in \bbR_{\ge0}. \label{EShp}
\end{align}
A seminal result of H.\ Rost \cite{rost} identifies \eqref{EShp} as the following   limit: 
\begin{align}
\lim_{N\to\infty}N^{-1}\G(\lc Nx \rc, \lc Ny \rc) = \Shp(x, y) \quad \text{ for } x, y \in \bbR_{>0} \quad \P\text{-a.s.}
\label{EShpLim}
\end{align}

The unique minimizer in \eqref{EShp} is given by 
\begin{align}
\Min(x, y) = \frac{\sqrt{x}}{\sqrt{x}+\sqrt{y}} \quad \text{ for } x, y > 0. \label{EMin}
\end{align}
This defines a bijection between  directions (unit vectors) in $\bbR_{>0}^2$ and  the interval $(0, 1)$.  When  $z = \Min(x, y)$, $(x, y)$ is the   {\it characteristic direction}  of the $\Gb^z$-process.  
 In this direction  the geodesic from the origin exits the boundary within a  submacroscopic neighborhood of the origin. More precisely, for fixed $x, y > 0$ and $z \in (0, 1)$, 
\be\label{EChDir}\begin{aligned}
&\lim_{N\to\infty} N^{-1}\max \bigl\{\Eh^{z, \hor}\bigl(\lc Nx \rc, \lc Ny \rc\bigr), \Ev^{z, \ver}\bigl(\lc Nx \rc, \lc Ny \rc\bigr)\bigr\} \stackrel{\text{a.s.}}{=} 0 \\
&\qquad \quad \text{if and only if} \quad z = \Min(x, y). 
\end{aligned}\ee
 Theorem \ref{TExitPr} below implies the \emph{if} part. 

\begin{figure}
\centering
\begin{overpic}[scale=0.6]{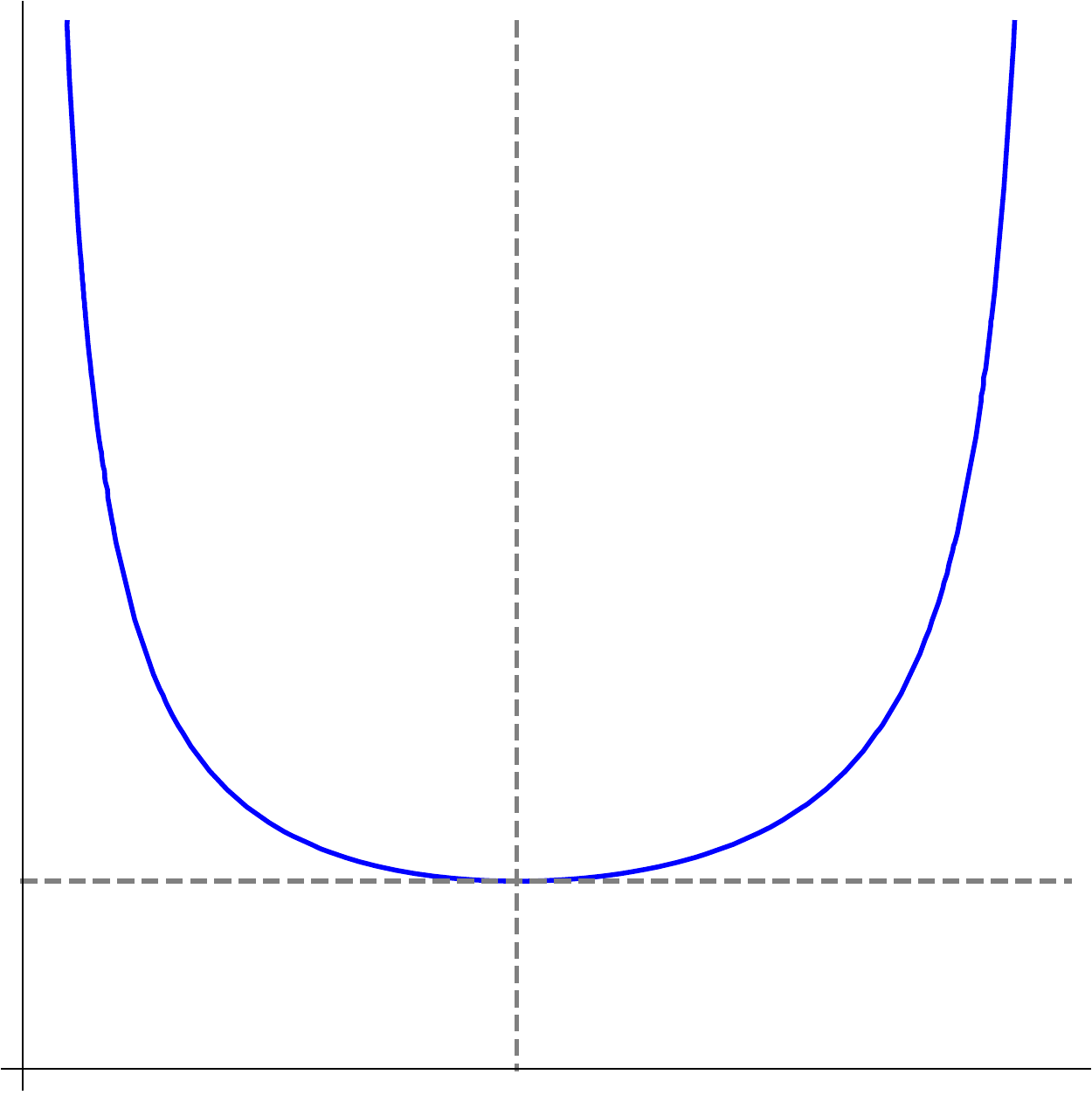}
\put (-13, 18){$\displaystyle \Shp(x, y)$}
\put (41, -3){$\displaystyle \Min(x, y)$}
\end{overpic}
\vspace{0.2in}
\caption{The plot of the curve $z \mapsto \M^z(x, y), z \in (0, 1)$ (blue) with $x = 4$ and $y = 5$. The minimum value $\Shp(x, y)$ and the unique minimizer $\Min(x, y)$ are indicated.}
\label{FMplot}
\end{figure}

\subsection{The l.m.g.f.\ of the LPP process with boundary conditions}

Define
\begin{align}
\Lb^{w, z}(x, y) = x\log \bigg(\frac{w}{z}\bigg)+y \log \bigg(\frac{1-z}{1-w}\bigg) \quad \text{ for } x, y \in \bbR_{\ge0} \text{ and } w, z \in (0, 1). \label{ELbd}
\end{align}
We record an identity that links \eqref{ELbd} to the l.m.g.f.\ of the $\Gb^{w, z}$-process. 
\begin{prop}[\cite{rain-00}]
\label{PLMId}
Let $m, n \in \bbZ_{\ge 0}$ and $w, z \in (0, 1)$. Then 
\begin{align*}
\log \E[\exp\{(w-z) \Gb^{w, z}(m, n)\}] = \Lb^{w, z}(m, n). 
\end{align*}
\end{prop}
\begin{proof}
Recall \eqref{Ewbd}.  The second equality below  uses the product inside the expectation  as a Radon-Nikodym derivative.  This changes the rates of the exponential weights on the vertices $\{(i, 0): i \in [m]\}$ from $w$ to $z$. Then we use shift invariance \eqref{Eincst}. 
\begin{align*}
\log \E[e^{(w-z) \Gb^{w, z}(m, n)}]  
&=  \log \E\biggl[ \biggl(\;  \prod_{i=1}^m e^{(w-z) \wb^{w, z}(i,0)} \biggr)    e^{(w-z)(\Gb^{w,z}(m, n)-\Gb^{w,z}(m, 0))}\biggr] \\
&= m \log \bigg(\frac{w}{z}\bigg) + \log \E[e^{(w-z)(\Gb^{z}(m, n)-\Gb^{z}(m, 0))}] \\
&= m \log \bigg(\frac{w}{z}\bigg) + \log \E[e^{(w-z)\Gb^{z}(0, n)}] \\
&= m \log \bigg(\frac{w}{z}\bigg) + n \log \bigg(\frac{1-z}{1-w}\bigg) = \Lb^{w, z}(m, n). \qedhere
\end{align*}
\end{proof}

A more general form of Proposition \ref{PLMId} appeared in a preprint of E.\ Rains \cite[Corollaries 3.3--3.4]{rain-00}. His version covers mixtures of exponential and Poisson LPP, and mixtures of geometric and Bernoulli LPP, and allows some inhomogeneity in parameters. \cite{rain-00} provides two proofs for the identity, both of which ultimately rely on exact determinantal formulas for the distribution of the last-passage times developed in \cite{baik-rain-01a}. 
The short argument above extends readily to inhomogeneous  exponential  LPP but we have not attempted to verify this in the full setting of \cite{rain-00}. 


\subsection{Right tail moderate deviation upper bound for the bulk LPP}

Our first result is a rederivation of a well-known upper bound \cite{joha, sepp98ebp} for the right tail deviations of the bulk LPP process. We obtain this bound as a fairly immediate consequence of the easier ``$\le$'' half of Proposition \ref{PLMId} via the exponential Markov's inequality and some simple estimates. 

In the next statement and beyond, to ensure uniformity, we often restrict to the vertices inside the cone  
\begin{align}
\label{ESc}
S_\delta = \{(x, y) \in \bbR_{>0}^2: x \ge \delta y \text{ and } y \ge \delta x\}.
\end{align}
The function  
\begin{align}
\curv(x, y) = \bigg(\frac{\Shp(x, y)}{\Min(x, y) (1-\Min(x, y))}\bigg)^{1/3} = \frac{(\sqrt{x}+\sqrt{y}\hspace{1pt})^{4/3}}{x^{1/6}y^{1/6}}
 \label{Ecurv}
\end{align}
defined for $x, y \in \bbR_{>0}$ 
acts as a scaling factor below.   It is connected to \eqref{EM} through 
\begin{align}
\curv(x, y)^3 = \frac{1}{2}\partial_z^2\bigg|_{z = \Min(x, y)} \bigg\{\M^{z}(x, y)\bigg\} \quad \text{ for } x, y \in \bbR_{>0}. \label{Ecurv2}
\end{align}
Note that, when $x$ and $y$ are both large, $\curv(x, y)$ is of order $(x+y)^{1/3}$. 

\begin{thm}
\label{TLppMDBnd}
Fix $\delta > 0$. There exist constants $c_0, C_0 > 0$ that depend only on $\delta$ such that 
\begin{align*}
\log \P\{\G(m, n) \ge \Shp(m, n) + \curv(m, n) s\} \le -\,\frac{4s^{3/2}}{3}+\frac{C_0s^2}{(m+n)^{1/3}}
\end{align*}
for $(m, n) \in S_\delta\cap\bbZ_{>0}^2$ and $s \in [0, c_0 (m+n)^{2/3}]$. 
\end{thm}

This bound was first deduced in \cite[p.\ 622]{sepp98ebp} along the diagonal direction from an explicit computation of the right tail large deviation rate function (recalled in \eqref{ESepp} below), based on couplings with 
the stationary totally asymmetric simple exclusion process,
the superadditivity of the last-passage times and coarse graining arguments. As observed in \cite{sepp98ebp}, by virtue of superadditivity, the rate function serves as an upper bound on the right tail deviations, not just asymptotically but also for finite $(m,n)$. Therefore, Theorem \eqref{TLppMDBnd} can be obtained from an expansion of the rate function around the shape function. An alternative computation of the rate function utilizing an exact distributional formula \cite[Proposition 1.4]{joha} for the last-passage times appeared in \cite[Theorem 1.6]{joha}. Although the rate in \cite[(1.21)]{joha} is not explicit, it can presumably be made so as in the geometric case \cite[(2.21)]{joha}.  Then superadditivity and an expansion similar to \cite[(2.23)]{joha} would establish the theorem above. 

Compared with \cite{joha, sepp98ebp}, our proof does not require a distributional formula or   superadditivity. Consequently, it adapts  more readily  to various directed percolation and polymer models that possess tractable increment-stationary versions. Eliminating the need for superadditivity can be  useful in inhomogeneous settings (such as those in \cite{emra-janj-sepp-19-shp-long})  where the absence of translation invariance of the weights prevents the use of superadditive ergodicity. 

The bound in Theorem \ref{TLppMDBnd} is expected to capture accurately the behavior of the right tail of the bulk LPP process,  based on the same prediction \cite[(1.19)]{baik-etal-01} made for the LPP with i.i.d.\ geometric weights. In \cite{baik-etal-01}, the authors carried out a Riemann-Hilbert analysis to obtain precise asymptotics for the left tail deviations of the geometric LPP, and suggested that a similar analysis would yield the asymptotics for the right tail. Presumably, one can also adapt the analysis to the present setting. 
An  important further step  would be the derivation of the matching lower bound  entirely from the increment-stationary LPP.  This appears to require an entirely new idea. 

\subsection{Right tail moderate deviation rate for the increment-stationary LPP}

\begin{thm}
\label{TisLppMDRate}
Fix $\delta > 0$, $K \ge 0$ and $p > 0$. Let $(m, n) \in S_\delta\cap\bbZ_{>0}^2$ and $z \in (0, 1)$ with 
\begin{align}|z-\Min(m, n)| \le K(m+n)^{-1/3+p}. \label{ENearMin}\end{align} 
The following statements hold. 
\begin{enumerate}[\normalfont (a)]
\item Fix $s_0 > 0$. There are constants $C_0 = C_0(\delta) > 0$, $\epsilon_0 = \epsilon_0(\delta) > 0$ and $N_0 = N_0(\delta, K, p, s_0) > 0$ such that 
\begin{align*}
\log \P\{\Gb^{z}(m, n) \ge \Shp(m, n) + \curv(m, n)s\} \le  -\frac{2s^{3/2}}{3}+\log 2 + \frac{C_0Ks}{(m+n)^p} + \frac{C_0s^2}{(m+n)^{1/3}}
\end{align*}
whenever $s \in [s_0, \epsilon_0(m+n)^{2/3}]$ and $m, n \ge N_0$. 
\item There are constants $C_1 = C_1(\delta) > 0$, $\epsilon_1 = \epsilon_1(\delta) > 0$, $N_1 = N_1(\delta, K, p) > 0$, $s_1 > 0$ and $\eta > 0$ such that for $s \in [s_1, \epsilon_1(m+n)^{2/3}]$ and $m, n \ge N_1$,
\begin{align*}
\log \P\{\Gb^{z}(m, n) \ge \Shp(m, n) + \curv(m, n)s\} \ge  -\frac{2s^{3/2}}{3}-\eta s -\frac{C_1s^2}{(m+n)^{1/3}}. 
\end{align*}
\end{enumerate}
\end{thm}





We state a limiting version as a corollary. 

\begin{cor}
\label{CisLppMDRate}
Fix $p \in (0, 2/3)$, $x, y, s \in \bbR_{>0}$ and write $\Min = \Min(x, y)$. Then 
\begin{align*}
\lim_{N \rightarrow \infty} N^{-3p/2}\log \P\{\Gb^{\Min}(\lc Nx \rc, \lc Ny \rc) \ge N \Shp(x, y) + N^{1/3+p}\curv(x, y)s\}= -\frac{2s^{3/2}}{3}. 
\end{align*}
\end{cor}


\smallskip 

\subsection{Upper bounds for the exit points}

The next pair of results provides upper bounds for the deviations of the exit points \eqref{EEh}--\eqref{EEv} in the increment-stationary case $w = z$. We consider separately two regimes distinguished by the proximity of the $z$-parameter to the minimizer \eqref{EMin}.

\subsubsection{Around the characteristic direction {\rm(}small $|z-\Min|$ regime{\rm)}}

When $m, n \in \bbZ_{>0}$ are both large and $z \in (0, 1)$ is sufficiently close to $\Min = \Min(m, n)$, the geodesic $\pi_{0, 0}^{m, n, z}$ from the origin to $(m, n) \in \bbZ_{>0}^2$ typically visits order $(m+n)^{2/3}$ many vertices on the boundary before entering into the bulk. For large $s > 0$, the probability that one exit point is at least distance $s(m+n)^{2/3}$ away from the origin is expected to be of order $\exp\{-cs^3\}$ for some constant $c > 0$, as explained on p.\ 7 in \cite{shen-sepp-19}. 

The next theorem  establishes the expected  upper bound.  The best previous upper bound accessible without   integrable probability was polynomial of order $s^{-3}$ \cite[Proposition 5.9]{sepp-cgm-18}.   Remarks in \cite[Remark 1.3]{basu-sark-sly-19} and \cite[p.\ 7]{shen-sepp-19} suggest that the result can also be derived from a  left tail bound for the largest eigenvalue of the Laguerre ensemble from \cite{ledo-ride-10}.  

\begin{thm}
\label{TExitPr}
For $\delta > 0$ and $K \ge 0$ there exist finite constants $c_0 = c_0(\delta) > 0$, $N_0 = N_0(\delta, K) > 0$ and $s_0 = s_0(\delta, K) > 0$ such that the following holds: 
\begin{align*}
\P\{\max \{\Eh^{z, \hor}(m, n), \Ev^{z, \ver}(m, n)\} \ge s(m+n)^{2/3}\} \le \exp\{-c_0 s^3\}
\end{align*}
for all $(m, n) \in S_\delta \cap \bbZ_{\ge N_0}^2$, $s \ge s_0$, and $z \in (0, 1)$  such that  $|z-\Min(m, n)| \le K(m+n)^{-1/3}$.  
\end{thm}


\begin{rem}  \label{RExitLB}  \cite[Theorem 4.4]{shen-sepp-19} gives the following  lower bound that matches the upper bound of Theorem \ref{TExitPr}.  For $\delta>0$ there exist finite positive constants $c_1$, $N_1$, $s_1$, and $s_2$, all functions of $\delta$,  such that the following holds: 
\be\label{Exit_LB} 
\P\{\max \{\Eh^{z, \hor}(m, n), \Ev^{z, \ver}(m, n)\} \ge   s\min(m^{2/3}, n^{2/3}) 
\} \ge \exp\{-c_1s^3\}
\ee
for all $(m, n) \in \bbZ_{\ge N_1}^2$ and all $z \in [\delta,  1-\delta]$,  whenever $s_1\le s \le s_2\min(m^{1/3}, n^{1/3})$. 
This bound was proved without integrable probability, by an adaptation of a  change-of-measure argument from \cite{bala-sepp-aom, sepp-cgm-18}. 
\end{rem}  

The exit point upper bound serves as an input to various  proofs in  the non-integrable literature on LPP.  Some of these results  can potentially be strengthened through Theorem \ref{TExitPr}. For example, the upper bound in \cite[Theorem 2.3]{shen-sepp-19} can presumably be made optimal. 

\subsubsection{Off-characteristic directions {\rm(}large $|z-\Min|$ regime{\rm)}}

For large $m, n \in \bbZ_{>0}$ and $z \in (0, 1)$ sufficiently away from the minimizer $\Min = \Min(m, n)$, with high probability, the geodesic $\pi_{0, 0}^{m, n, z}$ from the origin to vertex $(m, n)$ prefers to move horizontally if $z < \Min$ and vertically if $z > \Min$. The following proposition provides an upper bound on the probability of the geodesic taking the less likely initial direction. The statement is a main ingredient in the proofs of several results below and may be of independent interest. 

\begin{prop}
\label{PGeoInStep}
Fix $\delta > 0$. Let $(m, n) \in \bbZ_{>0}^2 \cap S_\delta$, $\Min = \Min(m, n)$ and $z \in (0, 1)$. There exists a constant $c_0 = c_0(\delta) > 0$ such that the following statements hold. 
\begin{enumerate}[\normalfont (a)]
\item If $z < \Min$ then
\begin{align*}
\P\{\Ev^{z, \ver}(m, n) > 0\} \le \exp\{-c_0(m+n)(\Min-z)^3\}. 
\end{align*}
\item If $z > \Min$ then 
\begin{align*}
\P\{\Eh^{z, \hor}(m, n) > 0\} \le \exp\{-c_0(m+n)(z-\Min)^3\}. 
\end{align*}
\end{enumerate}
\end{prop}

\subsection{Speed of the distributional convergence to Busemann functions}

Denote the increments of the bulk LPP process with respect to the initial point by 
\begin{align}
\B_{i, j}^{\hor}(m, n) &= \G_{i, j}(m, n)-\G_{i+1, j}(m, n) \label{Ehi}\\ 
\B_{i, j}^{\ver}(m, n) &= \G_{i, j}(m, n)-\G_{i, j+1}(m, n)  \label{Evi}
\end{align}
for $m, n \in \bbZ_{>0}$, $i \in [m]$ and $j \in [n]$. Due to the convention from \eqref{Elpp}, these increments are equal to  $+\infty$ when $i = m$ and $j =n$, respectively. 

The a.s.\ directional limits of \eqref{Ehi} and \eqref{Evi} are known to exist.  For any given direction vector $(x, y) \in \bbR_{>0}^2$  there exists a stationary stochastic process $\{ \b_{i, j}^{\hor}(x, y), \b_{i, j}^{\ver}(x, y): i,j\in\bbZ_{>0}\}$ and  an event of full probability  on which the    limits   
\begin{align}
\lim_{N \rightarrow \infty}\B_{i, j}^{\hor}(m_N, n_N) &= \b_{i, j}^{\hor}(x, y) \quad \text{ and } \quad \lim_{N \rightarrow \infty}\B_{i, j}^{\ver}(m_N, n_N) = \b_{i, j}^{\ver}(x, y) 
\label{EBuse}
\end{align}
hold  for all $(i, j) \in \bbZ^2_{>0}$ and for all sequences  $\{(m_N, n_N)\}_{N\ge 1}\subset\bbZ^2_{>0}$ such that 
$\min \{m_N, n_N\} \rightarrow \infty$ and $m_N/n_N\to x/y$. 

Following the approach of C.\ Newman \cite{newm-icm-95},  the existence of these  limits was first proved by P.\ A.\ Ferrari and L.\ Pimentel in \cite{ferr-pime-05} for a deterministic set of directions of full Lebesgue measure, see the remarks following Propositions 7-9 therein. The result was subsequently extended to all directions by  D.\ Coupier \cite[Theorem 1]{coup-11}. The limits were later established in broad generality in a joint work of the third author \cite[Theorem 3.1]{geor-rass-sepp-17-buse}. Their result covers LPP with i.i.d.\ weights bounded from below and of finite $p$th moment for some $p > 2$, and applies to all directions except those into the closed (possibly degenerate) flat regions of the shape function with at least one boundary direction where the shape function is not differentiable. 
These restrictions on directions disappear with exponential weights since the shape function \eqref{EShp} is differentiable and strictly concave on $\bbR_{>0}^2$. Follow-up work by the authors \cite{emra-janj-sepp-19-buse} will consider the structure of these limits in a non-stationary generalization of the model studied in this paper.

The limits \eqref{EBuse} are examples of \emph{Busemann functions} evaluated, respectively, at pairs $((i, j), (i+1, j))$ and $((i, j), (i, j+1))$ of successive vertices. The Busemann function of the $\G$-process in direction of $(x, y)$ can be fully defined as the a.s.\ $(x, y)$-directional limit of the discrete gradients of the $\G$-process with respect to the initial point. We do not formally introduce the more general notion here since \eqref{EBuse} suffices for the sequel. 

The following distributional properties of the Busemann functions were obtained in \cite[Lemma 3.3]{cato-pime-13}. The marginal distributions are given by 
\begin{align}
\b^{\hor}_{i, j}(x, y) \sim \Exp\{\Min(x, y)\} \quad \text{ and } \quad \b^{\ver}_{i, j}(x, y) \sim \Exp\{1-\Min(x, y)\} \label{EBuseDis}
\end{align}
for $(i, j) \in \bbZ_{>0}^2$ and $(x, y) \in \bbR_{>0}^2$. Furthermore, for any down-right path $\pi = (\pi_k)_{k \in [l]}$ in $\bbZ_{>0}^2$, 
\begin{align}
\label{EBuseInd}
\begin{split}
&\{\b_{\pi_k}^{\hor}(x, y): k \in [l-1] \text{ and } \pi_{k+1} = \pi_k + (1, 0)\} \\
&\cup \{\b_{\pi_k}^{\ver}(x, y):k \in [l] \smallsetminus \{1\} \text{ and } \pi_{k-1} = \pi_k+(0, 1)\} \text{ are jointly independent.} 
\end{split}
\end{align} 

For $m, n \in \bbZ_{>0}$ and $k, l, p, q \in \bbZ_{\ge 0}$ with $k \le m$ and $l \le n$, the joint c.d.f.\ of the pre-limit  increments
\begin{align}
\{\B_{i+p, 1+q}^{\hor}(m+p, n+q): i \in [k]\} \cup \{\B_{1+p, j+q}^{\ver}(m+p, n+q): j \in [l]\} \label{EInc}
\end{align}
is denoted  for $s = (s_i)_{i \in [k]} \in \bbR^k$ and $t = (t_j)_{j \in [l]} \in \bbR^l$ by 
\begin{align}
\Bcdf_{k, l}^{m, n}(s, t) = \P\{\B_{i, 1}^{\hor}(m, n) \le s_i \text{ for } i \in [k] \text{ and } \B_{1, j}^{\ver}(m, n) \le t_j \text{ for } j \in [l]\}. \label{EBcdf}
\end{align} 
There is no dependence on $p, q$ in \eqref{EBcdf} due to the distributional translation invariance of the weights. Also denote the joint c.d.f.\ of the boundary weights \eqref{Ewbd} in the case $w=z$  by 
\begin{align}
\bcdf_{k, l}^{z}(s, t) &= \P\{\wb^{z}(i, 0) \le s_i \text{ for } i \in [k] \text{ and }  \wb^{z}(0, j) \le t_j \text{ for } j \in [l]\} \label{Ewbcdf} \\
&= \prod_{i \in [k]}(1-\exp\{-s_i^+z\}) \prod_{j \in [l]}(1-\exp\{-t_j^+(1-z)\}) \quad \text{ for } z \in (0, 1). \nonumber
\end{align}
In light of \eqref{EBuse}, \eqref{EBuseDis} and \eqref{EBuseInd}, for fixed $(k, l) \in \bbZ_{\ge 0}^2$ and $(x, y) \in \bbR_{>0}^2$, \eqref{EBcdf} converges pointwise to the limiting c.d.f.\  
\begin{align}
\lim_{N \rightarrow \infty} \Bcdf_{k, l}^{\lc Nx \rc, \lc Ny \rc}(s, t) = f_{k, l}^{\Min(x, y)}(s, t). \label{EcdfLim}
\end{align}
A natural problem is then to quantify the speed of convergence. The next result provides some bounds in this direction. 

The result concerns the following modifications of the c.d.f.\ \eqref{EBcdf}.
\begin{align}
\Bcdf^{m, n, \hor}_{k, l}(s, t) &= \P\{\B_{i, 1}^{\hor}(m, n) > s_i \text{ for } i \in [k] \text{ and } \B_{1, j}^{\ver}(m, n) \le t_j \text{ for } j \in [l]\} \label{EBcdfh} \\
\Bcdf^{m, n, \ver}_{k, l}(s, t) &= \P\{\B_{i, 1}^{\hor}(m, n) \le s_i \text{ for } i \in [k] \text{ and } \B_{1, j}^{\ver}(m, n) > t_j \text{ for } j \in [l]\} \label{EBcdfv}
\end{align}
for $m, n \in \bbZ_{>0}$, $k \in [m] \cup \{0\}$, $l \in [n] \cup \{0\}$, $s = (s_i)_{i \in [k]} \in \bbR^k$ and $t = (t_j)_{j \in [l]} \in \bbR^l$. Together these functions encode the same information as those in \eqref{EBcdf}. However, due to the monotonicity structure of the increments \eqref{Ehi}--\eqref{Evi} described in Lemma \ref{LCros}, it turns out easier to work with \eqref{EBcdfh}--\eqref{EBcdfv} 
than with \eqref{EBcdf}. In this same vein, introduce the functions 
\begin{align}
\bcdf_{k, l}^{z, \hor}(s, t) &= \P\{\wb^{z}(i, 0) > s_i \text{ for } i \in [k] \text{ and }  \wb^{z}(0, j) \le t_j \text{ for } j \in [l]\label{Ebcdfh}\\
&= \prod_{i \in [k]}\exp\{-s_i^+z\} \prod_{j \in [l]}(1-\exp\{-t_j^+(1-z)\}) \nonumber \\
\bcdf_{k, l}^{z, \ver}(s, t) &= \P\{\wb^{z}(i, 0) \le s_i \text{ for } i \in [k] \text{ and }  \wb^{z}(0, j) > t_j \text{ for } j \in [l]\}\label{Ebcdfv}\\
&= \prod_{i \in [k]}(1-\exp\{-s_i^+z\}) \prod_{j \in [l]}\exp\{-t_j^+(1-z)\}\nonumber
\end{align}
for $z \in (0, 1)$ and $k, l, s, t$ as above. 

\begin{thm}
\label{TBuse}
Let $\delta > 0$ and $\epsilon_0 > 0$. There exist constants $N_0 = N_0(\delta, \epsilon_0) > 0$ and $c_0 = c_0(\delta) > 0$ such that 
\begin{align*}
|\Bcdf_{k, l}^{m, n, \hor}(s, t)-\bcdf_{k, l}^{\Min(m, n), \hor}(s, t)| \le c_0 (1+\log l) \bigg\{\frac{\log (m+n)}{(m+n)}\bigg\}^{1/3}\\
|\Bcdf_{k, l}^{m, n, \ver}(s, t)-\bcdf_{k, l}^{\Min(m, n), \ver}(s, t)| \le c_0 (1+\log k) \bigg\{\frac{\log (m+n)}{(m+n)}\bigg\}^{1/3}
\end{align*}
whenever $(m, n) \in S_\delta \cap \bbZ_{\ge N_0}^2$, $k \in [m-1]\cup\{0\}$, $l \in [n-1]\cup\{0\}$ with $k, l \le \epsilon_0(m+n)^{2/3}$, $s \in \bbR^k$ and $t \in \bbR^l$. 
\end{thm}

\subsection{Speed of the distributional convergence to the competition interface}
For the definitions in this subsection, restrict to the full probability event on which the geodesic $\pi^{m, n}$ from $(1, 1)$ to $(m, n)$ is unique for all $m, n \in \bbZ_{>0}$. Partition $\bbZ_{>0}^2 \smallsetminus \{(1, 1)\}$ into the subsets
\begin{align}
\tree^{\hor} &= \{(m, n) \in \bbZ_{>0}^2: (2, 1) \in \pi^{m, n}\} = \{(m, n) \in \bbZ_{>0}^2: \G_{2, 1}(m, n) > \G_{1, 2}(m, n)\} \label{ETh} \\
\tree^{\ver} &= \{(m, n) \in \bbZ_{>0}^2: (1, 2) \in \pi^{m, n}\} = \{(m, n) \in \bbZ_{>0}^2: \G_{2, 1}(m, n) < \G_{1, 2}(m, n)\}. \label{ETv}
\end{align}
See Figure \ref{Fcif}. As a consequence of planarity and the uniqueness of geodesics,  the sets above  enjoy the following structure:
\begin{align}
(k, l) \in \tree^{\hor} &\text{ implies that } (i, j) \in \tree^{\hor} \text{ for } i \in \bbZ_{\ge k} \text{ and } j \in [l] \label{EThmon}\\
(k, l) \in \tree^{\ver} &\text{ implies that } (i, j) \in \tree^{\ver} \text{ for } i \in [k] \text{ and } j \in \bbZ_{\ge l}. \label{ETvmon}
\end{align}

The \emph{competition interface} (depicted in Figure \ref{Fcif}) is a notion of a boundary between $\tree^{\hor}$ and $\tree^{\ver}$ introduced by P.\ A.\ Ferrari and L.\ Pimentel in \cite{ferr-pime-05}. One precise definition of it is as the unique sequence $\cif = (\cif_n)_{n \in \bbZ_{>0}} = (\cif_n^{\hor}, \cif_n^{\ver})_{n \in \bbZ_{>0}}$ in $\bbZ_{>0}^2$ such that 
\begin{align}
(\cif_n^{\hor}+1, \cif_n^{\ver}) \in \tree^{\hor} \quad (\cif_n^{\hor}, \cif_n^{\ver}+1) \in \tree^{\ver} \quad \text{ and } \quad \cif_n^{\hor}+\cif_n^{\ver} = n+1 \label{Ecif} 
\end{align} 
for $n \in \bbZ_{>0}$. The existence and uniqueness of $\cif$ can be seen from properties \eqref{EThmon}--\eqref{ETvmon}. The original definition from \cite{ferr-pime-05} describes the competition interface recursively as follows: 
\begin{align}
\cif_1^{\hor} =1, \quad \cif_n^{\hor} = \cif_{n-1}^{\hor} + \one\{\G(\cif_{n-1}^{\hor}+1, \cif_{n-1}^{\ver}) < \G(\cif_{n-1}^{\hor}, \cif_{n-1}^{\ver}+1)\} \label{Ecifh}\\
\cif_1^{\ver} = 1, \quad \cif_n^{\ver} =  \cif_{n-1}^{\ver} + \one\{\G(\cif_{n-1}^{\hor}+1, \cif_{n-1}^{\ver}) > \G(\cif_{n-1}^{\hor}, \cif_{n-1}^{\ver}+1)\} \label{Ecifv}
\end{align}
for $n \in \bbZ_{>1}$. The equivalence of \eqref{Ecif} and \eqref{Ecifh}--\eqref{Ecifv} can be verified by induction. 

\begin{figure}
\centering
\begin{overpic}[scale=0.8]{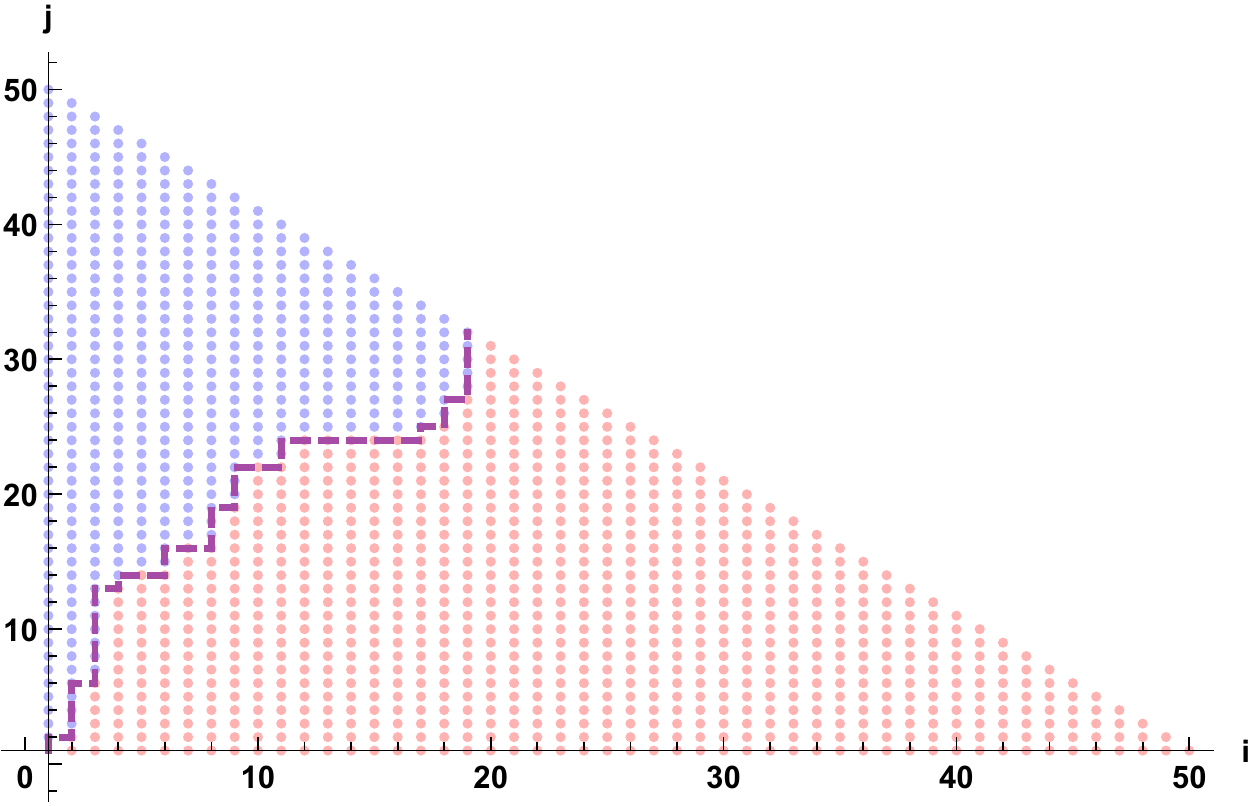}
\end{overpic}
\vspace{0.2in}
\caption{A simulation of the first $N = 50$ steps of the competition interface (purple). The vertices in $\tree^{\hor}$ (light red) and $\tree^{\ver}$ (light blue) below the diagonal $i+j = N+1$ are shown.}
\label{Fcif}
\end{figure}

Theorem 1 of \cite{ferr-pime-05}
proves that the rescaled competition interface $n^{-1}\cif_n$ converges a.s.\ as $n \rightarrow \infty$ and identifies the limit distribution. The existence of the a.s.\ limit has been later proved for i.i.d.\ weights bounded from below and with continuous marginal  distributions of finite $p$th moment for some $p > 2$, and under the assumption that the shape function is differentiable at the endpoints of its linear segments \cite[Theorem 2.6]{geor-rass-sepp-17-geod}. 

The distributional limit of the competition interface can be phrased in our notation as follows: For $x \in [0, 1]$, 
\begin{align}
\lim_{n \rightarrow \infty} \P\{\cif_n^{\hor} \le nx \} = \frac{\sqrt{x}}{\sqrt{x}+\sqrt{1-x}} \quad \text{ and } \quad \lim_{n \rightarrow \infty} \P\{\cif_n^{\ver} \le nx\} = \frac{\sqrt{1-x}}{\sqrt{x}+\sqrt{1-x}}. \label{EcifLimDis}
\end{align}
The next result bounds the speed of convergence above.  
\begin{thm}
\label{TCif}
Let $\delta > 0$. There exists a constant $C_0 = C_0(\delta) > 0$ such that 
\begin{align*}
\bigg|\,\P\{\cif_n^{\hor} \le nx\}-\frac{\sqrt{x}}{\sqrt{x}+\sqrt{1-x}}\,\bigg| &\le C_0 \bigg(\frac{\log n}{n}\bigg)^{1/3} \\
\bigg|\,\P\{\cif_n^{\ver} \le nx\}-\frac{\sqrt{1-x}}{\sqrt{x}+\sqrt{1-x}}\,\bigg| &\le C_0 \bigg(\frac{\log n}{n}\bigg)^{1/3}
\end{align*}
for $x \in [\delta, 1-\delta]$ and $n \in \bbZ_{>1}$.
\end{thm} 


\section{Proof of Theorem \ref{TLppMDBnd}}  \label{s:pf88} 

Recalling \eqref{ELbd}, define  
\begin{align}
\L^{\lambda}(x, y) = \inf_{\substack{w, z \in (0, 1) \\ w-z = \lambda}} \Lb^{w, z}(x, y) \quad \text{ for } x, y \in \bbR_{> 0} \text{ and } \lambda \in \bbR_{\ge 0}. 
\label{EL}
\end{align}
Note that the right-hand side is infinite when $\lambda \ge 1$ due to our convention that $\inf \emptyset = \infty$. 
This function provides an upper bound for the l.m.g.f.\ of the $\G$-process. 
\begin{lem}
\label{LLMUB}
Let $m, n \in \bbZ_{>0}$ and $\lambda \in \bbR_{\ge 0}$. Then 
$\log \E[e^{\lambda \G(m, n)}] \le \L^{\lambda}(m, n)$. 
\end{lem}
\begin{proof}
For $\lambda \ge 1$, the inequality holds trivially since then the right-hand side is infinite. For $\lambda \in [0, 1)$, the inequality follows from Proposition \ref{PLMId} since $\G(m, n) \le \Gb^{w, z}(m, n)$ for any $w, z \in (0, 1)$. 
\end{proof}

Introduce the function
\begin{align}
\Ib^{w, z}_s(x, y) = (w-z)s-\Lb^{w, z}(x, y) \quad \text{ for } w, z \in (0, 1), x, y \in \bbR_{>0} \text{ and } s \in \bbR.  \label{EIbd}
\end{align}
In the sequel, this function serves to describe the rate of deviations of the last-passage times. From definitions \eqref{EM} and \eqref{ELbd}, one has the identity 
\begin{align}
\Lb^{w, z}(x, y) = \int_{z}^w \M^t(x, y) \dd t \quad \text{ for } w, z \in (0, 1) \text{ with } w \ge z \text{ and } x, y \in \bbR_{>0}. \label{ELbdMId}
\end{align}
Therefore, one can rewrite \eqref{EIbd} as 
\begin{align}
\label{EIbdM}
\Ib^{w, z}_s(x, y) = \int_z^w (s- \M^t(x, y))\, \dd t \quad \text{ when } w \ge z.  
\end{align}
Thus, \eqref{EIbd} acquires a geometric interpretation as the \emph{signed area} between the horizontal line passing through $(0, s)$ and the curve $t \mapsto \M^t(x, y)$, $t \in (0, 1)$. This will be helpful in visualizing various formulas below, see Figure \eqref{FIbd}.  
\begin{figure}
\centering
\begin{overpic}[scale=0.6]{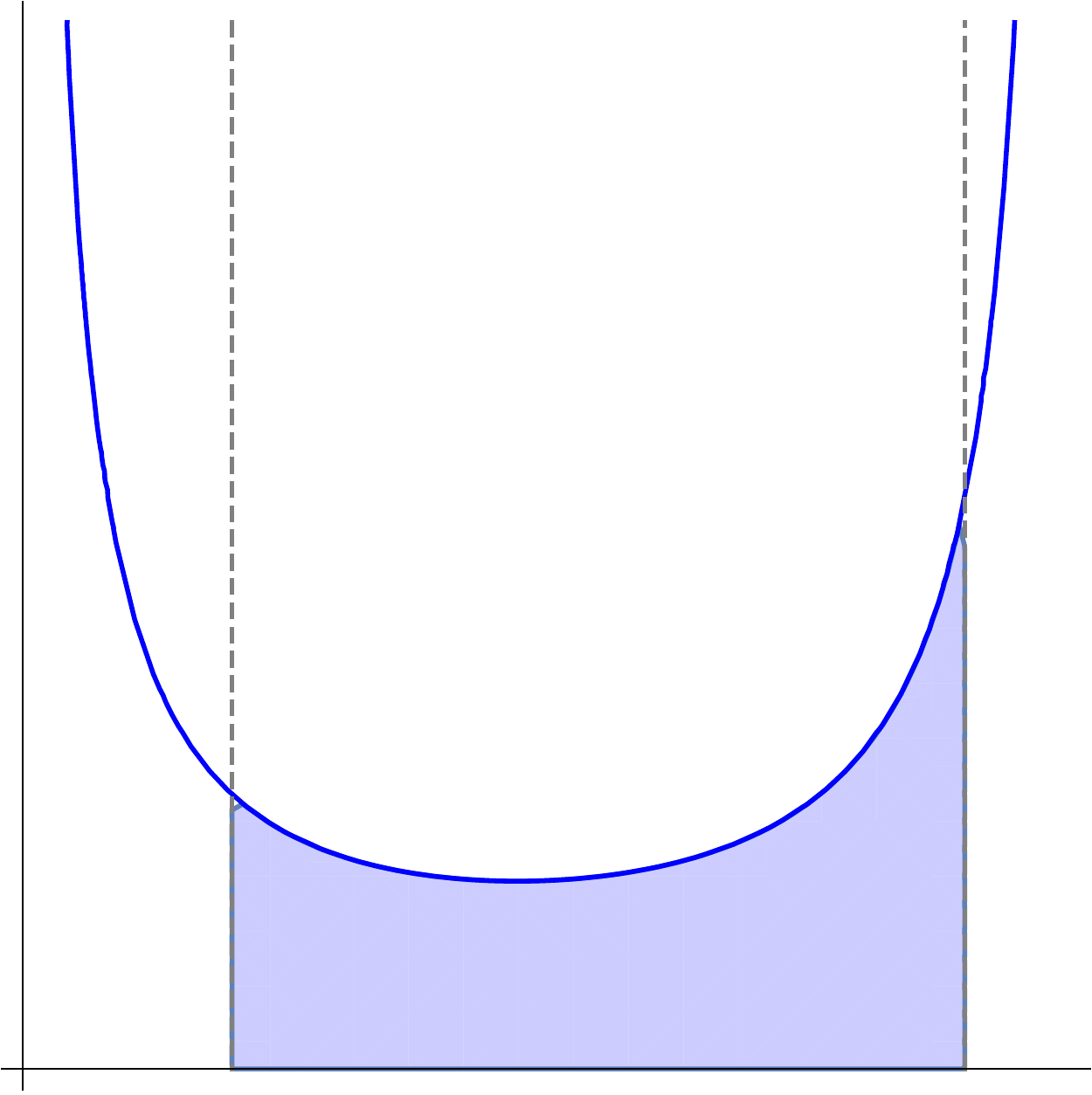}
\put (20, -2){$\displaystyle z$}
\put (86, -2){$\displaystyle w$}
\end{overpic}
\begin{overpic}[scale=0.6]{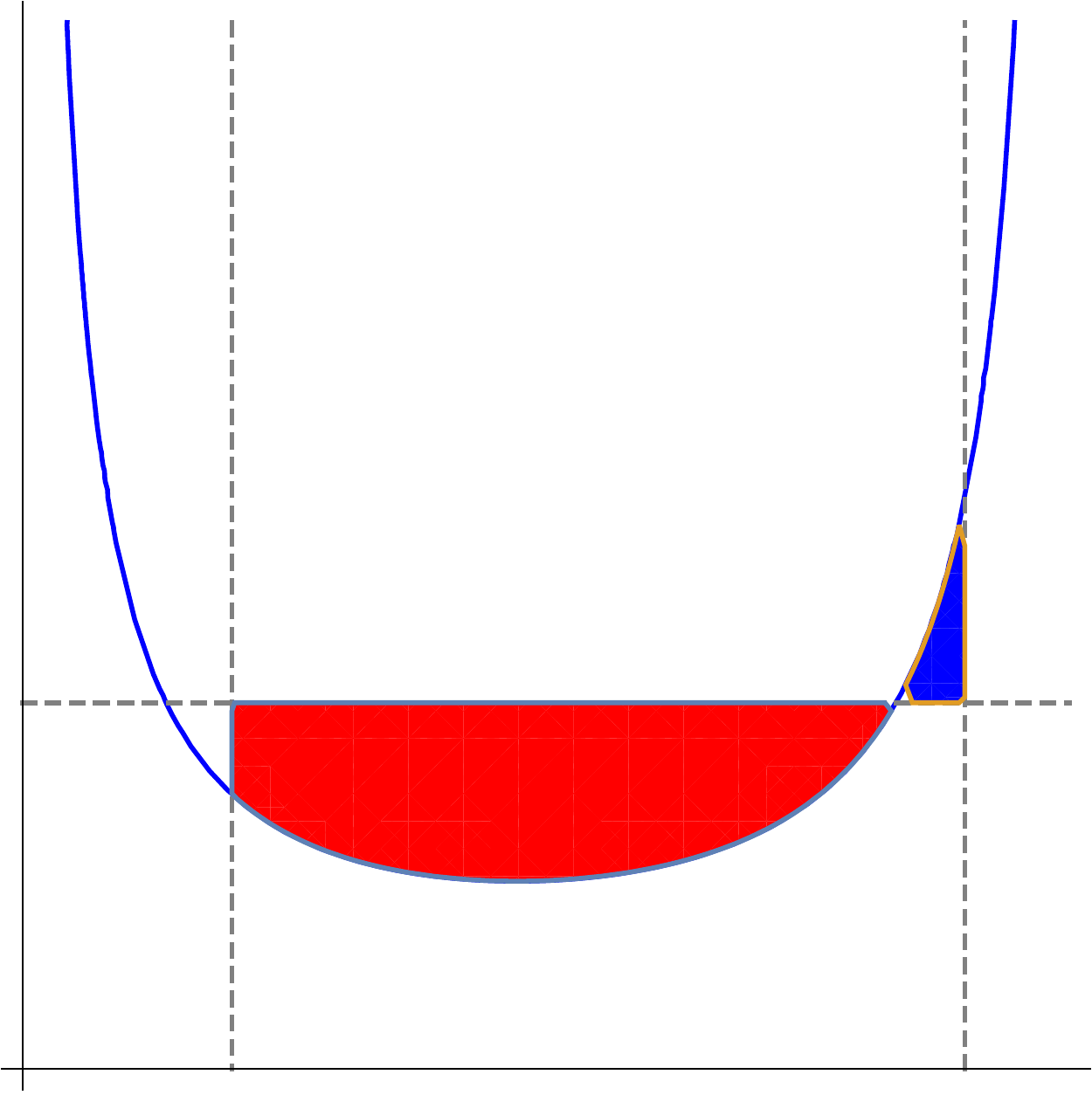}
\put (-3,35) {$\displaystyle s$}
\put (20, -2){$\displaystyle z$}
\put (86, -2){$\displaystyle w$}
\end{overpic}
\vspace{0.2in}
\caption{Illustrations of \eqref{ELbdMId} (left) and \eqref{EIbdM} (right) with $x = 4, y = 5, z = 0.2, w=0.9$ and $s = 35$. The curve $t \mapsto \M^t(x, y), t \in (0, 1)$ (blue) is shown. $\Lb^{w, z}(x, y)$ equals the area of the light blue region while $\Ib^{w, z}_s(x, y)$ equals the area of the red region minus the area of the blue region.}
\label{FIbd}
\end{figure}
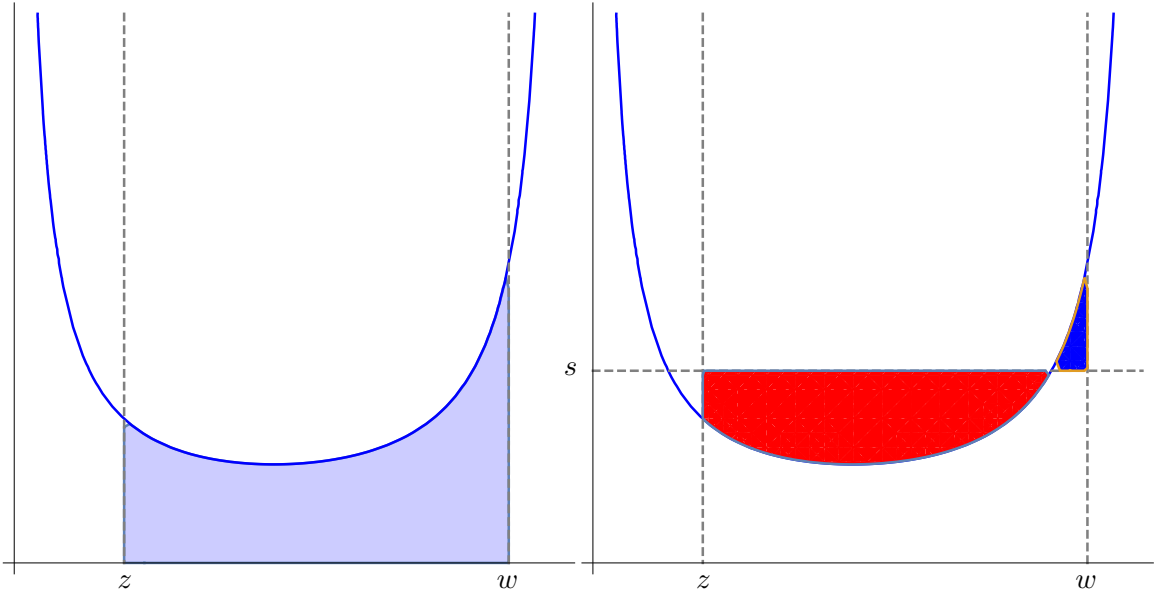

Now define 
\begin{align}
\I_s(x, y) = \sup_{\substack{w, z \in (0, 1) \\ w \ge z}}\Ib^{w, z}_s(x, y) \quad \text{ for } x, y \in \bbR_{>0} \text{ and } s \in \bbR. \label{EI}
\end{align}
In the next lemma, \eqref{EI} appears as an upper bound for the right tail deviations of the bulk LPP. This function is also the homogeneous case of the representation for the right tail large deviation rate function in \cite[Theorem 2.7]{emra-janj-17}. As mentioned in \cite[Example 2.6]{emra-janj-17}, one can compute \eqref{EI} explicity. For $s \ge \Shp(x, y)$, 
\begin{align}
\I_s(x, y) = \sqrt{(x+y-s)^2-4xy}-2x \cosh^{-1}\bigg(\frac{x-y+s}{2\sqrt{sx}}\bigg)-2y\cosh^{-1}\bigg(\frac{y-x+s}{2\sqrt{sy}}\bigg),\label{ESepp}
\end{align}
which recovers the formula from \cite[Theorem 4.4]{sepp98ebp}. 
Above, the inverse hyperbolic cosine   is given by $\cosh^{-1}(t) = \log (t + \sqrt{t^2-1}\hspace{1pt})$ for $t \in \bbR_{\ge 1}$. 

The functions in \eqref{EL} and \eqref{EI} are connected through convex duality: 
\begin{align}
\I_s(x, y) &= \sup_{\lambda \in [0, 1)} \sup_{\substack{w, z \in (0, 1) \\ w-z = \lambda}} \{(w-z)s - \Lb^{w, z}(x, y)\} = \sup_{\lambda \in [0, 1)} \{\lambda s - \L^{\lambda}(x, y)\} \nonumber\\
&= \sup_{\lambda \in \bbR_{\ge 0}}\{\lambda s - \L^{\lambda}(x, y)\}. \label{EConv}
\end{align}
\begin{lem}
\label{LRTUB}
Let $m, n \in \bbZ_{>0}$ and $s \in \bbR$. Then 
$\log \P\{\G(m, n) \ge s\} \le -\I_s(m, n)$. 
\end{lem}
\begin{proof}
By the exponential Markov's inequality and Lemma \ref{LLMUB}, 
\begin{align*}
\log \P\{\G(m, n) \ge s\} \le -\lambda s + \L^{\lambda}(m, n) \quad \text{ for } \lambda \in \bbR_{\ge 0}. 
\end{align*}
The claim then follows from \eqref{EConv}. 
\end{proof}

Our proof of Theorem \ref{TLppMDBnd} comes from a suitable lower bound on \eqref{EI}. With this and also applications in subsequent sections in mind, we turn to developing some estimates for \eqref{ELbd} and \eqref{EIbd} beginning with some trivial observations. Recall the definition of the cone $S_\delta$ from \eqref{ESc}. 

\begin{lem}
\label{LShpMinBnd} Let $\delta > 0$. 
\begin{enumerate}[\normalfont (a)]
\item 
$x+y \le \Shp(x, y) \le 2(x+y)$ for $x, y > 0$.
\item  
$
\Min(x, y) \in (\epsilon, 1-\epsilon) 
$
for $(x, y) \in S_\delta$ for some constant $\epsilon = \epsilon(\delta) > 0$. 
\item $c_0(x+y)^{1/3} \le \curv(x, y) \le C_0(x+y)^{1/3}$ for $(x, y) \in S_\delta$ for some constants $c_0 = c_0(\delta) > 0$ and $C_0 = C_0(\delta) > 0$. 
\end{enumerate}
\end{lem}
\begin{proof}
Immediate from definitions \eqref{EShp}, \eqref{EMin} and \eqref{Ecurv}. 
\end{proof}

\begin{lem}
\label{LMShpId}
Let $x, y \in \bbR_{>0}$ and $z \in (0, 1)$. Abbreviate $\Shp = \Shp(x, y)$, $\Min = \Min(x, y)$ and $\curv = \curv(x, y)$. The following statements hold. 
\begin{enumerate}[\normalfont (a)]
\item 
$
\M^z(x, y) = \Shp +  \dfrac{(z-\Min)^2\Shp}{z(1-z)} = \Shp + \curv^3 (z-\Min)^2 + \dfrac{\Shp(z-\Min)^3(\Min+z-1)}{z(1-z)\Min(1-\Min)}. 
$\\
\item   For $\delta > 0$ and $\epsilon > 0$, there exists a constant $C_0 = C_0(\delta, \epsilon)$ such that  
\begin{align*}
|\M^z(x, y)-\Shp-\curv^3(z-\Min)^2| \le C_0 (x+y) |z-\Min|^3 \quad \text{ whenever } (x, y) \in S_\delta \text{ and } z \in (\epsilon, 1-\epsilon). 
\end{align*}
\end{enumerate}
\end{lem}
\begin{proof}
From definitions \eqref{EM}, \eqref{EShp} and \eqref{EMin}, and the identity $\dfrac{x}{\Min^2} = \dfrac{y}{(1-\Min)^2} = \Shp$, one has 
\begin{align*}
\M^{z}(x, y)-\Shp &= (z-\Min)\bigg(-\frac{x}{z\Min}+ \frac{y}{(1-z)(1-\Min)}\bigg) \\
&= (z-\Min)\bigg(-\frac{x}{z\Min}+ \frac{x}{\Min^2} + \frac{y}{(1-z)(1-\Min)}-\frac{y}{(1-\Min)^2}\bigg) \\
&= (z-\Min)^2\bigg(\frac{x}{z\Min^2}+\frac{y}{(1-z)(1-\Min)^2}\bigg) = \frac{(z-\Min)^2\Shp}{z(1-z)}. 
\end{align*}
Recalling  the definition of $\curv$ from \eqref{Ecurv}, one obtains that 
\begin{align*}
\M^{z}(x, y)-\Shp - \curv^3 (z-\Min)^2  = (z-\Min)^2\Shp \bigg\{\frac{1}{z(1-z)}-\frac{1}{\Min(1-\Min)}\bigg\} = (z-\Min)^3\Shp\bigg\{\frac{\Min+z-1}{z(1-z)\Min(1-\Min)}\bigg\}. 
\end{align*}
This verifies (a). Part (b) follows from bounding the last expression above using Lemma \ref{LShpMinBnd}(a)--(b) and the assumption $z \in (\epsilon, 1-\epsilon)$. 
\end{proof}

\begin{lem}
\label{LLIBdEst}
Let $x, y \in \bbR_{>0}$, $s \in \bbR$ and $w, z \in (0, 1)$ with $w \ge z$. Abbreviate $\Shp = \Shp(x, y)$, $\Min = \Min(x, y)$ and $\curv = \curv(x, y)$. The following statements hold. 
\begin{enumerate}[\normalfont (a)]
\item $
\Lb^{w, z}(x, y) = (w-z) \Shp+ \dfrac{\curv^3}{3} \{(w-\Min)^3-(z-\Min)^3\}+ \Shp \displaystyle \int_{z}^w \dfrac{(t-\Min)^3(\Min+t-1)}{t(1-t)\Min(1-\Min)} \dd t. 
$\\
\item Fix $\delta > 0$ and $\epsilon > 0$. There exists a constant $C_0 = C_0(\delta, \epsilon) > 0$ such that  
\begin{align*}
\bigg|\Lb^{w, z}(x, y)-(w-z) \Shp- \dfrac{\curv^3}{3} \{(w-\Min)^3-(z-\Min)^3\}\bigg| &\le C_0 (x+y)\{(w-\Min)^4 + (z-\Min)^4\} \\
\bigg|\Ib^{w, z}_s(x, y)-(w-z) (s-\Shp) + \dfrac{\curv^3}{3} \{(w-\Min)^3-(z-\Min)^3\}\bigg| &\le C_0 (x+y)\{(w-\Min)^4 + (z-\Min)^4\} 
\end{align*}
whenever $(x, y) \in S_\delta$ and $w, z \in (\epsilon, 1-\epsilon)$. 
\end{enumerate}
\end{lem}
\begin{proof}
This follows from the identities \eqref{ELbdMId}--\eqref{EIbdM} and Lemma \ref{LMShpId}. 
\end{proof}


\begin{proof}[Proof of Theorem \ref{TLppMDBnd}]
Let $(m, n) \in S_\delta$, and abbreviate $\Shp = \Shp(m, n)$, $\Min = \Min(m, n)$ and $\curv = \curv(m, n)$. By Lemma \ref{LShpMinBnd}(b), there exist $\lambda_0 = \lambda_0(\delta) > 0$ and $\epsilon_0 = \epsilon_0(\delta) > 0$ such that 
\begin{align}
\Min-\lambda/2 \ge \epsilon_0 \text{ and } \Min + \lambda/2 \le 1-\epsilon_0 \quad \text{ for } \lambda \in [0, \lambda_0].  \label{E54}
\end{align}
Part (c) of the same lemma then implies that 
\begin{align}
\frac{2\sqrt{s}}{\curv} \le \lambda_0 \quad \text{ for } s \in [0, c_0(m+n)^{2/3}] \label{E55}
\end{align}
for some small enough constant $c_0 = c_0(\delta) > 0$. By virtue of Lemmas \ref{LShpMinBnd}(c) and \ref{LLIBdEst}(b), and \eqref{E54}--\eqref{E55}, 
\begin{align}
\Ib_{\Shp + \curv s}^{\Min + \sqrt{s}/\curv, \Min-\sqrt{s}/\curv}(m, n) \ge \frac{4s^{3/2}}{3}-\frac{C_0s^2}{(m+n)^{1/3}} \quad 
\text{ for } s \in [0, c_0(m+n)^{2/3}] \label{E56}
\end{align}
for some constant $C_0 = C_0(\delta) > 0$. The result follows from \eqref{E56}, Lemma \ref{LRTUB} and definition \eqref{EI}. 
%
%
\end{proof}

\section{Proofs of Theorem \ref{TisLppMDRate} and Corollary \ref{CisLppMDRate}} 
\label{SisLppMDRate}

\subsection{Upper bound for the l.m.g.f.}

To derive the upper bound in Theorem \ref{TisLppMDRate}, we proceed analogously to Section \ref{s:pf88}. Introduce the functions  
\begin{align}
\L^{\lambda, w, \hor}(x, y) &= \inf \limits_{\substack{u, v \,\in\, (0, w) \\ u-v = \lambda}} \Lb^{u, v}(x, y), \qquad \L^{\lambda, z, \ver}(x, y) = \inf \limits_{\substack{u, v \,\in\, (z, 1) \\ u-v = \lambda}} \Lb^{u, v}(x, y) \label{ELhv} \\
\L^{\lambda, w, z}(x, y) &= \max \{\L^{\lambda, w, \hor}(x, y), \L^{\lambda, z, \ver}(x, y)\}. \label{ELbd2}
\end{align}
for $x, y \in \bbR_{>0}$, $\lambda \in \bbR_{\ge 0}$, $w \in (0, 1]$ and $z \in [0, 1)$, 
These definitions specialize to \eqref{EL} when $w = 1$ and $z = 0$. Lemma \ref{LLMBdUB} below shows that the functions in \eqref{ELhv}--\eqref{ELbd2} furnish upper bounds for the l.m.g.f.s of $\Gb^{w}_{1, 0}$, $\Gb^{z}_{0, 1}$ and $\Gb^{w, z}$- processes, respectively. 

For ease of reference, the next lemma collects a few simple properties of the function \eqref{EM}. See Figure \ref{FMplot} for a plot. 
\begin{lem}
\label{LMProp}
Fix $x, y \in \bbR_{>0}$. The function $t \mapsto \M^t(x, y)$ is 
\begin{enumerate}[\normalfont (\romannumeral1)]
\item continuous and strictly convex on $(0, 1)$, 
\item decreasing on $(0, \Min(x, y)]$ with image $[\Shp(x, y), \infty)$, 
\item increasing on $[\Min(x, y), 1)$ with image $[\Shp(x, y), \infty)$. 
\end{enumerate}
\end{lem}
\begin{proof}
This is straightforward to verify from definitions \eqref{EM}, \eqref{EShp} and \eqref{EMin}. 
\end{proof}

Lemma \ref{LMProp} implies that, for each $\lambda \in (0, 1)$, there exists a unique pair $w = \Min_+^\lambda(x, y)$ and $z = \Min_-^\lambda(x, y)$ in $(0, 1)$ such that 
\begin{align}
w-z = \lambda \quad \text{ and } \quad \M^w(x, y) = \M^{z}(x, y). \label{EMinpm}
\end{align}
See Figure \ref{FMinpm}. In fact, it is possible to compute this pair explicitly. 
\begin{align}
\Min_{\pm}^\lambda(x, y) &= \Min(x, y)  \pm \frac{\lambda}{2} + \frac{\lambda^2 (y-x)}{4\sqrt{xy}+2\sqrt{\lambda^2 (x-y)^2 + 4xy}}. \label{EMinpm2}
\end{align}
Extend \eqref{EMinpm} continuously by setting 
\begin{align}
\Min_\pm^0(x, y) = \Min(x, y), \quad \text{ and } \quad \Min_{+}^\lambda(x, y) = 1, \ \Min_-^{\lambda}(x, y) = 0  \text{ when } \lambda \ge 1. \label{EMinpm3}
\end{align}
The next lemma describes the minimizers of \eqref{ELhv} in terms of \eqref{EMinpm}. In particular, taking $w = 1$ in part (a) (or $z = 0$ in part (b)) shows that the infimum in \eqref{EL} is attained uniquely at $w = \Min_+^\lambda(x, y)$ and $z= \Min_-^\lambda(x, y)$. 

\begin{lem}
\label{LMinpm}
Let $x, y \in \bbR_{>0}$, $\lambda \in \bbR_{\ge 0}$, $w \in (0, 1]$ and $z \in [0, 1)$. Abbreviate $\Min_\pm^\lambda = \Min_\pm^\lambda(x, y)$. Let $u, v \in (0, 1)$ with $u-v =\lambda$. 
\begin{enumerate}[\normalfont (a)]
\item If $u \le w$ then $\Lb^{u, v}(x, y) \ge \L^{\lambda, w, \hor}(x, y)$ with equality if and only if $u = \min \{w, \Min_+^\lambda\}$ or $\lambda = 0$. 
\item If $v \ge z$ then $\Lb^{u, v}(x, y) \ge \L^{\lambda, z, \ver}(x, y)$ with equality if and only if $v = \max \{z, \Min_-^\lambda\}$ or $\lambda = 0$. 
\end{enumerate}
\end{lem}
\begin{proof}
By assumption, $\lambda = u-v \in [0, 1)$. Definitions \eqref{ELbd} and \eqref{EL} give $\Lb^{u, v}(x, y) = 0 = \L^0(x, y)$ when $\lambda = 0$. Consider the case $\lambda \in (0, 1)$ now. One computes from \eqref{ELbd} that 
\begin{align}
\partial_t \{\Lb^{t, t-\lambda}(x, y)\} = \partial_t \bigg\{\int_{t-\lambda}^{t} \M^s(x, y) \dd s\bigg\} = \M^{t}(x, y)-\M^{t-\lambda}(x, y) \quad \text{ for } t \in (\lambda, 1). \label{E29}
\end{align}
The right-hand side is an increasing, continuous function with range $\bbR$ due to Lemma \ref{LMProp}, and equals zero at $t = \Min_+^\lambda \in (\lambda, 1)$ by \eqref{EMinpm}. This together with definition \eqref{ELhv} implies (a). The proof of (b) is similar. 
\end{proof}

\begin{figure}
\centering
\begin{overpic}[scale=0.6]{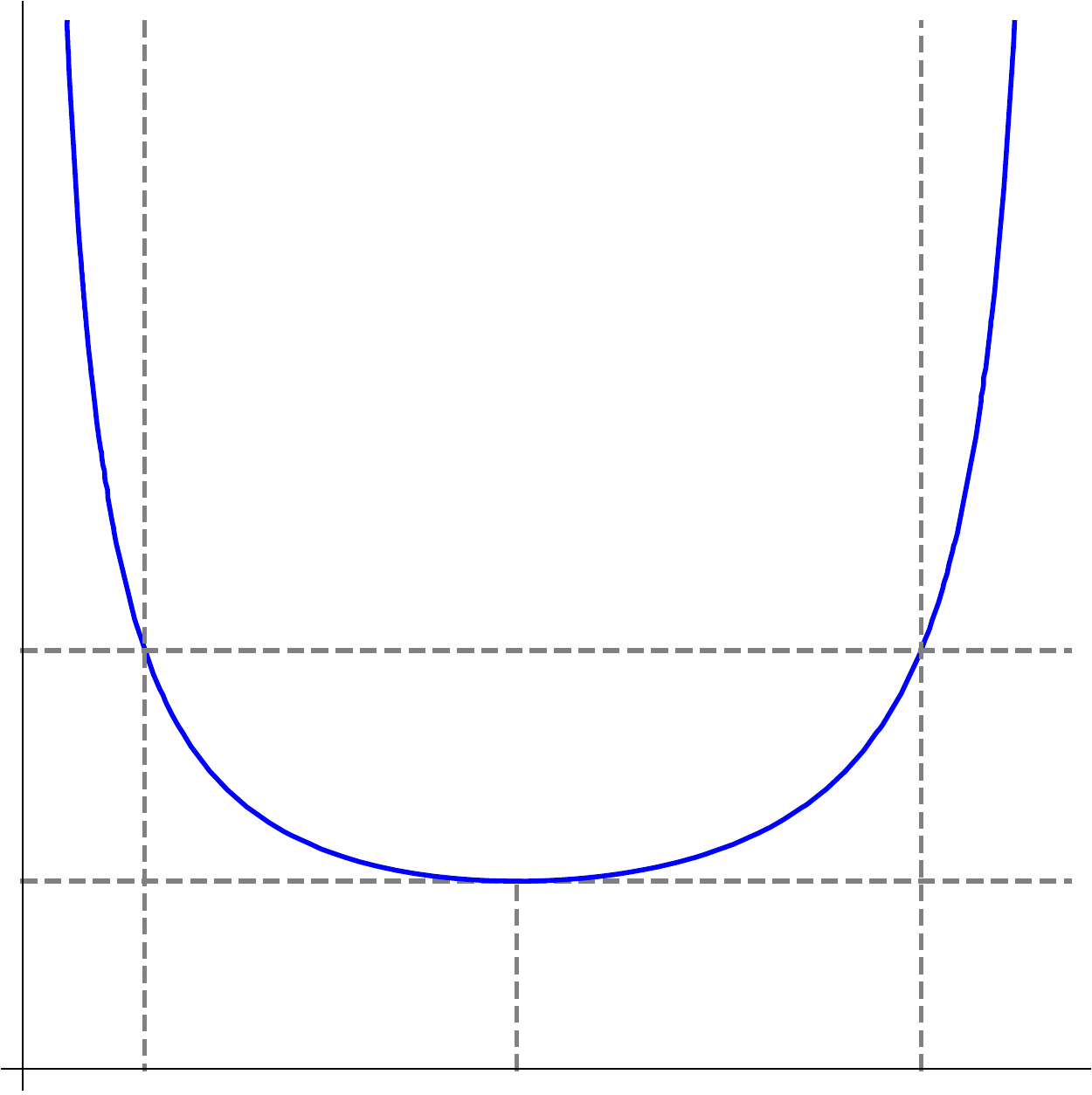}
\put (-13, 18){$\displaystyle \Shp(x, y)$}
\put (41, -3){$\displaystyle \Min(x, y)$}
 \put (-25,38) {$\displaystyle \M^{\Min_-^{\lambda}} = \M^{\Min_+^\lambda}$}
\put (10,-3) {$\displaystyle \Min_-^{\lambda}$}
\put (80,-3) {$\displaystyle \Min_+^{\lambda}$}
\put (13, 40){{\color{gray}$\displaystyle \overbrace{\hspace{2.15in}}$}}
\put (48, 45){$\displaystyle \lambda$}
\end{overpic}
\vspace{0.2in}
\caption{An illustration of $\Min_{\pm}^\lambda(x, y)$ defined at \eqref{EMinpm} with $x = 4$, $y = 5$ and $\lambda \approx 0.74$. The curve $z \mapsto \M^z(x, y), z \in (0, 1)$ is also shown (blue).}
\label{FMinpm}
\end{figure}
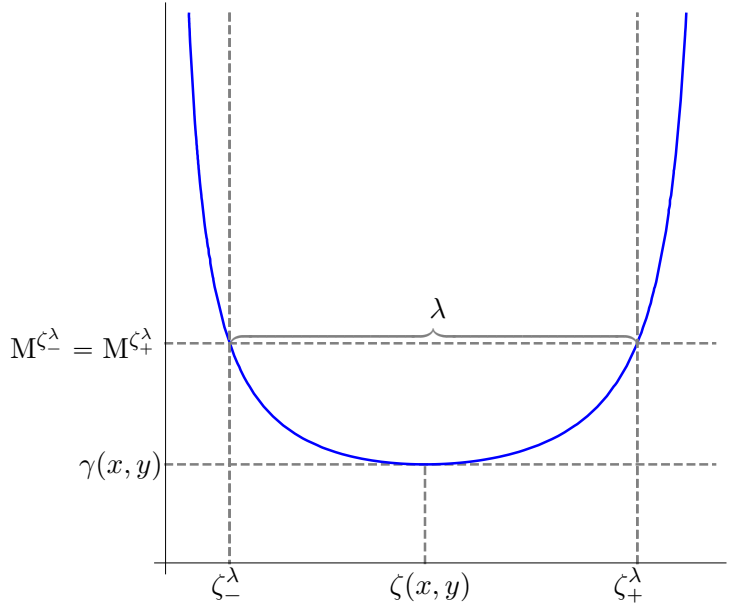

\begin{lem}
\label{LLMBdUB}
Let $m, n \in \bbZ_{>0}$, $\lambda \in \bbR_{\ge 0}$, $w \in (0, 1]$ and $z \in [0, 1)$. Then 
\begin{enumerate}[\normalfont (a)]
\item $\log \E[e^{\lambda \Gb_{1, 0}^w(m, n)}] \le \L^{\lambda, w, \hor}(m, n)$
\item $\log \E[e^{\lambda \Gb_{0, 1}^z(m, n)}] \le \L^{\lambda, z, \ver}(m, n)$
\item $\log \E[e^{\lambda \Gb^{w, z}(m, n)}] \le \L^{\lambda, w, z}(m, n) + \log 2$. 
\end{enumerate}
\end{lem}
\begin{proof}
To obtain (a), assume $\lambda \in (0, w)$ since the claim is trivial otherwise because the right-hand side is infinite when $\lambda \ge w$ and both sides are zero when $\lambda = 0$. By monotonicity and Proposition \ref{PLMId}, 
\begin{align}
\log \E[e^{\lambda \Gb_{1, 0}^w(m, n)}] \le \log \E[e^{\lambda \Gb^{w, w-\lambda}(m, n)}] = \Lb^{w, w-\lambda}(m, n). \label{E22}
\end{align}
In the case $w \ge \Min_+^\lambda = \Min_+^\lambda(m, n)$, monotonicity and \eqref{E22} also give 
\begin{align}
\log \E[e^{\lambda \Gb_{1, 0}^w(m, n)}] \le \log \E[e^{\lambda \Gb_{1, 0}^{\Min_+^\lambda}(m, n)}] \le \Lb^{\Min_+^\lambda, \Min_-^\lambda}(m, n). \label{E38}
\end{align}
Bounds in \eqref{E22}--\eqref{E38} can be combined as 
\begin{align*}
\Lb^{w, w-\lambda}(m, n)\one_{\{w < \Min_+^\lambda\}} + \Lb^{\Min_+^\lambda, \Min_-^\lambda}(m, n)\one_{\{w \ge \Min_+^\lambda\}} = \Lb^{\min \{w, \Min_+^\lambda\}, \min \{w, \Min_+^\lambda\}-\lambda} = \L^{\lambda, w, \hor}(m, n). 
\end{align*}
The last equality appeals to Lemma \ref{LMinpm}(a). The proof of (b) is analogous.  For (c), note from definition \eqref{Elppbd} that 
\begin{align}
\label{EGbmax} \Gb^{w, z}(m, n) = \max \{\Gb_{1, 0}^{w}(m, n), \Gb_{0, 1}^{z}(m, n)\}. 
\end{align}
Therefore, 
\begin{align}
\log \E[e^{\lambda \Gb^{w, z}(m, n)}] &\le \max \{\log \E[e^{\lambda \Gb^w_{1, 0}(m, n)}], \log \E[e^{\lambda \Gb^z_{0, 1}(m, n)}]\} + \log 2.  \label{ElogmaxUB}
\end{align}
Combining this with parts (a)-(b) and definition \eqref{ELbd2}, one obtains that 
\begin{equation*}
\log \E[e^{\lambda \Gb^{w, z}(m, n)}]  \le \max \{\L^{\lambda, w, \hor}(m, n), \L^{\lambda, z, \ver}(m, n)\} + \log 2 = \L^{\lambda, w, z}(m, n) + \log 2. \qedhere
\end{equation*}
\end{proof}

\subsection{Upper bound for the right tail}

For $x, y \in \bbR_{>0}$, $w \in (0, 1]$, $z \in [0, 1)$ and $s \in \bbR$, define 
\begin{align}
\I^{w, \hor}_s(x, y) &= \sup_{\substack{u, v \,\in\, (0, w) \\ u \ge v}} \Ib^{u, v}_s(x, y), \qquad \I^{z, \ver}_s(x, y) = \sup_{\substack{u, v \,\in\, (z, 1) \\ u \ge v}} \Ib^{u, v}_s(x, y)  \label{EIhv} \\[3pt]
\text{and } \quad 
\I^{w, z}_s(x, y) &= \min \{\I^{w, \hor}_s(x, y), \I^{z, \ver}_s(x, y)\}. \label{EIbdRate}
\end{align}
At the endpoints $w = 1$ and $z = 0$, these functions recover the bulk rate function of \eqref{EI}:  
\begin{align}
\I^{1, 0}_s(x, y) = \I^{0, \ver}_s(x, y) = \I^{1, \hor}_s(x, y) = \I_s(x, y). \label{EIhv10}
\end{align}
Lemma \ref{LRTBdUB} below bounds from above the right tail deviations of the LPP processes \eqref{Elppbd} with boundary weights in terms of the functions in \eqref{EIhv}--\eqref{EIbdRate}. 

As another consequence of Lemma \ref{LMProp}, for each $s > \Shp(x, y)$, there exists a unique pair $w = \Minx_+^s(x, y)$ and $z = \Minx_-^s(x, y)$ in $(0, 1)$ characterized by 
\begin{align}
z < \Min(x, y) < w  \quad \text{ and } \quad \M^w(x, y) = s = \M^z(x, y). \label{EMinxpm}
\end{align}
A straightforward, if tedious, computation gives 
\begin{align}
\Minx_{\pm}^s 
&= \frac{1}{2} + \frac{x-y\pm \sqrt{(s-\Shp(x, y))^2 + 4\sqrt{xy}(s-\Shp(x, y))}}{2s}.\label{EMinxpm2}
\end{align}
Extend \eqref{EMinxpm} via 
\begin{align}
\Minx_{\pm}^s(x, y) = \Min(x, y) \quad \text{ when } s \le \Shp(x, y). \label{EMinxpm3}
\end{align}
It can be verified that \eqref{EMinxpm3} is a continuous extension. By the uniqueness of the pairs in \eqref{EMinpm} and \eqref{EMinxpm}, for $s > \Shp(x, y)$ and $\lambda \in (0, 1)$, these three equalities are equivalent with each other: 
\begin{align}
\Minx_+^s(x, y) = \Min_+^\lambda(x, y), \quad \lambda = \Minx_+^s(x, y)-\Minx_-^s(x, y), \quad \Minx_-^s(x, y) = \Min_-^\lambda(x, y). \label{EExtId}
\end{align}
 See Figure \ref{FMinMax}. 

\begin{figure}
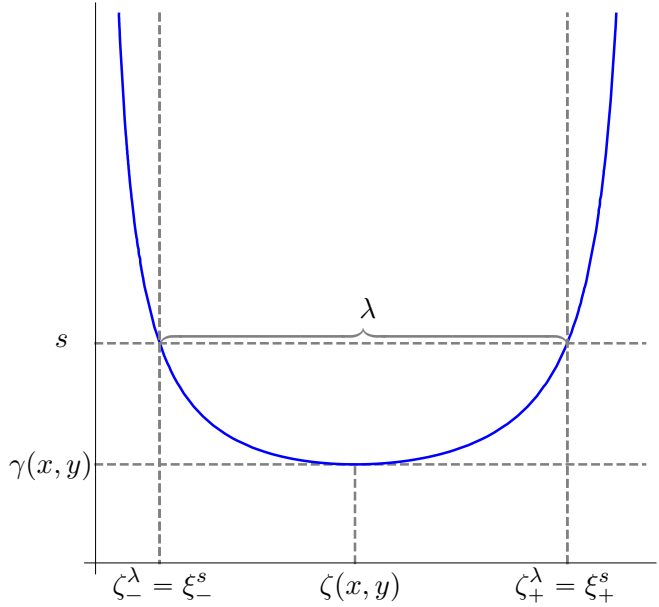

\centering
\begin{overpic}[scale=0.6]{Minpm4.pdf}
\put (-13, 18){$\displaystyle \Shp(x, y)$}
\put (41, -3){$\displaystyle \Min(x, y)$}
 \put (-5,40) {$\displaystyle s$}
\put (5,-3) {$\displaystyle \Min_-^{\lambda}= \Minx_-^{s}$}
\put (75,-3) {$\displaystyle \Min_+^{\lambda} = \Minx_+^{s}$}
\put (13, 40){{\color{gray}$\displaystyle \overbrace{\hspace{2.15in}}$}}
\put (48, 45){$\displaystyle \lambda$}
\end{overpic}
\vspace{0.2in}
\caption{An illustration of the identities $\Min^\lambda_{\pm}(x, y) = \Minx_\pm^s(x, y)$ and the curve $z \mapsto \M^z(x, y)$ (blue) with $x = 4$, $y = 5$, $s = 40$ and $\lambda \approx 0.74$. }
\label{FMinMax}
\end{figure}

The next lemma expresses the maximizers of the suprema in \eqref{EIhv} in terms of \eqref{EMinxpm}. In preparation for its statement, introduce 
\begin{align}
\Shp^{w, \hor}(x, y) &= \inf_{t \in (0, w)} \M^t(x, y) = \Shp(x, y)\one_{\{w \ge \Min(x, y)\}} + \M^{w}(x, y)\one_{\{w < \Min(x, y)\}}\label{EShph} \\
\Shp^{z, \ver}(x, y) &= \inf_{t \in (z, 1)} \M^t(x, y) = \Shp(x, y)\one_{\{z \le \Min(x, y)\}} + \M^{z}(x, y)\one_{\{z > \Min(x, y)\}}\label{EShpv}
\end{align}
for $x, y \in \bbR_{>0}$, $w \in (0, 1]$ and $z \in [0, 1)$. The second equalities above follow from Lemma \ref{LMProp}. For fixed $w \in (0, 1]$ and $z \in [0, 1)$, the functions in \eqref{EShph}--\eqref{EShpv} are the shape functions of the $\Gb_{1, 0}^{w}$ and $\Gb_{0, 1}^{z}$ -processes, respectively, in a sense analogous to the limit  \eqref{EShpLim}. This fact is recorded, for example, in \cite[Lemma 6.4]{geor-rass-sepp-17-buse}. 
The functions in \eqref{EIhv}--\eqref{EIbdRate} are clearly nonnegative. The regions of positivity for these functons can also be readily identified from the definitions and Lemma \ref{LMProp} as 
\begin{align}
\I^{w, \hor}_{s}(x, y) > 0 \quad &\text{ if and only if } \quad s > \Shp^{w, \hor}(x, y) \label{EIhpos}\\
\I^{z, \ver}_s(x, y) > 0 \quad &\text{ if and only if } \quad s > \Shp^{z, \ver}(x, y) \label{EIvpos}. 
\end{align}
The equivalence \eqref{EIhpos}, for example, can also be seen from Figure \ref{Frate} below depicting the function $\I_s^{w, \hor}(x, y)$ evaluated with specific choices of the parameters. 

\begin{figure}
\centering
\begin{overpic}[scale=0.6]{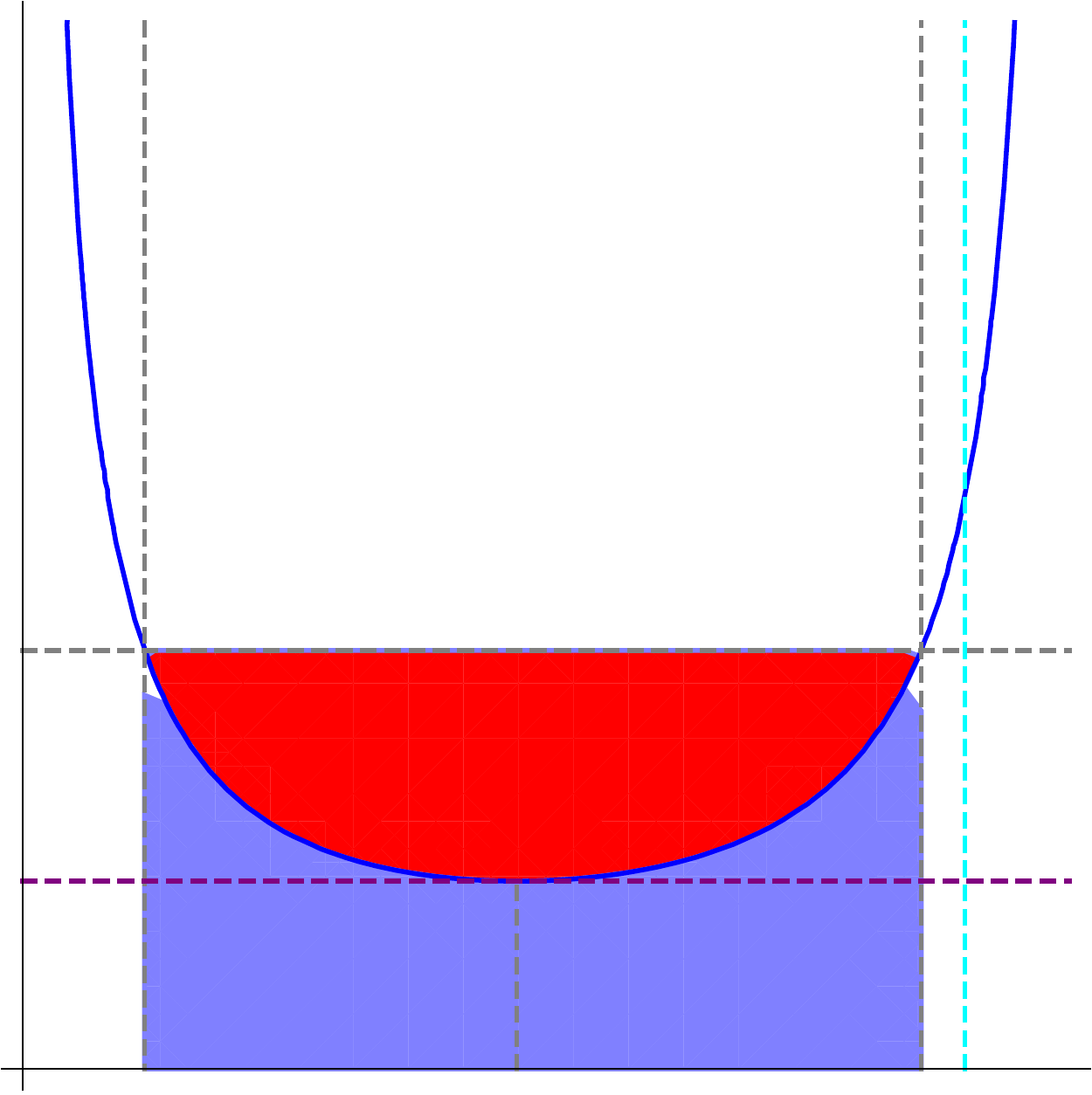}
\put (-3,39) {$\displaystyle s$}
\put (87, -2){$\displaystyle w$}
\put (46, -2){$\displaystyle \Min$}
\put (10,-3) {$\displaystyle \Minx_-^{s} = \Min_-^\lambda$}
\put (68,-3) {$\displaystyle \Minx_+^{s} = \Min_+^\lambda$}
\put (13, 40){{\color{gray}$\displaystyle \overbrace{\hspace{2.15in}}$}}
\put (48, 45){$\displaystyle \lambda$}
\end{overpic}
\begin{overpic}[scale=0.6]{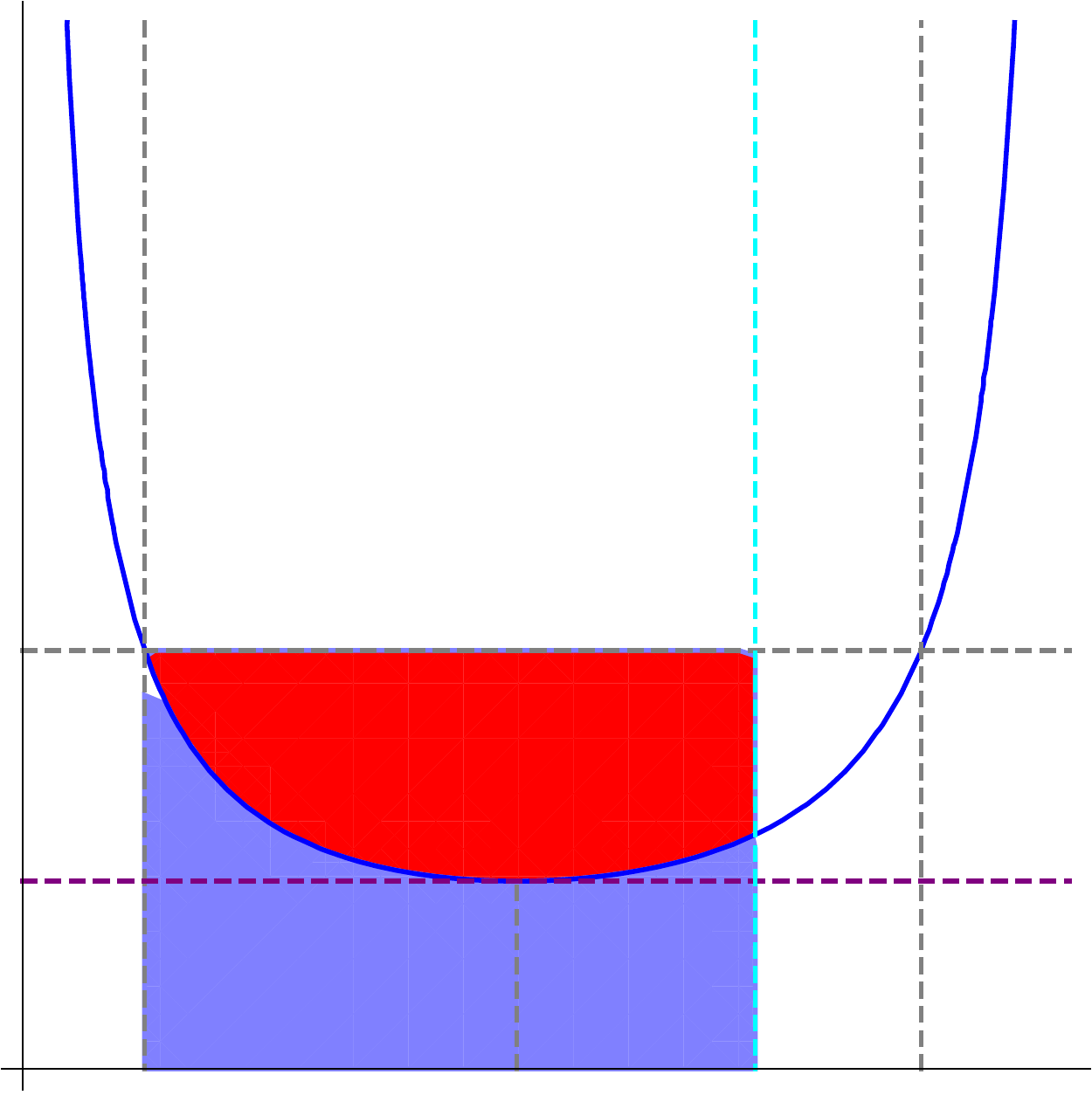}
\put (-125, 18){$\displaystyle \Shp^{w, \hor} = \Shp$}
\put (67, -2){$\displaystyle w$}
\put (46, -2){$\displaystyle \Min$}
\put (12,-3) {$\displaystyle \Minx_-^{s}$}
\put (82,-3) {$\displaystyle \Minx_+^{s}$}
\put (13, 40){{\color{gray}$\displaystyle \overbrace{\hspace{1.70in}}$}}
\put (40, 45){$\displaystyle \lambda$}
\end{overpic}
\par  
\vspace{0.2in}
\begin{overpic}[scale=0.6]{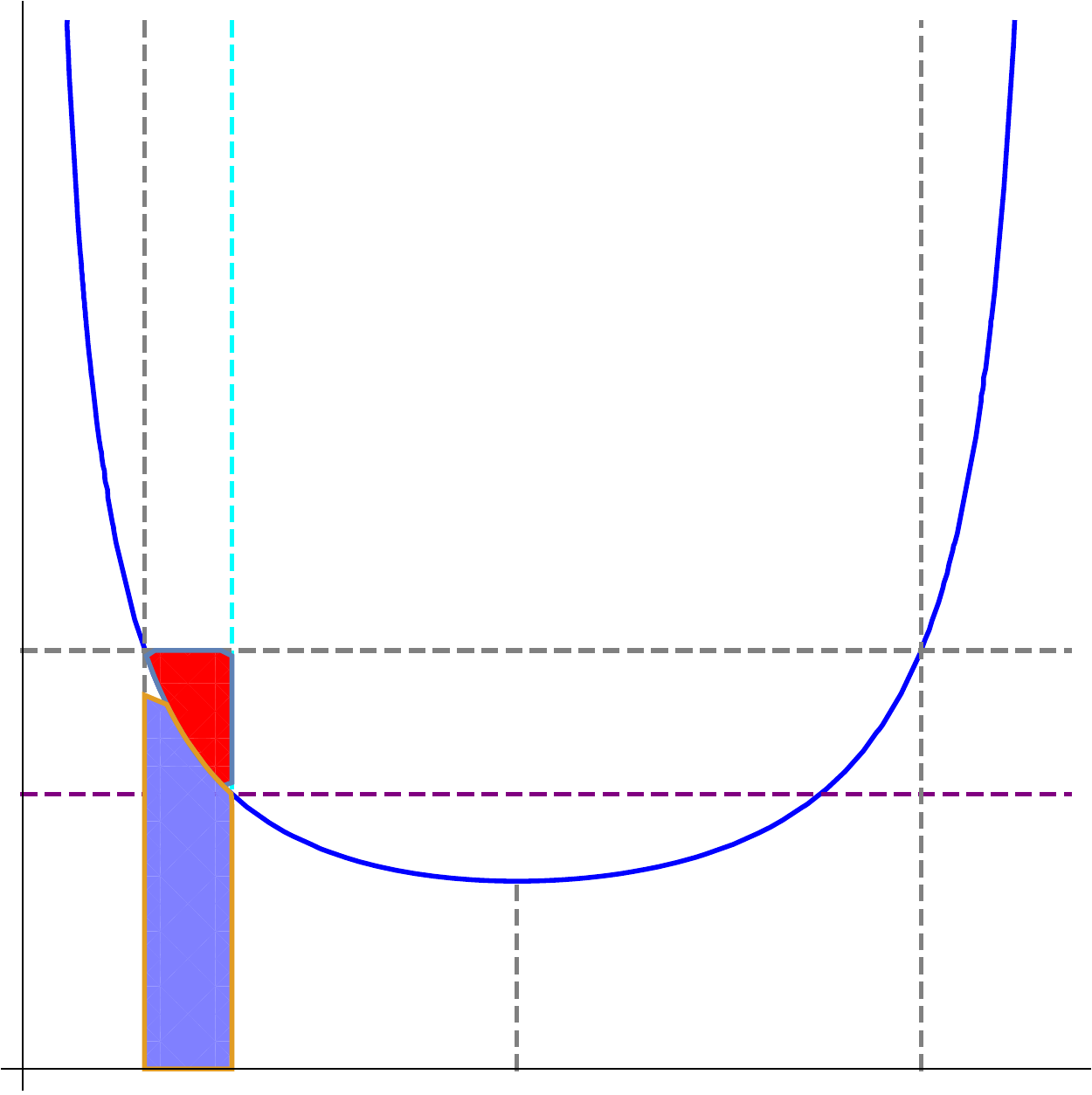}
\put (-3,39) {$\displaystyle s$}
\put (-27, 26){$\displaystyle \Shp^{w, \hor} = \M^{w}$}
\put (20, -2){$\displaystyle w$}
\put (46, -2){$\displaystyle \Min$}
\put (12,-3) {$\displaystyle \Minx_-^{s}$}
\put (82,-3) {$\displaystyle \Minx_+^{s}$}
\put (15, 42){$\displaystyle \lambda$}
\end{overpic}
\vspace{0.2in}
\caption{An illustration of the values of the functions $\I_s^{w, \hor}(x, y)$ (the areas of the red regions) and $\L^{\lambda, w, \hor}(x, y)$ (the areas of the light blue regions) for $x = 4, y = 5, s = 40$, $\lambda = \min \{w, \Minx_+^s(x, y)\}$ and three values $\{0.9, 0.7, 0.2\}$ (top left, top right, bottom pictures, respectively) of $w$. The curve $t \mapsto \M^t(x, y), t \in (0, 1)$ (blue), the horizontal levels at $\Shp^{w, \hor}(x, y)$ (dashed purple) and $s$ (dashed gray), and the vertical levels at $w$ (dashed cyan) and $\Minx_{\pm}^s(x, y)$ (dashed gray) are indicated.}
\label{Frate}
\end{figure}
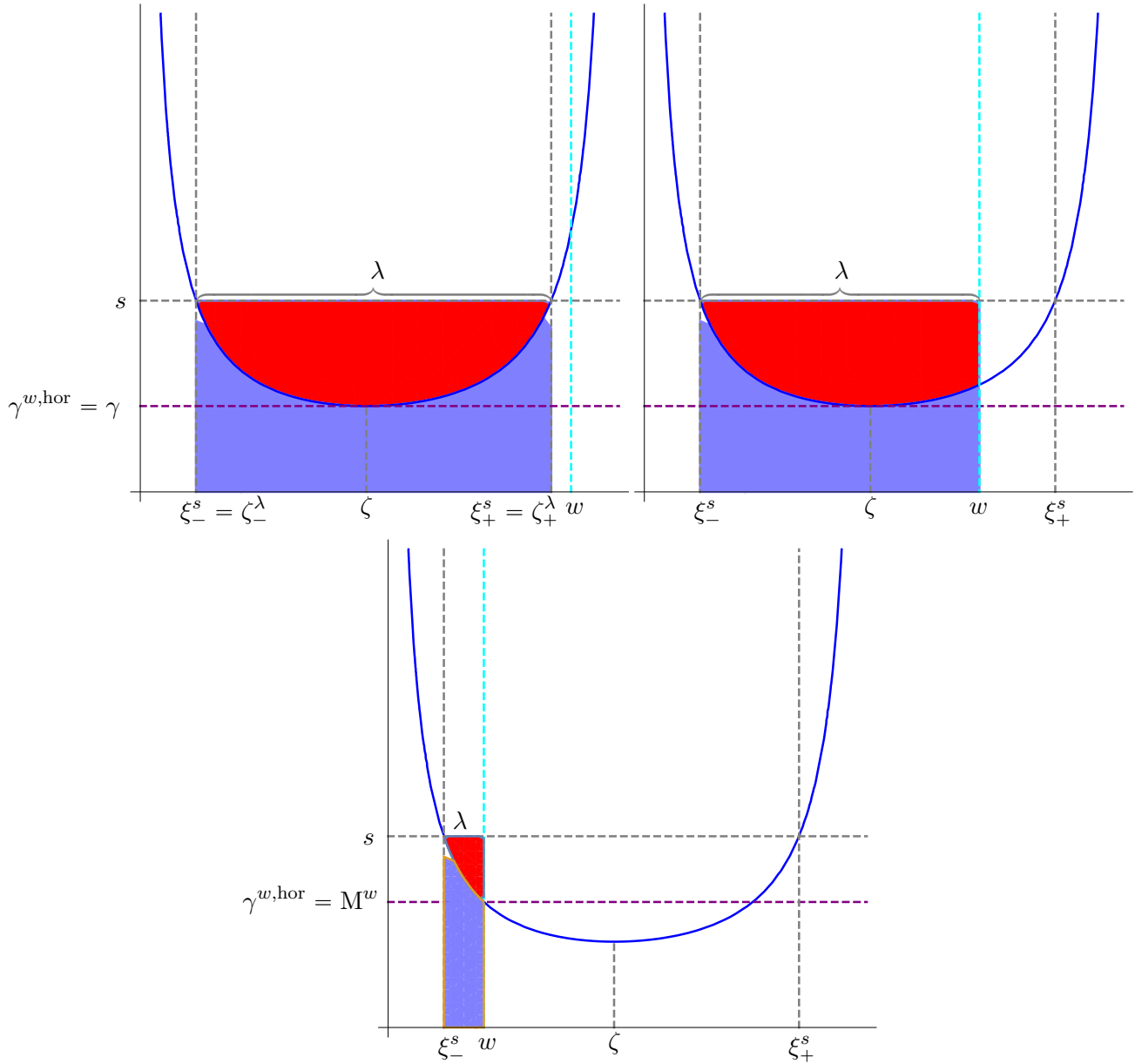

\begin{lem}
\label{LMinxpm}
Let $x, y \in \bbR_{>0}$, $w \in (0, 1], z \in [0, 1)$ and $s \in \bbR$. Abbreviate $\Minx_\pm^s = \Minx_\pm^s(x, y)$. Let $u, v \in (0, 1)$ with $u \ge v$. 
\begin{enumerate}[\normalfont (a)]
\item If $u \le w$ then $\Ib^{u, v}_s(x, y) \le \I_s^{w, \hor}(x, y)$ with equality if and only if $s > \Shp^{w, \hor}(x, y)$, $u = \min \{w, \Minx_+^s\}$ and $v = \Minx_-^s$, or $s \le \Shp^{w, \hor}(x, y)$ and $u = v$. 
\item If $u \le w$ then $\Ib^{u, v}_s(x, y) \le \I_s^{w, \hor}(x, y)$ with equality if and only if $s > \Shp^{z, \ver}(x, y)$, $u = \Minx_+^s$ and $v = \max \{z, \Minx_-^s\}$, or $s \le \Shp^{z, \ver}(x, y)$ and $u = v$. 
\end{enumerate}
\end{lem}
\begin{proof}
The claimed inequality is built in to definition \eqref{EIhv} when $u < w$ and extends to the case $u = w$ by continuity. It remains to determine when the equality holds. 

Consider the case $s \le \Shp^{w, \hor}(x, y)$. Then $\I_s^{w, \hor}(x, y) = 0 \ge \Ib_s^{u, v}(x, y)$ by \eqref{EIhpos}. Recalling \eqref{EIbdM} and \eqref{EShph}, it becomes clear that the inequality is strict if and only if $u > v$. 

Now the case $s > \Shp^{w, \hor}(x, y)$. From definitions \eqref{EIbd}, \eqref{ELhv} and \eqref{EIhv}, one develops 
\begin{align}
\I_s^{w, \hor}(x, y) &= \sup \limits_{\substack{u, v \,\in\, (0, w) \\ u > v}} \Ib_s^{u, v}(x, y) = \sup \limits_{\lambda \in (0, w)}\sup \limits_{\substack{u, v \,\in\, (0, w) \\ u-v = \lambda}} \{\lambda s - \Lb^{u, v}(x, y)\} = \sup \limits_{\lambda \in (0, w)}\{\lambda s-\L^{\lambda, w, \hor}(x, y)\} \nonumber\\
&= \sup \limits_{\lambda \in (0, w)}\{\lambda s - \Lb^{\min \{w, \Min_+^\lambda\}, \min \{w, \Min_+^\lambda\}-\lambda}(x, y)\}. \label{E30}
\end{align} 
Since $\I_s^{w, \hor}(x, y) > 0$ by \eqref{EIhpos}, $(u, v)$ pairs with $u = v$ can be omitted from the first supremum above. Line \eqref{E30} invokes Lemma \ref{LMinpm}(a). Using \eqref{ELbd} and \eqref{EMinpm}, the $\lambda$-derivative on $(0, w)$ of the expression inside the supremum in \eqref{E30} is seen to be 
\begin{alignat}{2}
\partial_\lambda \{\lambda s - \Lb^{\Min_+^\lambda, \Min_-^\lambda}(x, y)\} &= s-\partial_\lambda \{\Min_+^\lambda\}\M^{\Min_+^\lambda}(x, y)+\partial_\lambda\{\Min_-^\lambda\}\M^{\Min_-^\lambda}(x, y) \nonumber\\
&= s-\M^{\Min_+^\lambda}(x, y) = s-\M^{\Min_-^\lambda}(x, y) \qquad &&\text{ when } \Min_+^\lambda < w \label{E31}\\
\partial_\lambda \{\lambda s - \Lb^{w, w-\lambda}(x, y)\} &= s-\M^{w-\lambda}(x, y) \qquad &&\text{ when } \Min_+^\lambda \ge w. \label{E32}
\end{alignat}
The derivative exists at the transition point (which occurs when $w < 1$) since \eqref{E31}--\eqref{E32} are continuous in $\lambda$ and coincide when $\Min_+^\lambda = w$. Combining the two cases, the derivative can be written as
\begin{align}
s-\M^{\min \{w, \Min_+^\lambda\}-\lambda}(x, y) = s-\M^{\min \{w-\lambda, \Min_-^\lambda\}}(x, y). \label{E35}
\end{align}
Recalling \eqref{EMinpm}--\eqref{EMinpm3}, the superscript on the right-hand side is a decreasing continuous function of $\lambda$ from $(0, w)$ onto $(0, \min \{w, \Min\})$ where $\Min = \Min(x, y)$. Then it follows from Lemma \ref{LMProp} that the right-hand side of \eqref{E35} defines a continuous, decreasing function with range $(-\infty, s-\M^{\min\{w, \Min\}}(x, y)) = (-\infty, s-\Shp^{w, \hor}(x, y))$. Therefore, this function changes sign once and from positive to negative at unique $\lambda_0 \in (0, w)$ given by
\begin{align}
\min \{w-\lambda_0, \Min_-^{\lambda_0}\} = \Minx_-^s \in (0, \min \{w, \Min\}). \label{E36}
\end{align}
Returning to the development leading to \eqref{E30}, and using \eqref{E36} and the uniqueness from Lemma \ref{LMinpm}(a), one concludes that $\I_s^{w, \hor}(x, y) = \Ib^{u, v}_s(x, y)$ if and only if 
\begin{align}
u = \min \{w, \Min_+^{\lambda_0}\} = \min \{w, \Min_+^s\} \quad \text{ and } \quad v = \min \{w-\lambda_0, \Min_-^{\lambda_0}\} = \Minx_-^s. \label{E37}
\end{align}
The second equality in \eqref{E37} is a claim to be justified now. If $w > \Min_+^{\lambda_0}$ then, by \eqref{EMinpm} and \eqref{E36}, $\Min_-^{\lambda_0} = \Min_+^{\lambda_0}-\lambda_0 = \min \{w-\lambda_0, \Min_+^{\lambda_0}-\lambda_0\} = \min \{w-\lambda_0, \Min_-^{\lambda_0}\} = \Minx_-^s$. Therefore, $\Min_+^{\lambda_0} = \Minx_+^s$ in view of \eqref{EExtId}, and the claim holds. In the remaining case $w \le \Min_+^{\lambda_0}$, \eqref{E36} implies that $w-\lambda_0 = \Minx_-^s \le \Min_-^{\lambda_0}$. Then $\lambda_0 \le \Minx_+^s - \Minx_-^s$ by \eqref{EExtId} and since the function $\lambda \mapsto \Min_-^{\lambda}(x, y)$ is decreasing on $(0, 1)$. Consequently, $w = \lambda_0 + \Minx_-^s \le \Minx_+^s$, and the claim holds. 
\end{proof}

Extracted from the preceding proof, the next lemma relates \eqref{ELhv}--\eqref{ELbd2} with \eqref{EIhv}--\eqref{EIbdRate} through convex duality. The case of equality in part (a) is apparent in Figure \ref{Frate}. 
\begin{lem}
\label{LConv}
Let $x, y \in \bbR_{>0}$, $\lambda \in \bbR_{\ge 0}$, $w \in (0, 1], z \in [0, 1)$, and $s \in \bbR$. Abbreviate $\Minx_\pm^s = \Minx_\pm^s(x, y)$ and $\Min_\pm^\lambda = \Min_\pm^\lambda(x, y)$. 
\begin{enumerate}[\normalfont (a)]
\item $\I_s^{w, \hor}(x, y) + \L^{\lambda, w, \hor}(x, y) \ge \lambda s$ with equality if and only if $s > \Shp^{w, \hor}(x, y)$ and $\lambda = \min \{w, \Minx^s_+\} - \Minx^s_-$, or $s \le \Shp^{w, \hor}(x, y)$ and $\lambda = 0$. 
\item $\I_s^{z, \ver}(x, y) + \L^{\lambda, z, \ver}(x, y) \ge \lambda s$ with equality if and only if $s > \Shp^{z, \ver}(x, y)$ and $\lambda = \Minx^s_+ - \max \{z, \Minx^s_-\}$, or $s \le \Shp^{z, \ver}(x, y)$ and $\lambda = 0$. 
\end{enumerate}
\end{lem}
\begin{proof}
The inequality claimed in (a) holds strictly if $\lambda \ge w$ because then $\L^{\lambda, w, \hor}(x, y) = \infty$ by definition \eqref{ELhv}. In the case $\lambda = 0$, the inequality also holds since $\I^{w, \hor}_s (x, y) \ge 0$ and $\L^{0}(x, y) = 0$ as clear from definitions \eqref{EL} and \eqref{EI}. Equality occurs precisely when $s \le \Shp^{w, \hor}(x, y)$ by \eqref{EIhpos}. Now consider the case $\lambda \in (0, w)$. Then the claimed inequality follows from the third step in the derivation of \eqref{E30}. Since the derivative \eqref{E32} is not identically zero on $(0, w)$, \eqref{E30} and the nonnegativity of \eqref{EIhv} imply that $\I_s^{w, \hor}(x, y) > 0$. Therefore, $s > \Shp^{w, \hor}(x, y)$ by \eqref{EIhpos}. Now \eqref{E36}--\eqref{E37} show that the case of equality is precisely when $\lambda = \lambda_0 = \min \{w, \Minx_+^s\}-\Minx_-^s$. The proof of (b) is similar. 
\end{proof}

\begin{lem}
\label{LRTBdUB}
Let $m, n \in \bbZ_{>0}$, $w \in (0, 1]$, $z \in [0, 1)$ and $s \in \bbR$. Then 
\begin{enumerate}[\normalfont (a)]
\item $\log \P\{\Gb_{1, 0}^w(m, n) \ge s\} \le -\I_s^{w, \hor}(m, n)$. 
\item $\log \P\{\Gb_{0, 1}^z(m, n) \ge s\} \le -\I_s^{z, \ver}(m, n)$. 
\item $\log \P\{\Gb^{w, z}(m, n) \ge s\} \le -\I_s^{w, z}(m, n) + \log 2$. 
\end{enumerate}
\end{lem}
\begin{proof}
Let $\lambda \in [0, w)$. Then $\L^{\lambda, w, \hor}(m, n)$ is finite as evident from definition \eqref{ELhv}. Applying the exponential Markov's inequality and Lemmas \ref{LLMBdUB}(a) and \ref{LConv}(a) yields 
\begin{align*}
\log \P\{\Gb_{1, 0}^w(m, n) \ge s\} \le -\lambda s + \log \E[e^{\lambda \Gb_{1, 0}^w(m, n)}] \le -\lambda s + \L^{\lambda, w, \hor}(m, n) \le -\I_s^{w, \hor}(m, n)
\end{align*}
proving (a). The proof of (b) is similar. Using (a)-(b), \eqref{EGbmax}, a union bound and definition \eqref{EIbdRate}, one also obtains (c): 
\begin{align*}
\log \P\{\Gb^{w, z}(m, n) \ge s\} &\le \log 2 + \max \{\log \P\{\Gb^{w}_{1, 0}(m, n) \ge s\}, \log \P\{\Gb^{z}_{0, 1}(m, n) \ge s\}\} \\
&\le \log 2 + \max \{-\I_s^{w, \hor}(m, n), -\I_s^{z, \ver}(m, n)\} \\
&= \log2 -\I_s^{w, z}(m, n). \qedhere
\end{align*}
\end{proof}

The next lemma implies part (a) of Theorem \ref{TisLppMDRate} by taking $\epsilon = \epsilon_0$ and $s \in [s_0, \epsilon_0(m+n)^{2/3}]$. In addition, the lemma provides a large deviation upper bound that serves later in the proof of part (b) of the theorem. 
\begin{lem}
\label{LRTBdUB2}
Fix $\delta > 0$, $K \ge 0$, $p > 0$ and $s_0 > 0$. There exist constants $C_0 = C_0(\delta), \epsilon_0 = \epsilon_0(\delta)$ and $N_0 = N_0(\delta, K, p, s_0) > 0$ such that, with $t = \min \{s, \epsilon(m+n)^{2/3}\}$,  
\begin{align*}
\log \P\{\Gb^{w, z}(m, n) \ge \Shp(m, n) + \curv(m, n) s\} &\le - \frac{2t^{3/2}}{3}-\one_{\{s \ge t\}}(s-t)\sqrt{t}\\
&+\log 2 +\frac{C_0 Kt}{(m+n)^p}+\frac{C_0t^2}{(m+n)^{1/3}}
\end{align*}
whenever $(m, n) \in S_\delta \cap \bbZ_{\ge N_0}^2$, $\epsilon \in [s_0(m+n)^{-2/3}, \epsilon_0]$, $s \in \bbR_{\ge s_0}$ and $w, z \in (0, 1)$ with 
\begin{align}\max \{|w-\Min(m, n)|, |z-\Min(m, n)|\} \le K(m+n)^{-1/3-p}.\label{E53}\end{align} 
\end{lem}
\begin{proof}
Let $(m ,n) \in S_\delta$, and abbreviate $\Shp = \Shp(m, n)$, $\Min = \Min(m, n)$ and $\curv = \curv(m, n)$. Pick $w, z \in (0, 1)$ such that \eqref{E53} holds. Put $w' = \min \{w, \Min\}$ and $z' = \max \{z, \Min\}$. It follows from Lemma \ref{LShpMinBnd}(b) that there are constants $\eta_0 = \eta_0(\delta) > 0$, $\lambda_0 = \lambda_0(\delta) > 0$ and $N = N_0(\delta, K) > 0$ such that 
\begin{align}
w'-\lambda/2 \ge \eta_0 \quad \text{ and } \quad z'+\lambda/2 \le 1-\eta_0 \quad \text{ for } \lambda \in [0, \lambda_0] \text{ and } m, n \ge N_0. \label{E57}
\end{align}
Furthermore, by Lemma \ref{LShpMinBnd}(c), there exist $\epsilon_0 = \epsilon_0(\delta) > 0$ such that 
\begin{align}
\frac{K}{(m+n)^{1/3+p}} \le \frac{2\sqrt{s}}{\curv} \le \lambda_0 \quad \text{ for } s_0 \le s \le \epsilon_0 (m+n)^{2/3} \text{ and } m, n \ge N_0 \label{E58}
\end{align}
for sufficiently large $N_0 = N_0(\delta, K, p, s_0)$ depending also on $p$ and $s_0$. Above conditions on $m, n$ and $s$ are in force hereafter. Let $\epsilon \in [s_0(m+n)^{-2/3}, \epsilon_0]$. Consider the case $s \le \epsilon (m+n)^{2/3}$. Using definition \eqref{EIhv}, \eqref{E57}, \eqref{E58} and invoking Lemma \ref{LLIBdEst}(b), one finds that 
\begin{align}
\I_{\Shp+\curv s}^{w, \hor}(m, n) &\ge \Ib^{w', w'-\sqrt{s}/\curv}_{\Shp+\curv s}(m, n) \nonumber\\
&\ge s^{3/2}-\frac{\curv^3}{3}\{(w'-\Min)^3-(w'-\Min - \sqrt{s}/\curv)^3\}\nonumber\\
&-C_0(m+n)\{(w'-\Min)^4+(w'-\Min-\sqrt{s}/\curv)^4\} \nonumber\\
&\ge \frac{2s^{3/2}}{3}-\frac{\curv^3}{3}\{(w'-\Min)^3-(w'-\Min - \sqrt{s}/\curv)^3-(\sqrt{s}/\curv)^3\}\nonumber\\
&-C_0(m+n)\{(w'-\Min)^4+(\sqrt{s}/\curv)^4\} \nonumber\\
&\ge \frac{2s^{3/2}}{3}-C_0\bigg(\frac{Ks}{(m+n)^{p}}+\frac{s^2}{(m+n)^{1/3}}\bigg)\label{E59}
\end{align}
for some constant $C_0 = C_0(\delta) > 0$. The last step appeals to Lemma \ref{LShpMinBnd}(c) and uses the inequality $|w'-\Min| \le K(m+n)^{-1/3-p}$. When $s > \epsilon(m+n)^{2/3}$, applying \eqref{E59} with $S = \epsilon (m+n)^{2/3}$, one also has
\begin{align}
\I_{\Shp+\curv s}^{w, \hor}(m, n)  &\ge \Ib^{w', w'-\sqrt{S}/\curv}_{\Shp+\curv s}(m, n) = (s-S)\sqrt{S} + \Ib^{w', w'-\sqrt{S}/\curv}_{\Shp+\curv S}(m, n)\nonumber\\
&\ge (s-S)\sqrt{S} + \frac{2S^{3/2}}{3}-C_0\bigg(\frac{KS}{(m+n)^{p}}+\frac{S^2}{(m+n)^{1/3}}\bigg).\label{E60}
\end{align}
Bounds \eqref{E59}--\eqref{E60} can be put together as 
\begin{align*}
\I_{\Shp+\curv s}^{w, \hor}(m, n)  &\ge (s-t)\sqrt{S} + \frac{2t^{3/2}}{3}-C_0\bigg(\frac{Kt}{(m+n)^{p}}+\frac{t^2}{(m+n)^{1/3}}\bigg), 
\end{align*}
where $t = \min \{s, S\}$. The same bound also holds for $\I_{\Shp + \curv s}^{z, \ver}(m, n)$ via similar arguments. The result then follows from definition \eqref{EIbdRate} and Lemma \ref{LRTBdUB}. 
\end{proof}

\subsection{Lower bound for the l.m.g.f.}

Our development until now only required the ``$\le$'' half of Proposition \ref{PLMId}. The next lemma is the place in this section where the ``$\ge$'' half comes in. 

The restriction $\lambda \ge w-z$ in the statement is vacuous in the case $w \le z$. This permits the derivation of sharp lower bounds for right tail of the increment-stationary LPP from the lemma. The upper bound in Lemma \ref{LLMUB} is expected to be a sharp estimate up to a constant error for the l.m.g.f.\ of the bulk. 
This means that the conclusion of the next lemma should hold up to a constant in the most restrictive case $w = 1$ and $z = 0$. Therefore, it is plausible that the lemma essentially holds for any $\lambda \in \bbR_{\ge 0}$. 

\begin{lem}
\label{LLMBdLB}
Let $m, n \in \bbZ_{>0}$, $w \in (0, 1]$, $z \in [0, 1)$ and $\lambda \in \bbR_{\ge 0}$ such that $\lambda \ge w-z$. Then   
\begin{align*}
\log \E[e^{\lambda \Gb^{w, z}(m, n)}] \ge \L^{\lambda, w, z}(m, n).  
\end{align*}
\end{lem}
\begin{proof}
Begin with the case $\lambda < \min \{w, 1-z\}$. The first step below comes from the inequalities $w-\lambda \le z$ and $z+\lambda \ge w$, and the monotonicity of \eqref{Elppbd} in $w, z$. The subsequent steps use Proposition \ref{PLMId}, Lemma \ref{LMinpm}(a) and definitions \eqref{ELhv}--\eqref{ELbd2}.  
\begin{align*}
\log \E[e^{\lambda \Gb^{w, z}(m, n)}] &\ge \max \{\log \E[e^{\lambda \Gb^{w, w-\lambda}(m, n)}], \log \E[e^{\lambda \Gb^{z+\lambda, z}(m, n)}]\} \\
&= \max \{\Lb^{w, w-\lambda}(m, n), \Lb^{z+\lambda, z}(m, n)\} \\
&\ge \max \{\L^{\lambda, w, \hor}(m, n), \L^{\lambda, z, \ver}(m, n)\} = \L^{\lambda, w, z}(m, n). 
\end{align*}
In the case $\lambda \ge \min \{w, 1-z\}$, the bound also holds trivially since  
\begin{equation*}
\log \E[e^{\lambda \Gb^{w, z}(m, n)}]  \ge \max \{\log \E[e^{\lambda \w^{w}(1, 0)}], \log \E[e^{\lambda \w^{z}(0, 1)}]\} = \infty. \qedhere
\end{equation*}
\end{proof}



\subsection{Lower bound for the right tail}


We begin with developing some estimates towards the proof of the lower bound in Theorem \ref{TisLppMDRate}. 

The following lemma approximates the maximizers defined at \eqref{EMinxpm}. The reader might wonder why the proof is based on the implicit representation \eqref{EMinxpm} not on the exact formulas \eqref{EMinxpm2}. The reason is our wish to keep the argument more amenable to generalization to the settings (as in \cite{emra-janj-sepp-19-shp-long}) where the analogues of \eqref{EMinxpm} are in place yet exact formulas such as \eqref{EMinxpm2} are not available.   

\begin{lem}
\label{LMinxpmEst}
Fix $\delta > 0$. There exist constants $\epsilon_0 = \epsilon_0(\delta) > 0$ and $C_0 = C_0(\delta) > 0$ such that, with $s = \Shp(x, y) + \curv(x, y) t$,   
\begin{align*}
\bigg|\Minx_+^s(x, y) - \Min(x, y)-\frac{\sqrt{t}}{\curv(x, y)}\bigg| &\le \frac{C_0t}{(x+y)^{2/3}} \\ 
\bigg|\Minx_-^s(x, y) - \Min(x, y)+\frac{\sqrt{t}}{\curv(x, y)}\bigg| &\le \frac{C_0t}{(x+y)^{2/3}}
\end{align*}
whenever $(x, y) \in S_\delta$ and $t \in [0, \epsilon(x+y)^{2/3}]$. 
\end{lem}
\begin{proof}
By Lemmas \ref{LShpMinBnd}(a) and \ref{LMShpId}(a),
\begin{align}
\M^z(x, y) - \Shp = \frac{(z-\Min)^2\Shp}{z(1-z)} \ge (z-\Min)^2 (x+y) \quad \text{ for } x, y \in \bbR_{>0} \text{ and } z \in (0, 1), \label{E42}
\end{align}
where $\Shp = \Shp(x, y)$ and $\Min = \Min(x, y)$. Abbreviate $\Minx_{\pm}^s = \Minx_{\pm}^s(x, y)$ and $\curv = \curv(x, y)$. Let $\epsilon \in (0, 1/2)$ to be chosen below, and work with $t \le \epsilon (x+y)^{2/3}$. Since $\M^{\Minx_{\pm}^s}(x, y) = s = \Shp + \curv t$, setting $z = \Minx_{\pm}^s$ in \eqref{E42}, rearranging terms and using Lemma \ref{LShpMinBnd}(c), one obtains that 
\begin{align}
|\Minx_{\pm}^s-\Min| \le \sqrt{\frac{\curv t}{x+y}} \le \sqrt{\epsilon}c \quad \text{ for } (x, y) \in S_\delta \label{E43}
\end{align}
for some constant $c = c(\delta) > 0$. Lemma \ref{LShpMinBnd}(b) and \eqref{E43} imply that, with $\epsilon = \epsilon(\delta)$ chosen sufficiently small, 
\begin{align}
\Minx_{\pm}^s \in (c_1, 1-c_1) \quad \text{ for } (x, y) \in S_\delta \text{ for some constant } c_1 = c_1(\delta) > 0. \label{E44} 
\end{align}
By Lemma \ref{LMShpId}(b), there exists $C_0 = C_0(\delta) > 0$ such that 
\begin{align}
|\M^{z}(x, y)-\Shp-(z-\Min)^2 \curv^3| \le C_0 |z-\Min|^3(x+y) \quad \text{ for } (x, y) \in S_\delta \text{ and } z \in (c_1, 1-c_1). \label{E41}
\end{align}
In particular, with $z = \Minx_{\pm}^s$, \eqref{E41} gives 
\begin{align}
\bigg|(\Minx_{\pm}^s-\Min)^2 -\frac{t}{\curv^2}\bigg| \le C_0 |\Minx_{\pm}^s-\Min|^3 \le C_0 \sqrt{\epsilon} (\Minx_{\pm}^s-\Min)^2 \quad \text{ for } (x, y) \in S_\delta\label{E39}
\end{align} 
after dividing through by $\curv^3$ and using Lemma \ref{LShpMinBnd}(c) and \eqref{E43}. The second inequality in \eqref{E39} combined with the triangle inequality yields 
\begin{align}
(\Minx_{\pm}^s-\Min)^2 \le \frac{2t}{\curv^2} \label{E40}
\end{align}
upon choosing $\epsilon$ smaller if necessary to have $C_0\sqrt{\epsilon} \le 1/2$. Using \eqref{E40} in the first inequality in \eqref{E39}, that $\Minx_+^s \ge \Min$ and Lemma \ref{LShpMinBnd}(c) once more, one arrives at 
\begin{align*}
\bigg|\Minx_{+}^s-\Min-\frac{\sqrt{t}}{\curv}\bigg| = \frac{(\Minx_{+}^s-\Min)^2-\dfrac{t}{\curv^2}}{\Minx_{+}^s-\Min+\dfrac{\sqrt{t}}{\curv}} \le \frac{C_0\curv(\Minx_{+}^s-\Min)^3}{\sqrt{t}} \le \frac{C_0t}{\curv^2} \le \frac{C_0 t}{(x+y)^{2/3}} \quad \text{ for } (x, y) \in S_\delta. 
\end{align*}
This proves the claimed estimate for $\Minx_+^s$. The estimate for $\Minx_-^s$ follows similarly.  
\end{proof}

The next lemma estimates the functions from \eqref{EIhv}. 
\begin{lem}
\label{LIEst}
Fix $\delta > 0$ and $C > 0$. Let $(x, y) \in S_\delta$, and abbreviate $\Shp = \Shp(x, y)$, $\Min = \Min(x, y)$ and $\curv = \curv(x, y)$. Let $t \in \bbR_{>0}$ and $s = \Shp + \curv t$. There exist constants $c_0, C_0, K_0 > 0$ that depend only on $\delta$ and $C$ such that the following statements hold whenever $t \le c_0 (x+y)^{2/3}$. 
\begin{enumerate}[\normalfont (a)]
\item Let $w = \Min + \dfrac{u \sqrt{t}}{\curv}$ for some $u \in \bigg[-1+\dfrac{K_0 \sqrt{t}}{\curv}, C\bigg]$. Then 
\begin{align*}
\bigg|\I^{w, \hor}_{s}(x, y)-\frac{t^{3/2}}{3}\bigg(3\min \{u, 1\}-\min \{u^3, 1\}+2\bigg)\bigg| \le \frac{C_0t^2}{(x+y)^{1/3}}. 
\end{align*}
\item Let $z = \Min + \dfrac{v\sqrt{t}}{\curv}$ for some $v \in \bigg[-C, 1-\dfrac{K_0 \sqrt{t}}{\curv}\bigg]$. Then
\begin{align*}
\bigg|\I^{z, \ver}_{s}(x, y)-\frac{t^{3/2}}{3}\bigg(3\max \{v, -1\}-\max \{v^3, -1\}+2\bigg)\bigg| \le \frac{C_0t^2}{(x+y)^{1/3}}. 
\end{align*}
\end{enumerate}
\end{lem}
\begin{proof}
We only prove (a) since the proof of (b) is similar. By Lemmas \ref{LShpMinBnd}(b) and \ref{LMShpId}(b), 
\begin{align}
|\M^z(x, y)-\Shp-\curv^3 (z-\Min)^2| \le C_0(x+y)|z-\Min|^3 \quad \text{ for } z \in (0, 1) \text{ with } |z-\Min| \le \epsilon \label{E47}
\end{align}
for some constants $C_0 = C_0(\delta)$ and $\epsilon = \epsilon(\delta) > 0$. Let $K_0 = K_0(\delta, C) > 0$ be an unspecified constant for the moment. By virtue of Lemma \ref{LShpMinBnd}(c), one can pick a constant $c_0 = c_0(\delta, C) > 0$ such that the interval for $u$ is nonempty and $|w-\Min| \le \epsilon$ for $t \le c_0 (x+y)^{2/3}$. Work with such $t$ below. The next step is to ensure that $s$ represents a right tail deviation, namely, 
\begin{align}
s > \Shp^{w, \hor}(x, y). \label{E48}
\end{align}
This is clear in the case $u \ge 0$ in view of \eqref{EShph}. Consider the case $u \in (-1, 0)$ now. Then the estimate \eqref{E47} and Lemma \ref{LShpMinBnd}(c) give
\begin{align}
\Shp^{w, \hor}(x, y) &= \M^{w}(x, y) \le \Shp + \curv^3 (w-\Min)^2 + C_0(x+y)|w-\Min|^3 \nonumber\\
&= \Shp + \curv t u^2 + \frac{C_0(x+y) |u|^3 t^{3/2}}{\curv^3} \le \Shp + \curv t u^2 + C_0 t^{3/2} \nonumber \\
&\le \Shp + \curv t - c (x+y)^{1/3}(u+1) + C_0t^{3/2}\nonumber \\
&\le \Shp + \curv t + t^{3/2}(-cK_0+C_0) \nonumber\\
&= s  +  t^{3/2}(-cK_0+C_0) < s, \label{E46}
\end{align}
for some constant $c = c(\delta, C) > 0$, which does not depend on the choice of $K_0$. The final inequality in \eqref{E46} then holds provided that $K_0$ is chosen sufficiently large. 

By \eqref{E48} and Lemma \ref{LMinxpm}(a), one has 
\begin{align}
\I_s^{w, \hor}(x, y) = \Ib_s^{\min \{w, \Minx_+^s\}, \Minx_-^s}(x, y), \label{E49}
\end{align}
where $\Minx_{\pm}^s = \Minx_{\pm}^s(x, y)$. Appealing to Lemma \ref{LMinxpmEst}, one obtains the estimates
\begin{align}
\bigg|\min \{w, \Minx_+^s\}-\Min-\frac{\sqrt{t}}{\curv}\min \{u, 1\}\bigg| \le \frac{C_1t}{(x+y)^{2/3}} \quad \text{ and } \quad \bigg|\Minx_-^s-\Min+\frac{\sqrt{t}}{\curv}\bigg| \le \frac{C_1t}{(x+y)^{2/3}}. \label{E50}
\end{align}
for some constant $C_1 = C_1(\delta, C) > 0$. For $\lambda = \min \{w, \Minx_+^s\}-\Minx_-^s$, \eqref{E50} gives
\begin{align}
\bigg|\lambda - \frac{\sqrt{t}(\min \{u, 1\} + 1)}{\curv}\bigg| \le \frac{C_1t}{(x+y)^{2/3}}. \label{E51}
\end{align}
Consequently, after choosing $c_0$ smaller if necessary, $\lambda \ge 0$ and the second estimate in Lemma \ref{LLIBdEst}(b) applies to the right-hand side of \eqref{E49}. Hence, 
\begin{align}
\bigg|\I_s^{w, \hor}(x, y)-\lambda \curv t + \frac{\curv^3}{3}\bigg\{(\min\{w, \Minx_+^s\}-\Min)^3-(\Minx_-^s-\Min)^3\bigg\}\bigg| \le C_2 (x+y)\lambda^4 \label{E52}
\end{align}
for some constant  $C_2 = C_2(\delta, C) > 0$. Utilizing the bounds from \eqref{E50} and Lemma \ref{LShpMinBnd}(c) in \eqref{E52} and rearranging terms via the triangle inequality, one arrives at the bound in (a). 
\end{proof}

The final lemma of this section establishes the lower bound in Theorem \ref{TisLppMDRate}. 
\begin{lem}
\label{LLppBdRateLB}
Fix $\delta > 0$, $K \ge 0$ and $p > 0$. There exist constants $C_0 = C_0(\delta) > 0$, $\epsilon_0 = \epsilon_0(\delta) > 0$, $N_0 = N_0(\delta, K, p) > 0$, and absolute constants $s_0 > 0$ and $\eta > 0$ such that 
\begin{align*}
\log \P\{\Gb^{w, z}(m, n) \ge \Shp(m, n) + \curv(m, n) s\} \ge -\frac{2s^{3/2}}{3} - \eta s -\frac{C_0s^2}{(m+n)^{1/3}} 
\end{align*}
whenever $(m, n) \in S_\delta \cap \bbZ_{\ge N_0}^2$, $s \in [s_0, \epsilon_0(m+n)^{2/3}]$ and $w, z \in (0, 1)$ with 
\begin{align}\max \{|w-\Min(m, n)|, |z-\Min(m, n)|\} \le K(m+n)^{-1/3-p}.\label{E61}\end{align}
\end{lem}
\begin{proof}
Let $\epsilon_0(\delta) > 0$, $N_0 = N_0(\delta, K, p, \eta) > 0$, $s_0 > 0$ and $\eta > 0$ denote constants to be specified below. Let $(m, n) \in S_\delta \cap \bbZ_{\ge N_0}^2$ and $w, z \in (0, 1)$ satisfy \eqref{E61}. Let 
\begin{align}
s \in  [s_0, \epsilon_0(m+n)^{2/3}\hspace{0.9pt}], \label{E68}
\end{align}
after taking $N_0$ large enough to ensure that the above interval is nonempty. 

First, some preliminaries. Remembering definition \eqref{EIbdRate}, due to the symmetry, it causes no loss in generality below to assume that 
\begin{align}
\I_{\Shp + \curv s}^{w, z} = \I_{\Shp + \curv s}^{w, \hor}. \label{E69}
\end{align}
By virtue of Lemma \ref{LMinxpmEst}, after possibly shrinking $\epsilon_0$, 
\begin{align}
\max \bigg\{\bigg|\Minx_+^{\Shp+\curv s}-\frac{\sqrt{s}}{\curv}\bigg|, \bigg|\Minx_-^{\Shp+\curv s} + \frac{\sqrt{s}}{\curv}\bigg|\bigg\} \le \frac{cs}{(m+n)^{2/3}} \label{E100}
\end{align}
for some constant $c = c(\delta) > 0$. Also, by Lemma \ref{LShpMinBnd}(c), given any constant $C = C(\delta) > 0$, 
\begin{align}
\frac{\sqrt{s}}{\curv} \ge C\max \bigg\{\frac{K}{(m+n)^{1/3+p}}, \frac{s}{(m+n)^{2/3}}\bigg\} \quad \label{E97}
\end{align}
for sufficiently small $\epsilon_0$ and sufficiently large $N_0$. With    
\begin{align}
\lambda = \min \{w, \Minx_+^{\Shp+\curv s}\}-\Minx_-^{\Shp + \curv s}, \label{E88}
\end{align}
one obtains from \eqref{E61}, \eqref{E100} and \eqref{E97} that
\begin{align}
\bigg|\lambda - \frac{\sqrt{s}}{\curv}\bigg| &\le |\min \{w-\Min, \Minx_+^{\Shp + \curv s}-\Min\}| + \bigg|\Minx_-^{\Shp+\curv s}-\Min + \frac{\sqrt{s}}{\curv}\bigg| \nonumber \\
&\le \frac{K}{(m+n)^{1/3+p}} + \frac{cs}{(m+n)^{2/3}}. \label{E63}
\end{align}
The minimum above is attained by the first term because  
\begin{align}
\Minx_+^{\Shp+\curv s} - \Min \ge \frac{\sqrt{s}}{\curv} - \frac{cs}{(m+n)^{2/3}} \ge \frac{\sqrt{s}}{2\curv} \ge \frac{K}{(m+n)^{1/3+p}} \ge w-\Min, \label{E98}
\end{align}
where the second and third inequalities hold provided that $C \ge 2\max \{c, 1\}$ in \eqref{E97}. Putting together \eqref{E61}, \eqref{E97} and \eqref{E63} reveals that 
\begin{align}
\lambda \ge \frac{\sqrt{s}}{2\curv} \ge \frac{2K}{(m+n)^{1/3+p}} \ge w-z \quad \text{ for small enough $\epsilon_0$ and large enough $N_0$.} \label{E62.1}
\end{align}
Putting together \eqref{E61}, \eqref{E97} and \eqref{E63} reveals that 
\begin{align}
\lambda \ge \frac{\sqrt{s}}{2\curv} \ge \frac{2K}{(m+n)^{1/3+p}} \ge w-z \quad \text{ for small enough $\epsilon_0$ and large enough $N_0$.} \label{E62}
\end{align}

The derivation beginning with \eqref{E70} below uses \eqref{E69} in the first step. The second step applies Lemma \ref{LConv}(a) recalling the choice \eqref{E88} of $\lambda$ and verifying from definition \eqref{EShph}, Lemmas \ref{LShpMinBnd} and \ref{LMShpId}(a), and assumptions \eqref{E61} and \eqref{E68} that 
\begin{align}
\Shp^{w, \hor} \le \M^w \le \Shp + c(w-\Min)^2(m+n) \le \Shp + cK^2 (m+n)^{1/3-2p} < \Shp + \curv s_0 \le \Shp + \curv s \nonumber
\end{align}
for some constant $c = c(\delta) > 0$ and sufficiently large $N_0$. The first line ends with using definition \eqref{ELbd2}. The second inequality below is by virtue of Lemma \ref{LLMBdLB} and condition \eqref{E62}. The subsequent step rewrites the expectation as an integral via Fubini's theorem. Putting $t_0 = s_0/4$, the last inequality comes from breaking the preceding integration at $\Shp + \curv t_0$ into two parts, and then trivially bounding the probability in the first part. The final step computes the first integral and changes the variable in the second. 
\begin{align}
\exp\{-\I_{\Shp+\curv s}^{w, z}\} &= \exp\{-\I_{\Shp+\curv s}^{w, \hor}\} = \exp\{\L^{\lambda, w, \hor}-\lambda \Shp - \lambda \curv s\} \le \exp\{\L^{\lambda, w, z}-\lambda \Shp - \lambda \curv s\} \label{E70}\\
&\le e^{-\lambda \Shp-\lambda \curv s}\E[e^{\lambda \Gb^{w, z}}] = e^{-\lambda \Shp-\lambda \curv s}\bigg(1 + \lambda \int_0^\infty \P\{\Gb^{w, z} \ge t\} e^{\lambda t}\dd t\bigg)\nonumber \\
&\le e^{-\lambda \Shp-\lambda \curv s}\bigg(1 + \lambda \int_0^{\Shp + \curv t_0} e^{\lambda t} \dd t + \lambda \int_{\Shp + \curv t_0}^\infty \P\{\Gb^{w, z} \ge t\} e^{\lambda t} \dd t\bigg)\nonumber\\
&= e^{\lambda \curv (t_0-s)} + \lambda \curv \int_{t_0}^\infty \P\{\Gb^{w, z} \ge \Shp + \curv t\} e^{\lambda \curv (t-s)}\dd t.\label{E71}
\end{align}

The far left-hand side in \eqref{E70} can be bounded from below as follows: With $u = \dfrac{(w-\Min)\curv}{\sqrt{s}}$,  
\begin{align}
|u| \le \frac{c_0K}{\sqrt{s}(m+n)^p} \le \frac{c_0K}{\sqrt{s_0}(m+n)^p} \quad \text{ for some constant } c_0 = c_0(\delta) > 0\label{E89}
\end{align}
on account of \eqref{E61}, \eqref{E68} and Lemma \ref{LShpMinBnd}(c). Using \eqref{E89} and appealing to Lemma \ref{LShpMinBnd}(c) and \eqref{E68} yield 
\begin{align}
u \in \bigg[-1+\frac{K_0\sqrt{s}}{\curv}, 1\bigg] \quad \text{for small enough $\epsilon_0$ and large enough $N_0$,} \label{E120}
\end{align}
where $K_0 = K_0(\delta) > 0$ denotes the constant from Lemma \ref{LIEst}. Applying Lemma \ref{LIEst}(a) and using the triangle inequality along with \eqref{E69}, \eqref{E89}, \eqref{E120} and Lemma \ref{LShpMinBnd}(c) then lead to 
\begin{align}
\label{E90}
\begin{split}
\bigg|\I_{\Shp + \curv s}^{w, z} + \frac{2s^{3/2}}{3}\bigg| &\le \bigg|\I_{\Shp + \curv s}^{w, z} +\frac{s^{3/2}}{3}\bigg(3u-u^3+2\bigg)\bigg| + \frac{c_0 Ks}{(m+n)^p} \\
&\le \frac{c_0 Ks}{(m+n)^p} + \frac{C_0s^2}{(m+n)^{1/3}} 
\end{split}
\end{align}
after possibly decreasing $\epsilon_0$ and increasing $C_0, c_0$ and $N_0$. From \eqref{E90}, one gets 
\begin{align}
\exp\{-\I_{\Shp+\curv s}^{w, z}\}  \ge \exp\bigg\{-\frac{2s^{3/2}}{3}-\frac{c_0 Ks}{(m+n)^p}- \frac{C_0s^2}{(m+n)^{1/3}}\bigg\}. \label{E65}
\end{align}
The contribution from the first term  to \eqref{E71} can be bounded by means of \eqref{E97} and the lower bound in \eqref{E68}: For large enough $N_0$, 
\begin{align}
\exp\{\lambda \curv (t_0-s)\} &\le \exp\bigg\{-\frac{3\lambda \curv s}{4}\bigg\} \le \exp\bigg\{-\frac{17s^{3/2}}{24}\bigg\}. \label{E72}
\end{align}

The next task is to bound from above the second term in \eqref{E71}. For the sake of more compact notation, write 
\begin{align}
S = 4\epsilon_0(m+n)^{2/3} \quad \text{ and } \quad \mu = \mu(t) = \min \{t, S\} \quad \text{ for } t \in \bbR. \nonumber
\end{align}
By virtue of Lemma \ref{LRTBdUB2}, for small enough $\epsilon_0$ and large enough $C_0$ and $N_0$, 
\begin{align}
\P\{\Gb^{w, z} \ge \Shp + \curv t\} &\le 2\exp\bigg\{-\frac{2\mu^{3/2}}{3}-\one_{\{t \ge \mu\}}(t-\mu)\mu^{1/2}+\frac{C_0K\mu}{(m+n)^p}+\frac{C_0\mu^2}{(m+n)^{1/3}}\bigg\} \label{E73}
\end{align}
for $t \ge t_0$. Utilizing \eqref{E97} and \eqref{E63}, one also has 
\begin{align}
\lambda \curv e^{\lambda \curv (t-s)} &\le 2\sqrt{s} \exp\bigg\{\sqrt{s}(t-s)+C_0|t-s|\bigg(\frac{K}{(m+n)^{p}} + \frac{s}{(m+n)^{1/3}}\bigg)\bigg\} \label{E91}
\end{align}
for $t \ge t_0$ and large enough $C_0$ and $N_0$. Define 
\begin{align*}
F(t) &= \lambda \curv \P\{\Gb^{w, z} \ge \Shp + \curv t\}e^{\lambda \curv(t-s)} \quad \text{ for } t \in \bbR_{\ge 0}, 
\end{align*}
and work with large enough $s_0$ to have 
$$[s-\eta s^{1/2} , s+\eta s^{1/2}] \subset [t_0, 4s] \quad \text{ for } s \ge s_0.$$
Putting together \eqref{E73} and \eqref{E91}, and recalling \eqref{E68}, one obtains from Lemma \ref{LIntEst} that 
\begin{align}
\label{E103}
\begin{split}
&\bigg(\int_{t_0}^{s-\eta s^{1/2}} + \int_{s+\eta s^{1/2}}^{4s}\bigg) F(t) \dd t  \\
&\le 4\sqrt{s}\exp\bigg\{\frac{C_0Ks}{(m+n)^p} + \frac{C_0s^2}{(m+n)^{1/3}}\bigg\} \\
&\cdot \bigg(\int_{t_0}^{s-\eta s^{1/2}} + \int_{s+\eta s^{1/2}}^{4s}\bigg) \exp\bigg\{-\frac{2t^{3/2}}{3}+\sqrt{s}(t-s)\bigg\}\dd t \\ 
&\le 8\sqrt{s} (\sqrt{2\pi}s^{1/4}+s^{-1/2})\exp\bigg\{-\frac{2s^{3/2}}{3}-\frac{\eta^2 s}{8}+\frac{C_0Ks}{(m+n)^p} + \frac{C_0s^2}{(m+n)^{1/3}}\bigg\} \\
&\le \exp\bigg\{-\frac{2s^{3/2}}{3}-\eta s\bigg\}
\end{split}
\end{align}
after increasing $C_0$ and $N_0$ if necessary and choosing $s_0$ and $\eta$ sufficiently large. Let $a \in [1/\sqrt{2}, 1)$ denote an absolute constant. Proceeding as above but now keeping the error terms inside the integral gives 
\begin{align}
\label{E105}
\begin{split}
\int_{4s}^{S} F(t) \dd t &\le 4\sqrt{s}\int_{4s}^{\infty} \exp\bigg\{-\frac{2t^{3/2}}{3} + \sqrt{s}(t-s) + \frac{C_0 Kt}{(m+n)^p} + \frac{C_0t^2}{(m+n)^{1/3}}\bigg\}\dd t\\
&\le 4\sqrt{s}\int_{4s}^{\infty} \exp\bigg\{-\frac{2at^{3/2}}{3} + \sqrt{s}(t-s)\bigg\}\dd t \\
&\le 8\sqrt{s} \bigg(\frac{\sqrt{2\pi}s^{1/4}}{a}+s^{-1/2}\bigg) \exp\bigg\{-s^{3/2}\bigg(1-\frac{1}{3a^2}+\frac{a^2}{2}\bigg)\bigg\} \le \exp\{-s^{3/2}\},  
\end{split}
\end{align}
where the second inequality holds for small enough $\epsilon_0$ and large enough $N_0$, and the last inequality holds provided that $a$ is chosen sufficiently close to $1$. The third inequality invokes Lemma \ref{LIntEst} (with $\epsilon = 2s$ noting that the preceding exponent is maximized at $a^{-2}s \le 2s$). In a similar way and using $4s \le S$, one also has 
\begin{align}
\label{E104}
\begin{split}
&\int_{S}^\infty F(t) \dd t \\
&\le 4\sqrt{s}\exp\bigg\{-\frac{2S^{3/2}}{3} + \frac{C_0KS}{(m+n)^p} + \frac{C_0S^2}{(m+n)^{1/3}}\bigg\} \\
&\cdot \int_{S}^{\infty}\exp\bigg\{-(t-S)\sqrt{S} + \sqrt{s}(t-s) + C_0 (t-s)\bigg(\frac{K}{(m+n)^p} + \frac{s}{(m+n)^{1/3}}\bigg)\bigg\} \dd t \\
&\le 4\sqrt{s}\exp\bigg\{-\frac{2S^{3/2}}{3} + \sqrt{s}(S-s) + \frac{C_0KS}{(m+n)^p} + \frac{C_0S^2}{(m+n)^{1/3}}\bigg\}\\
&\cdot \int_{S}^{\infty}\exp\bigg\{-(t-S)\sqrt{S} + \sqrt{s}(t-S) + C_0 (t-S)\bigg(\frac{K}{(m+n)^p} + \frac{s}{(m+n)^{1/3}}\bigg)\bigg\} \dd t \\
&\le 4\sqrt{s}\int_{S}^{\infty}\exp\bigg\{-\frac{(t-S)\sqrt{S}}{2}\bigg\} \dd t \cdot \exp\bigg\{-\frac{S^{3/2}}{2}\bigg\}\le \exp\bigg\{-\frac{S^{3/2}}{2}\bigg\} \le \exp\{-4s^{3/2}\}
\end{split}
\end{align} 
for large enough $N_0$. 

Return now to \eqref{E70}--\eqref{E71}. Combining the bounds from \eqref{E65}, \eqref{E72}, \eqref{E103}, \eqref{E105} and \eqref{E104} yields 
\begin{align}
\label{E21}
\begin{split}
&\exp\bigg\{-\frac{2s^{3/2}}{3}-\frac{c_0 Ks}{(m+n)^p}- \frac{C_0s^2}{(m+n)^{1/3}}\bigg\} \\
&\le \int_{s-\eta s^{1/2}}^{s+\eta s^{1/2}} F(t) \dd t + \exp\bigg\{-\frac{2s^{3/2}}{3}-\eta s\bigg\} + 3\exp\{-b s^{3/2}\} \\
\end{split}
\end{align}
for some absolute constant $b > 2/3$. Then, with $s_0$ and $N_0$ sufficiently large and $\epsilon_0$ sufficiently small, \eqref{E21} leads to 
\begin{align}
\label{E130}
\begin{split}
&\frac{1}{2}\exp\bigg\{-\frac{2s^{3/2}}{3}-\frac{c_0 Ks}{(m+n)^p}- \frac{C_0s^2}{(m+n)^{1/3}}\bigg\} \le \int_{s-\eta s^{1/2}}^{s+\eta s^{1/2}} F(t) \dd t \\
&= \lambda \curv \int_{s-\eta s^{1/2}}^{s+\eta s^{1/2}}\P\{\Gb^{w, z} \ge \Shp + \curv t\}e^{\lambda \curv (t-s)} \dd t \\
&\le 2\lambda \curv \eta s^{1/2}\exp\{\lambda \curv \eta s^{1/2}\}\P\{\Gb^{w, z} \ge \Shp + \curv (s-\eta s^{1/2})\} \\
&\le 4 s\eta \exp\{2s\eta\}\P\{\Gb^{w, z} \ge \Shp + \curv (s-\eta s^{1/2})\}. 
\end{split}
\end{align}
The final inequality relies on \eqref{E97} and \eqref{E63}. Taking logarithms in \eqref{E130} and rearranging terms then give 
\begin{align}
\label{E131}
\begin{split}
&\log \P\{\Gb^{w, z} \ge \Shp + \curv (s-\eta s^{1/2})\} \\
&\ge -\frac{2s^{3/2}}{3}-2s\eta -\log (s\eta) - \frac{c_0 Ks}{(m+n)^p}- \frac{C_0s^2}{(m+n)^{1/3}}-\log 8 \\ 
&\ge  -\frac{2s^{3/2}}{3} - 3 s \eta - \frac{C_0s^2}{(m+n)^{1/3}}
\end{split}
\end{align}
after possibly increasing $s_0, \eta$ and $N_0$. Any $t \ge s_0$ with $2t \le \epsilon_0 (m+n)^{2/3}$ can be represented as $t = s-s^{1/2}\eta$ for some $s$ with \eqref{E68} and $s-t = s^{1/2}\eta \le 2t^{1/2}\eta$ provided that $s_0$ is large enough. Then \eqref{E131} implies that 
\begin{align*}
\log \P\{\Gb^{w, z} \ge \Shp + \curv t\} &\ge -\frac{2s^{3/2}}{3} - 3 s \eta - \frac{C_0s^2}{(m+n)^{1/3}} \ge -\frac{2t^{3/2}}{3} - Ct \eta - \frac{2C_0t^2}{(m+n)^{1/3}}
\end{align*}
for some absolute constant $C > 0$. The result then follows upon adjusting the constants $C_0$ and $\epsilon_0$ by a factor of $2$. 
\end{proof}

\begin{proof}[Proof of Theorem \ref{TisLppMDRate}]
Parts (a) and (b) are special cases of Lemmas \ref{LRTBdUB2} and \ref{LLppBdRateLB}, respectively.  
\end{proof}

To obtain Corollary \ref{CisLppMDRate}, one needs one more lemma comparing the minimizer \eqref{EMin} at different vertices. 
\begin{lem}
\label{LMinEst}
Let $x, y \in \bbR_{>0}$ and $\delta \in \bbR_{\ge 0}$. Then  
\begin{align*}
\Min(x+\delta, y)-\Min(x, y) &= \frac{\delta (1-\Min(x, y))}{\Shp(x+\delta, y)} \quad \text{ and } \quad \Min(x, y+\delta)-\Min(x, y) = -\frac{\delta \Min(x, y)}{\Shp(x, y+\delta)}. 
\end{align*}
\end{lem}
\begin{proof}
Recalling \eqref{EShp} and \eqref{EMin}, for example, the first identity can be verified as follows: 
\begin{align*}
\Min(x+\delta, y)-\Min(x, y) &= \frac{\sqrt{x+\delta}}{\sqrt{x+\delta}+\sqrt{y}}-\frac{\sqrt{x}}{\sqrt{x}+\sqrt{y}} = \frac{\sqrt{y}(\sqrt{x+\delta}-\sqrt{x})}{(\sqrt{x+\delta}+\sqrt{y})(\sqrt{x}+\sqrt{y})} \\
&= \frac{\delta(1-\Min(x, y))}{\Shp(x+\delta, y)}. \qedhere
\end{align*}
\end{proof}

\begin{proof}[Proof of Corollary \ref{CisLppMDRate}]
Pick $\delta > 0$ small enough to have $(x, y) \in S_\delta$. By homogeneity, Lemma \ref{LMinEst} and the bounds from Lemma \ref{LShpMinBnd}, 
\begin{align*}
\Min(\lc Nx \rc, \lc Ny \rc)-\Min(x, y) = \Min(\lc Nx \rc, \lc Ny \rc)-\Min(Nx, Ny) \le \frac{c}{N(x+y)}
\end{align*}
for some constant $c = c(\delta) > 0$. The result then follows from an application of Theorem \ref{TisLppMDRate}. 
\end{proof}

\section{Proofs of Theorem \ref{TExitPr} and Proposition \ref{PGeoInStep}}

\begin{lem}
\label{LlMEst}
Fix $\delta > 0$. There exists a constant $C_0 = C_0(\delta) > 0$ such that 
\begin{align*}
\bigg|\L^\lambda(x, y)-\lambda \Shp(x, y)-\frac{\lambda^3 \curv(x, y)^3}{12}\bigg| \le C_0 (x+y) \lambda^4 \quad \text{ for } (x, y) \in S_\delta. 
\end{align*}
\end{lem}
\begin{proof}
Combine Lemmas \ref{LLIBdEst}(b) and \ref{LMinpmEst}.  
\end{proof}

\begin{lem}
\label{LMinpmEst}
Fix $\delta > 0$. There exists a constant $C_0 = C_0(\delta) > 0$ such that 
\begin{align*}
|\Min_+^\lambda(x, y)-\Min(x, y)-\lambda/2| \le C_0 \lambda^2 \quad \text{ and } \quad |\Min_-^\lambda(x, y)-\Min(x, y)+\lambda/2| \le C_0 \lambda^2
\end{align*}
for $(x, y) \in S_\delta$ and $\lambda \in (0, 1)$. 
\end{lem}
\begin{proof}
Immediate from Lemma \ref{LShpMinBnd}(b) and \eqref{EMinpm2}. 
\end{proof}


\begin{proof}[Proof of Theorem \ref{TExitPr}]
By Lemma \ref{LShpMinBnd}(b), there are constants $c = c(\delta) > 0$ and $N_0 = N_0(\delta, K) > 0$ such that, whenever $(x, y) \in S_{\delta/2} \cap \bbR_{\ge N_0}^2$, 
\begin{align}
|z-\Min(x, y)| \le \frac{K}{(x+y)^{1/3}} \quad \text{ implies that } \quad z \in (c, 1-c). \label{E19}
\end{align}
Let $(m, n) \in S_\delta \cap \bbZ_{\ge N_0}^2$ and $z \in (0, 1)$ with $|z-\Min| \le K(m+n)^{-1/3}$ where $\Min = \Min(m, n)$. Let $0 < s \le \epsilon (m+n)^{1/3}$ and $k = \lc s(m+n)^{2/3} \rc$ where $\epsilon = \epsilon(\delta) > 0$ is chosen sufficiently small to ensure that $(m-k, n) \in S_{\delta/2}$. Let $\eta = \eta(\delta) > 0$ denote a constant to be tuned below. Put $\lambda = \eta s(m+n)^{-1/3}$.  Shrink $\epsilon>0$ if necessary to  have $4\lambda < c$. Hence, by \eqref{E19}, 
\begin{align}
4\lambda < z \quad \text{ and } \quad z-2\lambda \ge \frac{c}{2}. \label{E85}
\end{align}  

The first inequality in the next development 
comes from the stochastic monotonicity of the exponential distribution. 
The second line follows because the exponent there equals zero on the event $\{\Eh^{z-\lambda, z, \hor}(m, n) > k\}$ in view of definition \eqref{EEh}. The subsequent steps use the Cauchy-Schwarz inequality, independence and shift invariance. 
\begin{align*}
&\P\{\Eh^{z, \hor}(m, n) > k\}^2 \le \P\{\Eh^{z-\lambda, z, \hor}(m, n) > k\}^2 \\
&= \E\bigg[\one\{\Eh^{z-\lambda, z, \hor}(m, n) > k\}\exp\bigg\{\frac{\lambda}{2} \bigg(\Gb^{z-\lambda}(k, 0)+\Gb_{k+1, 0}^{z-\lambda}(m, n)-\Gb^{z-\lambda, z}(m, n)\bigg)\bigg\}\bigg]^2 \\
&\le \E\bigg[\exp\bigg\{\lambda \bigg(\Gb^{z-\lambda}(k, 0)+\Gb_{k+1, 0}^{z-\lambda}(m, n)\bigg)\bigg\}\bigg] \E\bigg[\exp\bigg\{-\lambda \Gb^{z-\lambda, z}(m ,n)\bigg\}\bigg] \\
&= \E\bigg[\exp\bigg\{\lambda \Gb^{z-\lambda}(k, 0)\bigg\}\bigg]\E\bigg[\exp\bigg\{\lambda \Gb_{k+1, 0}^{z-\lambda}(m, n)\bigg\}\bigg]\E\bigg[\exp\bigg\{-\lambda \Gb^{z-\lambda, z}(m ,n)\bigg\}\bigg] \\
&= \E\bigg[\exp\bigg\{\lambda \Gb^{z-\lambda}(k, 0)\bigg\}\bigg]\E\bigg[\exp\bigg\{\lambda \Gb_{1, 0}^{z-\lambda}(m-k, n)\bigg\}\bigg]\E\bigg[\exp\bigg\{-\lambda \Gb^{z-\lambda, z}(m ,n)\bigg\}\bigg]. 
\end{align*}
Then taking logarithms and appealing to Proposition \ref{PLMId} and Lemma \ref{LLMBdUB}(a) yield
\begin{align}
2\log \P\{\Eh^{z, \hor}(m, n) > k\} \le k \log \bigg(\frac{z-\lambda}{z-2\lambda}\bigg) + \L^{\lambda, z-\lambda, \hor}(m-k, n)- \Lb^{z, z-\lambda}(m, n). \label{E14}
\end{align}
The last term has a minus sign in front since the order of the parameters $z-\lambda$ and $z$ has been switched in the superscript, see definition \eqref{ELbd}. 

Write $\wt{\Shp}, \wt{\Min}, \wt{\Min}_+^{\lambda}$ and $\wt{\curv}$ for the values of the functions $\Shp, \Min, \Min_+^\lambda$ and $\curv$, respectively, evaluated at $(m-k, n)$. Using Lemmas \ref{LMinEst} and \ref{LShpMinBnd}(a)--(b), and the choice of $k$, one obtains that 
\begin{align}
\Min-\wt{\Min} = \Min(m, n)-\Min(m-k, n) = \frac{k(1-\Min(m-k, n))}{\Shp(m, n)} \in \bigg[\frac{a_0s}{(m+n)^{1/3}}, \frac{A_0 s}{(m+n)^{1/3}}\bigg] \label{E78}
\end{align} 
for some constants $a_0 = a_0(\delta) > 0$ and $A_0 = A_0(\delta) > 0$. Therefore, 
\begin{align}
z-\wt{\Min} = (z-\Min) + (\Min - \wt{\Min}) &\in \bigg[\frac{a_0s-K}{(m+n)^{1/3}}, \frac{A_0 s+K}{(m+n)^{1/3}}\bigg] \subset \bigg[\frac{a_0s}{2(m+n)^{1/3}}, \frac{2A_0 s}{(m+n)^{1/3}}\bigg] \label{E82}
\end{align}
for $s \ge s_0$ for some sufficiently large constant $s_0 = s_0(\delta, K) > 0$. Then, after choosing $\eta$ small enough, 
\begin{align}
z-\wt{\Min}-2\lambda \ge \bigg(\frac{a_0}{2}-2\eta\bigg)\frac{s}{(m+n)^{1/3}} \ge \frac{a_0s}{4(m+n)^{1/3}} \quad \text{ for  } s \ge s_0. \label{E84}
\end{align}
Also, by Lemma \ref{LMinpmEst},  
\begin{align}
\bigg|\wt{\Min}_+^\lambda - \wt{\Min}-\frac{\lambda}{2}\bigg| \le C_0\lambda^2 \quad \text{ for some constant } C_0 = C_0(\delta) > 0. \label{E79}
\end{align}
Combining \eqref{E84} and \eqref{E79} gives 
\begin{align}
z-\lambda - \wt{\Min}_+^\lambda = (z-\wt{\Min}) + (\wt{\Min}-\wt{\Min}_+^\lambda)-\lambda \ge 2\lambda + \bigg(\frac{\lambda}{2} - C_0\lambda^2\bigg) - \lambda = \frac{\lambda}{2}-C_0\lambda^2 \ge 0. \label{E81}
\end{align}
The last inequality holds with sufficiently small $\epsilon$ since $\lambda \le \eta \epsilon$. 

As a consequence of \eqref{E81}, $z-\lambda \ge \wt{\Min}_+^\lambda$. Therefore, the middle term on the right-hand side of \eqref{E14} equals $\L^{\lambda}(m-k, n)$ by virtue of Lemma \ref{LMinpm}(a). Hence, 
\begin{align}
2\log \P\{\Eh^{z, \hor}(m, n) > k\} &\le k \log \bigg(\frac{z-\lambda}{z-2\lambda}\bigg) + \L^{\lambda}(m-k, n)- \Lb^{z, z-\lambda}(m, n) \nonumber\\
&= k \log \bigg(\frac{z-\lambda}{z-2\lambda}\bigg) - k \log\bigg(\frac{z}{z-\lambda}\bigg) + \L^{\lambda}(m-k, n)- \Lb^{z, z-\lambda}(m-k, n)\nonumber \\
&= -k \log \bigg(1-\frac{\lambda^2}{(z-\lambda)^2}\bigg)+ \L^{\lambda}(m-k, n)- \Lb^{z, z-\lambda}(m-k, n). \label{E17}
\end{align}
The first term on the right-hand side of \eqref{E17} is at most 
\begin{align}
\frac{2k\lambda^2}{(z-\lambda)^2} \le b_0k \lambda^2 \le b_0 \eta^2s^3 \quad \text{ for some constant } b_0=b_0(\delta) > 0 \label{E18}
\end{align}
on account of the inequality $-\log(1-t) \le 2t$ for $t \in [0, 1/2]$, the condition $4\lambda < c < z$, and the choices of $k$ and $\lambda$. To bound the contribution of the remaining terms in \eqref{E17}, apply Lemmas \ref{LLIBdEst} and \ref{LlMEst} remembering that $z, z-\lambda \in [c/2, 1-c/2]$ from \eqref{E19} and \eqref{E85}. Then, for some constant $C_0 = C_0(\delta) > 0$, 
\begin{align}
\L^{\lambda}(m-k, n) - \Lb^{z, z-\lambda}(m-k, n) &\le \bigg(\lambda \wt{\Shp} + \frac{\lambda^{3}\wt{\curv}^3}{12}\bigg) - \bigg(\lambda \wt{\Shp} + \frac{\wt{\curv}^3}{3}\bigg\{(z-\wt{\Min})^3-(z-\lambda-\wt{\Min})^3\bigg\}\bigg) \nonumber\\
&+ C_0(m+n)\lambda^4 \nonumber \\
&= \wt{\curv}^3\bigg(-\lambda (z-\wt{\Min})^2 + \lambda^2 (z-\wt{\Min})-\frac{\lambda^3}{4}\bigg) + C_0(m+n)\lambda^4. \label{E83}
\end{align}
Recalling that $z-\wt{\Min} \ge 2\lambda$ from \eqref{E84} and using Lemma \ref{LShpMinBnd}(c) lead to 
\begin{align}
\eqref{E83} \le -c_0 (m+n) \lambda (z-\wt{\Min})^2 \le -c_0 \eta a_0^2 s^3 \quad \text{ for } s \ge s_0 \label{E86}
\end{align}
for some constant $c_0 = c_0(\delta) > 0$ provided that $\epsilon$ is sufficiently small. Comparing the powers of $\eta$-factors in \eqref{E18} and \eqref{E86}, one concludes from \eqref{E17} that, with $\eta$ sufficiently small,  
\begin{align}
\P\{\Eh^{z, \hor}(m, n) \ge s(m+n)^{2/3}\} \le \exp\{-c_0s^3\} \quad \text{ for } s \in [s_0, \epsilon (m+n)^{1/3}]. \label{E87}
\end{align}

Bound \eqref{E87} is vacuously true if $s \ge (m+n)^{1/3}$, and can be extended to the interval $s \in (\epsilon (m+n)^{1/3}, (m+n)^{1/3}]$ after adjusting $c_0$ by a constant factor depending only on $\epsilon = \epsilon(\delta)$. 

The analogous bound for the vertical exit point $\Ev^{z}$ can be similarly established. Then a union bound completes the proof. 
\end{proof}

\begin{proof}[Proof of Proposition \ref{PGeoInStep}]
By symmetry, it suffices to only prove (a). Assume $z < \Min$ and write $\lambda = (\Min-z)/4 > 0$. In the computations below, the vertex is fixed at $(m, n)$. Using monotonicity, definition \eqref{Elppbd}, the Cauchy-Schwarz inequality, Proposition \ref{PLMId}, Lemma \ref{LLMBdUB}(b) and definition \eqref{ELhv} in this order, one arrives at 
\begin{align*}
\P\{\Gb_{1, 0}^z \le \Gb_{0, 1}^z\} &\le \P\{\Gb_{1, 0}^z \le \Gb_{0, 1}^{z+2\lambda}\} = \E[\exp\{\lambda\Gb_{0, 1}^{z+2\lambda}-\lambda\Gb^{z, z+2\lambda}\}\one\{\Gb_{1, 0}^z \le \Gb_{0, 1}^{z+2\lambda}\}] \\
&\le \E[\exp\{\lambda \Gb_{0, 1}^{z+2\lambda}-\lambda \Gb^{z, z+2\lambda}\}] \\ 
&\le \E[\exp\{2\lambda \Gb_{0, 1}^{z+2\lambda}\}]^{1/2}\E[\exp\{-2\lambda \Gb^{z, z+2\lambda}\}]^{1/2} \\
&= \E[\exp\{2\lambda \Gb_{0, 1}^{z+2\lambda}\}]^{1/2}\exp\bigg\{\frac{1}{2}\Lb^{z, z+2\lambda}\bigg\} \\
&\le \exp\bigg\{\frac{1}{2}\Lb^{z+4\lambda, z+2\lambda} + \frac{1}{2}\Lb^{z, z+2\lambda}\bigg\}.
\end{align*}
The last exponent can be bounded by means of Lemmas \ref{LShpMinBnd}(c) and \ref{LLIBdEst}(b) as follows: 
\begin{align*}
\Lb^{z+4\lambda, z+2\lambda} + \Lb^{z, z+2\lambda} &= \Lb^{\Min, \Min-2\lambda} - \Lb^{\Min-2\lambda, \Min-4\lambda} \\
&\le \bigg(2\lambda \Shp + \frac{8\lambda^3\curv^3}{3}\bigg) - \bigg(2\lambda \Shp + \frac{(-8\lambda^3 + 64\lambda^3)\curv^3}{3}\bigg) + C_0 (m+n) \lambda^4\\
& = -16 \lambda^3 \curv^3+ C_0 (m+n) \lambda^4 \\
&\le -2c_0 (m+n)\lambda^3  + C_0(m+n) \lambda^4 \\ 
&\le -c_0(m+n)\lambda^3 
\end{align*}
provided that $\lambda \le \epsilon$ for some constants $C_0, c_0, \epsilon > 0$ depending only on $\delta$. This completes the proof in the case $\lambda \le \epsilon$. When $\lambda \in (\epsilon, 1)$, the same bound also holds after adjusting $c_0$ by a constant factor dependent only on $\epsilon = \epsilon(\delta)$. 
\end{proof}

\section{Proofs of Theorems \ref{TBuse} and \ref{TCif}}

To prove Theorem \ref{TBuse}, we combine Proposition \ref{PGeoInStep} with LPP processes with \emph{northeast} boundary weights, as opposed to the \emph{southwest} boundary weights in \eqref{Ewbd}. 
To introduce these processes, first define the weights 
\begin{align}
\label{Ewne}
\wne^{w, z, m, n}(i, j) = \eta(i, j) \bigg(\one_{\{i \le m, j \le n\}} + \frac{\one_{\{i \le m, j = n+1\}}}{w} + \frac{\one_{\{i = m+1, j \le n\}}}{1-z}\bigg)
\end{align}
for $w, z \in (0, 1)$, $m, n \in \bbZ_{>0}$, $i \in [m+1]$ and $j \in [n+1]$. Then let 
\begin{align}
\label{ELppne}
\Gne^{w, z, m, n}_{p, q}(k, l) = \max \limits_{\pi \in \Pi_{k, l}^{p, q}} \sum_{(i, j) \in \pi} \wne^{w, z, m, n}(i, j) \quad \text{ for } p, k \in [m+1] \text{ and } q, l \in [n+1]. 
\end{align}
A comparison of \eqref{Ewbd}--\eqref{Elppbd} with \eqref{Ewne}--\eqref{ELppne} shows the distributional identity
\begin{align}
\label{ElppDisId}
\begin{split}
&\{\Gne^{w, z, m, n}_{p, q}(k, l): p, k \in [m+1], q, l \in [n+1]\} \\
&\stackrel{\text{dist.}}{=} \{\Gb^{w, z}_{m+1-p, n+1-q}(m+1-k, n+1-l): p, k \in [m+1], q, l \in [n+1]\}
\end{split}
\end{align}
Below we  use \eqref{ELppne} only with $w = z$ and as before write $z$ only once in the superscript.

\begin{lem}
\label{LSqz}
Let $m, n \in \bbZ_{>0}$, $k \in [m-1] \cup \{0\}$, $l \in [n-1] \cup \{0\}$, $s \in \bbR^k$, $t  \in \bbR^l$ and $z \in (0, 1)$. Write $\underline{\Min} = \Min(m-k, n)$ and $\overline{\Min} = \Min(m, n-l)$. There exist constants $c = c(\delta) > 0$ and $\epsilon = \epsilon(\delta) > 0$ such that 
\begin{align*}
\Bcdf_{k, l}^{m, n, \hor}(s, t) &\le \bcdf_{k, l}^{z, \hor}(s, t) + \exp\{-c(m+n)(\underline{\Min}-z)^3\} \qquad \text{ if } z < \underline{\Min} \\
\Bcdf_{k, l}^{m, n, \hor}(s, t) &\ge  \bcdf_{k, l}^{z, \hor}(s, t) - \exp\{-c(m+n)(z-\overline{\Min})^3\} \qquad \text{ if } z > \overline{\Min} \\
\Bcdf_{k, l}^{m, n, \ver}(s, t) &\le \bcdf_{k, l}^{z, \ver}(s, t) + \exp\{-c(m+n)(z-\overline{\Min})^3\} \qquad \text{ if } z > \overline{\Min} \\
\Bcdf_{k, l}^{m, n, \ver}(s, t) &\ge  \bcdf_{k, l}^{z, \ver}(s, t) - \exp\{-c(m+n)(\underline{\Min}-z)^3\} \qquad \text{ if } z < \underline{\Min}
\end{align*}
provided that $k \le \epsilon m$ and $l \le \epsilon n$
\end{lem}
\begin{proof}
We prove the first two inequalities. The remaining two can be obtained in a similar manner. Write $E_0$ for the event 
\begin{align*}
\Gne^{z, m, n}_{m+1, n+1}(i, 1)- \Gne^{z, m, n}_{m+1, n+1}(i+1, 1) > s_i \quad \text{ for } i \in [k] \\
\Gne^{z, m, n}_{m+1, n+1}(1, j)-\Gne^{z, m, n}_{m+1, n+1}(1, j+1) \le t_j \quad \text{ for } j \in [l]. 
\end{align*}
On account of \eqref{EBurke}, \eqref{ElppDisId} and definition \eqref{Ebcdfh}, the probability of $E_0$ can be given exactly:  
\begin{align}
\P\{E_0\} = \bcdf_{k, l}^{z, \hor}(s, t). \label{E25}
\end{align}
Define the events $E_1$ and $E_2$ exactly as $E_0$ but  replace the base point $(m+1, n+1)$ with $(m, n+1)$ and $(m+1, n)$, respectively. From the union bound and \eqref{ElppDisId}--\eqref{E25}, one has 
\begin{align}
\P\{E_1\} &= \P\{E_1 \cap \{\Gne^{z, m, n}_{m, n+1}(k+1, 1) \ge \Gne^{z, m, n}_{m+1, n}(k+1, 1)\}\} \nonumber\\
&+ \P\{E_1 \cap \{\Gne^{z, m, n}_{m, n+1}(k+1, 1) < \Gne^{z, m, n}_{m+1, n}(k+1, 1)\}\}\nonumber\\
&\le \P\{E_0\} + \P\{\Gne^{z, m, n}_{m, n+1}(k+1, 1) < \Gne^{z, m, n}_{m+1, n}(k+1, 1)\} \nonumber\\ 
&= \P\{E_0\} + \P\{\Gb^{z}_{1, 0}(m-k, n) < \Gb^{z}_{0, 1}(m-k, n)\}. \label{E26}
\end{align}
In bounding the first term on the right-hand side above, we also utilized this consequence of planarity and a.s.\ uniqueness of geodesics: If the geodesic from $(m+1, n+1)$ to $(k+1, 1)$ visits $(m, n+1)$ then a.s.\ so does the geodesic from $(m+1, n+1)$ to any point in $\{(i, 1): i \in [k]\} \cup \{(1, j): j \in [l+1]\}$. 
Analogous reasoning gives 
\begin{align}
\P\{E_2\} &\ge \P\{E_0\}- \P\{\Gb^{z}_{1, 0}(m-k, n) < \Gb^{z}_{0, 1}(m-k, n)\}. \label{E27}
\end{align}

Pick $\epsilon = \epsilon(\delta) > 0$ sufficiently small to have $(m-k, n) \in S_{\delta/2}$. On the  two displayed lines below, the first  inequality holds   as a consequence of the definitions of the events $E_i$, $i \in \{0, 1, 2\}$, and the monotonicity in Lemma \ref{LCros}. The second inequality holds for a constant $c = c(\delta) > 0$, subject to the indicated restrictions on $z$ by virtue of \eqref{E26}--\eqref{E27} and Proposition \ref{PGeoInStep}. 
\begin{align*}
\Bcdf_{k, l}^{m, n, \hor}(s, t) &\le \P\{E_1\} \le \P\{E_0\} + \exp\{-c(m+n)(\underline{\Min}-z)^3\} \quad \text{ when } z < \underline{\Min} \\
\Bcdf_{k, l}^{m, n, \hor}(s, t) &\ge \P\{E_2\} \ge \P\{E_0\} - \exp\{-c(m+n)(z-\underline{\Min})^3\} \quad \text{ when } z > \overline{\Min}.
\end{align*}
These complete the proof in view of \eqref{E25}. 
\end{proof}

\begin{lem}
\label{LcdfLip}
Let $\delta > 0$, $k, l \in \bbZ_{\ge 0}$, $s \in \bbR^k$, $t  \in \bbR^l$ and $w, z \in (\delta, 1-\delta)$ with $w \ge z$. There exists a constant $C_0 = C_0(\delta) > 0$ such that 
\begin{align*}
0 &\le \bcdf_{k, l}^{z, \hor}(s, t) - \bcdf_{k, l}^{w, \hor}(s, t) \le C_0 (1+\log l) (w-z) \\
0 &\le \bcdf_{k, l}^{w, \ver}(s, t) - \bcdf_{k, l}^{z, \ver}(s, t) \le C_0 (1+\log k) (w-z).  
\end{align*}
\end{lem}
\begin{proof}
Writing $s = (s_i)_{i \in [k]}$ and $t = (t_j)_{j \in [l]}$, the derivative of $\bcdf^z = \bcdf_{k, l}^{z, \hor}(s, t)$ is given by 
\begin{align}
\partial_z f^z = -f^z \sum_{i \in [k]} s_i^+ - f^z \sum_{j \in [l]}\frac{t_j^+e^{-t_j^+(1-z)}}{1-e^{-t_j^+(1-z)}}, \label{E108}
\end{align}
where the  $j$th term in the second sum is interpreted as zero when $t_j \le 0$. Since $f^z > 0$, \eqref{E108} shows that $f^z$ is nonincreasing in $z$. To obtain the second inequality of the lemma, first note that the absolute value of the first term in \eqref{E108} is at most 
\begin{align}
f^z\sum_{i \in [k]}s_i^+  \le \sum_{i \in [k]}s_i^+ \exp\bigg\{-\sum_{i \in [k]}s_i^+ z\bigg\} \le \frac{1}{z}\sup_{t \ge 0} \{te^{-t}\} = \frac{1}{ez}. \label{E110}
\end{align}
Next bound the second term in \eqref{E108} in absolute value from above by the function 
\begin{align}
\varphi_l(t) = \sum_{p \in [l]} t_p^+e^{-t_p^+(1-z)} \prod_{\substack{j \in [l] \\ j \neq p}} (1-e^{-t_j^+(1-z)}). \label{E109}
\end{align}
If $t_p \le 0$ for some $p \in [l]$ then all term vanish on the right-hand side. In the case $l \in \{0, 1\}$, one has $\varphi(t) \le e^{-1}(1-z)^{-1}$ similarly to \eqref{E110}. Assume that $l > 1$ and $t_p > 0$ for $p \in [l]$ from here on. Our objective is to maximize $\varphi$ over $\bbR^l$. To aid the next computation, change the variables via $u_p = 1-\exp\{-t_p(1-z)\} \in (0, 1)$ for $p \in [l]$. Then \eqref{E109} turns into the following function of $u = (u_p)_{p \in [l]} \in (0, 1)^l$: 
\begin{align}
\psi_l(u) = -\frac{1}{1-z}\sum_{p \in [l]} (1-u_p) \log (1-u_p) \prod_{\substack{j \in [l] \\ j \neq p}} u_j. \label{E114}
\end{align} 
Note that $\psi_l$ extends continuously to $[0, 1]^l$, and the boundary values are given by 
\begin{align}
\psi_l(u)|_{u_j = 0} = 0 \quad \text{ and } \quad \psi_l(u)|_{u_j = 1} = \psi_{l-1}(u^j) \quad \text{ for } j \in [l], \label{E117}
\end{align}
where $u^j \in \bbR^{l-1}$ is obtained from $u$ by deleting the $j$th coordinate. 

The partial derivatives of $\psi_l$ are given by 
\begin{align}
(1-z)\partial_r \psi_l(u) &= (1+\log (1-u_r))\prod_{\substack{j \in [l] \smallsetminus \{r\}}} u_j \nonumber\\
& -\sum_{p \in [l] \smallsetminus \{r\}}(1-u_p) \log (1-u_p)\prod_{\substack{j \in [l] \smallsetminus \{p, r\}}} u_j  \nonumber \\
&= \prod_{\substack{j \in [l] \smallsetminus \{r\}}} u_j \cdot \bigg(1 + \log (1-u_r) - \sum_{p \in [l] \smallsetminus \{r\}}\bigg(\frac{1}{u_p}-1\bigg)\log(1-u_p)\bigg). \label{E111} 
\end{align}
Note from \eqref{E111} that $u \in (0, 1)^l$ is a zero of the gradient $\nabla \psi$  if and only if 
\begin{align}
1+\frac{\log (1-u_r)}{u_r} = \sum_{p \in [l]}\bigg(\frac{1}{u_p}-1\bigg)\log(1-u_p) \quad \text{ for each } r \in [l].  \label{E112}
\end{align}
Since the function $x \mapsto x^{-1}\log (1-x)$ is strictly decreasing on $(0, 1)$, \eqref{E112} holds if and only if the coordinates $u_r = v$ for $r \in [l]$ for some $v \in (0, 1)$ such that   
\begin{align}
1 + \bigg(l-\frac{l-1}{v}\bigg) \log (1-v) = 0. \label{E113}
\end{align}
The left-hand side of \eqref{E113} defines a continuous function $g = g(v)$ on $(0, 1)$ that is decreasing since its derivative
\begin{align*}
g'(v) = \frac{(l-1)}{v^2}\log (1-v)-\bigg(l-\frac{l-1}{v}\bigg)\frac{1}{1-v} = -\frac{1}{1-v} + \frac{(l-1)}{v^2}\bigg(\log (1-v) + v\bigg) < 0, 
\end{align*}
where the last inequality can be seen from the expansion $\log (1-t) = -\sum_{i=1}^\infty t^i/i$. Furthermore, the limits of $g$ at the endpoints $0$ and $1$ can be computed as $l$ and $-\infty$, respectively. Therefore, there exists a unique $v \in (0, 1)$ such that \eqref{E113} holds. The next step is to verify that 
\begin{align}
v \le 1-\frac{1}{e^2l}. \label{E116}
\end{align}
Arguing by contradiction, suppose that \eqref{E116} is false. Then, by the monotonicity of $g$, 
\begin{align}
0 < g\bigg(1-\frac{1}{e^2l}\bigg) = 1 - \frac{1-\dfrac{1}{e^2}}{1-\dfrac{1}{e^2l}} (2 + \log l) \le 1 - 2\bigg(1-\frac{1}{e^2}\bigg) < 0, 
\end{align}
a contradiction. Therefore, \eqref{E116} holds. 

Now, setting $u_j = v$ for $j \in [l]$ in \eqref{E114} leads to 
\begin{align}
(1-z)\psi_l((v)_{j \in [l]}) &= -l(1-v) \log(1-v) v^{l-1} \nonumber \\
&\le -(1-v) \log(1-v)-(l-1)(1-v) \log(1-v) v^{l-1} \nonumber \\
&= -(1-v) \log(1-v)-(1+\log(1-v)) v^{l} \label{E115}\\
&\le \frac{1}{e} +(1+\log l)v^l \le 2 + \log l. \label{E118}
\end{align}
Line \eqref{E115} above comes from \eqref{E113}. The first inequality in \eqref{E118} bounds the two terms of \eqref{E115} separately and uses \eqref{E116}. 

Putting together the bounds \eqref{E110} and \eqref{E118}, using the positivity of \eqref{E114} and recalling the boundary values from \eqref{E117}, one concludes that 
\begin{align}
\psi_l(u) \le C_0 (1+\log l) \quad \text{ for some constant } C_0 = C_0(\delta) > 0. \label{E119}
\end{align}
The second inequality of the lemma follows from \eqref{E119} and the mean value theorem. The proofs of the remaining inequalities are similar. 
\end{proof}

\begin{proof}[Proof of Theorem \ref{TBuse}]
Let $(m, n) \in S_\delta$ and $k \in [m-1]\cup\{0\}, l \in [n-1]\cup\{0\}$ with $k, l \le \epsilon_0 (m+n)^{2/3}$. Let 
\begin{align*}
w = \overline{\Min} + \bigg(\frac{\eta \log (m+n)}{m+n}\bigg)^{1/3} \quad \text{ and } \quad z = \underline{\Min}-\bigg(\frac{\eta \log (m+n)}{m+n}\bigg)^{1/3}, 
\end{align*}
where $\underline{\Min} = \Min(m-k, n)$, $\overline{\Min} = \Min(m, n-l)$ and $\eta = \eta(\delta) > 0$ is a constant to be chosen below. Lemmas \ref{LShpMinBnd}(b) and \ref{LMinEst} combined with the assumptions $k, l \le \epsilon_0 (m+n)^{2/3}$ imply that $w, z \in (\epsilon, 1-\epsilon) > 0$ for some constant $\epsilon = \epsilon(\delta) > 0$, provided that $m, n \ge N_0$ for some sufficiently large $N_0 = N_0(\delta, \epsilon_0) > 0$. From the the choice of $z$ and the first inequalities in Lemmas \ref{LSqz} and \ref{LcdfLip}, one obtains that
\begin{align*}
\Bcdf_{k, l}^{m, n, \hor}(s, t) &\le \bcdf_{k, l}^{z, \hor}(s, t) + \exp\{-c(m+n)(\underline{\Min}-z)^3\} \\
&= \bcdf_{k, l}^{z, \hor}(s, t) + \frac{1}{(m+n)^{\eta c}} \\
&\le \bcdf_{k, l}^{\Min, \hor}(s, t) + C(1+\log l) (\Min-z) + \frac{1}{(m+n)^{\eta c}} \\
&= \bcdf_{k, l}^{\Min, \hor}(s, t) + C(1+\log l) \bigg\{(\Min-\underline{\Min}) + \bigg(\frac{\eta \log (m+n)}{m+n}\bigg)^{1/3}\bigg\} + \frac{1}{(m+n)^{\eta c}} \\
&\le \bcdf_{k, l}^{\Min, \hor}(s, t) + 2C (1+\log l) \bigg(\frac{\eta \log (m+n)}{m+n}\bigg)^{1/3} + \frac{1}{(m+n)^{\eta c}}
\end{align*}
for some constants $c = c(\delta) > 0$, $C = C(\delta) > 0$ and sufficiently large $N_0$. The last inequality above appeals to Lemma \ref{LMinEst} and the assumption $k \le \epsilon_0(m+n)^{2/3}$. Choosing $\eta = \dfrac{1}{3c}$ yields 
\begin{align*}
\Bcdf_{k, l}^{m, n, \hor}(s, t)  \le \bcdf_{k, l}^{\Min, \hor}(s, t) + C_0 (1+\log l) \bigg(\frac{\log (m+n)}{m+n}\bigg)^{1/3}
\end{align*}
for some constant $C_0 = C_0(\delta) > 0$. The complementary lower bound is established similarly. The second set of bounds in the theorem are also proved similarly. 
\end{proof}

\subsection{Proof of Theorem \ref{TCif}}

\begin{proof}[Proof of Theorem \ref{TCif}]
Let $n \in \bbZ_{>1}$ and $k \in [n-1]$. Restrict to the full probability event on which the competition interface $\cif$ is well-defined. In the next display, the first equality follows from \eqref{ETvmon} and definition \eqref{Ecif}, and the subsequent equalities are due to definitions \eqref{ETv} and \eqref{EBuse}.  
\begin{align}
\label{E121}
\begin{split}
\{\cif^{\hor}_n > k\} &= \{(k+1, n-k+1) \in \tree^{\ver}\} \\
&= \{\G_{2, 1}(k+1, n-k+1) < \G_{1, 2}(k+1, n-k+1)\} \\
&= \{\B^{\hor}_{1, 1}(k+1, n-k+1) > \B^{\ver}_{1, 1}(k+1, n-k+1\}. 
\end{split}
\end{align}

Put $\Min = \Min(k+1, n-k)$ and let $z > \Min$ to be chosen below. The next derivation begins with \eqref{E121}. The first inequality holds by virtue of Lemma \ref{LCros}. The second inequality applies a union bound using the following implication of planarity and the uniqueness of geodesics (as in the proof of Lemma \ref{LSqz}): If the vertex $(k+2, n-k+1)$ is on the geodesic from $(1, 2)$ to $(k+2, n-k+2)$ then it must also be on the two geodesics from $(1, 1)$ and $(2, 1)$ to $(k+2, n-k+2)$. In terms of the northeast LPP process, 
this means that the inequality $\Gne_{k+1, n-k+2}^{z}(1, 2) < \Gne_{k+2, n-k+1}^{z}(1, 2)$ implies that $\Gne_{k+2, n-k+1}^{z}(i, j) = \Gne_{k+2, n-k+2}^{z}(i, j)$ for $(i, j) \in \{(1, 2), (1, 1), (2, 1)\}$, where the vertex $(k+2, n-k+2)$ is omitted from the superscript for brevity. The second and third equalities in the derivation follow from \eqref{ElppDisId} and \eqref{EBurke}, respectively. 
\begin{align}
&\P\{\cif^{\hor}_n \le k\} \nonumber\\
&= \P\{\B^{\hor}_{1, 1}(k+1, n-k+1) \le \B^{\ver}_{1, 1}(k+1, n-k+1)\} \nonumber\\
&\le \P\{\Gne_{k+2, n-k+1}^{z}(1, 1)-\Gne_{k+2, n-k+1}^z(2, 1) \le \Gne_{k+2, n-k+1}^{z}(1, 1)-\Gne_{k+2, n-k+1}^z(1, 2)\}\nonumber\\
&\le \P\{\Gne_{k+2, n-k+2}^{z}(1, 1)-\Gne_{k+2, n-k+2}^z(2, 1) \le \Gne_{k+2, n-k+2}^{z}(1, 1)-\Gne_{k+2, n-k+2}^z(1, 2)\}\nonumber\\
&\qquad+ \P\{\Gne_{k+1, n-k+2}^{z}(1, 2) > \Gne_{k+2, n-k+1}^{z}(1, 2)\} \nonumber\\
&= \P\{\Gb^{z}(k+1, n-k+1)-\Gb^z(k, n-k+1) \nonumber\\
&\qquad\qquad\qquad\qquad
\le \Gb^{z}(k+1, n-k+1)-\Gb^z(k+1, n-k)\}\nonumber\\
&\qquad + \P\{\Gb_{1, 0}^{z}(k+1, n-k) > \Gb_{0, 1}^{z}(k+1, n-k)\} \nonumber\\
&= \int_0^\infty ze^{-zx}\int_x^\infty (1-z)e^{-(1-z)y} \dd y \dd x + \P\{\Eh^{z, \hor}(k+1, n-k) > 0\}\nonumber\\
&= z + \P\{\Eh^{z, \hor}(k+1, n-k) > 0\}. \label{E122}
\end{align}

To avoid a vacuous statement, assume that $\delta \in (0, 1/2)$. Work with $k \in [\delta n, (1-\delta) n]$ and $n \ge N_0$ for some sufficiently large $N_0 = N_0(\delta) > 0$ that ensures that the preceding interval contains some integers. Then, by the assumption $z>\Min$, Lemma \ref{LShpMinBnd}(c) and Proposition \ref{PGeoInStep},  
\begin{align}
\P\{\Eh^{z, \hor}(k+1, n-k) > 0\} \le \exp\{-cn(z-\Min)^3\}  \label{E123}
\end{align}
for some constant $c = c(\delta) > 0$.    Set $z = \Min + \bigg(\dfrac{\log n}{3cn}\bigg)^{1/3}$ after increasing $N_0$ if necessary to have $z \in (0, 1)$. Resuming from \eqref{E122} and using \eqref{E123}, one obtains that 
\begin{align}
\P\{\cif^{\hor}_n \le k\} \le \Min + C\bigg(\frac{\log n}{n}\bigg)^{1/3} \label{E124}
\end{align}
for some constant $C = C(\delta) > 0$. For $x \in [\delta, 1-\delta]$, setting $k = \lc nx \rc$ in \eqref{E123} and using Lemma \ref{LMinEst} yield
\begin{align*}
\P\{\cif^{\hor}_n \le nx\} &\le \P\{\cif^{\hor}_n \le \lc nx \rc\} \le \Min(\lc nx \rc+1, n-\lc nx \rc) + C\bigg(\frac{\log n}{n}\bigg)^{1/3} \\
&\le \Min(x, 1-x) + \frac{c_0}{n} +C\bigg(\frac{\log n}{n}\bigg)^{1/3} \le \frac{\sqrt{x}}{\sqrt{x}+\sqrt{1-x}} + C_0\bigg(\frac{\log n}{n}\bigg)^{1/3}
\end{align*}
for some constants $c_0 = c_0(\delta) > 0$ and $C_0 = C_0(\delta) > 0$. The last bound also holds for $n \in \{2, \dotsc, N_0\}$ after adjusting $C_0$. 

To prove complementary lower bound, use Lemma \ref{LCros} to replace 
the first inequality  of the long display  above with 
\begin{align*}
&\P\{\B^{\hor}_{1, 1}(k+1, n-k+1) \le \B^{\ver}_{1, 1}(k+1, n-k+1)\}  \\
&\ge \P\{\Gne_{k+1, n-k+2}^{z}(1, 1)-\Gne_{k+1, n-k+2}^z(2, 1) \le \Gne_{k+1, n-k+2}^{z}(1, 1)-\Gne_{k+1, n-k+2}^z(1, 2)\}
\end{align*} 
and follow similar steps.  
The proof of the second bound is completely analogous. 
\end{proof} 

\appendix

\section{}

\subsection{Some integral bounds}

\begin{lem}
\label{LGausTailBnd}
$
\displaystyle \int_x^\infty e^{-t^2/2} \dd t \le \frac{\sqrt{\pi}}{\sqrt{2}}e^{-x^2/2} 
$
\quad for $x \in \bbR_{>0}$. 
\end{lem}
\begin{proof}
\begin{equation*}
\int_x^\infty e^{-t^2/2}\dd t = e^{-x^2/2}\int_0^\infty e^{-sx - s^2/2} \dd s \le e^{-x^2/2} \int_0^\infty e^{-s^2/2} \dd s = \frac{\sqrt{\pi}e^{-x^2/2}}{\sqrt{2}}. \qedhere 
\end{equation*}
\end{proof}

\begin{lem}
\label{LIntEst}
Let $a, b, s > 0$ and $f(t) = \dfrac{2at^{3/2}}{3}-b \sqrt{s}(t-s)$ for $t \in \bbR_{\ge 0}$. Let $t_0 = \dfrac{b^2s}{a^2}$ and $0 < \epsilon < t_0$. Then 
\begin{align*}
\int_{0}^{t_0-\epsilon} e^{-f(t)}\dd t &\le \frac{\sqrt{\pi b}s^{1/4}}{a}\exp\bigg\{-s^{3/2}\bigg(b-\frac{b^3}{3a^2}\bigg)-\frac{\epsilon^2 a^2}{4b\sqrt{s}}\bigg\} \\
\int_{t_0+\epsilon}^\infty e^{-f(t)}\dd t &\le \bigg(\frac{\sqrt{2\pi b}s^{1/4}}{a} + \frac{1}{b\sqrt{s}}\bigg) \exp\bigg\{-s^{3/2}\bigg(b-\frac{b^3}{3a^2}\bigg)-\frac{\epsilon^2 a^2}{8b\sqrt{s}}\bigg\}. 
\end{align*}
\end{lem}
\begin{proof}
The first two derivatives of $f$ read $f'(t) = a\sqrt{t}-b\sqrt{s}$ and $f''(t) = \dfrac{a}{2\sqrt{t}}$ for $t \in \bbR_{> 0}$. These show that $f'(t_0) = 0$ and $f'(t) \ge \dfrac{a}{2\sqrt{t_0}}$ for $t \in (0, t_0]$. Therefore, 
\begin{align}
f(t) = f(t_0)-\int_{t}^{t_0} f'(s) \dd s = f(t_0)+\int_{t}^{t_0} \int_{s}^{t_0} f''(u) \dd u \ge f(t_0)+\frac{a(t-t_0)^2}{4\sqrt{t_0}} \label{E107}
\end{align}
for $t \in [0, t_0]$. From this lower bound and Lemma \ref{LGausTailBnd}, one obtains that 
\begin{align}
\int_{0}^{t_0-\epsilon} e^{-f(t)}\dd t &\le e^{-f(t_0)}\int_{-\infty}^{t_0-\epsilon} \exp\bigg\{-\frac{a(t-t_0)^2}{4\sqrt{t_0}}\bigg\} \dd t = \frac{\sqrt{2}t_0^{1/4}e^{-f(t_0)}}{\sqrt{a}}\int^{\infty}_{\epsilon a^{1/2}2^{-1/2}t_0^{-1/4}}e^{-s^2/2}\dd s \nonumber \\
&\le \frac{\sqrt{\pi}t_0^{1/4}}{\sqrt{a}}\exp\bigg\{-f(t_0)-\frac{a\epsilon^2}{4\sqrt{t_0}}\bigg\}. \label{E106}
\end{align}
Inserting the value of $t_0$ in \eqref{E106} gives the claimed bound for the first integral. 

To bound the second integral, introduce $T > t_0+\epsilon$ to be chosen below. Since $f'(t) \ge \dfrac{a}{2\sqrt{T}}$ for $t \in (0, T]$, proceeding as in \eqref{E107}, one has 
\begin{align*}
f(t) \ge f(t_0) + \frac{a(t-t_0)^2}{4\sqrt{T}} \quad \text{ for } t \in [0, T]. 
\end{align*}
Also, since $f'$ is positive on $[t_0, \infty)$ and is increasing, 
\begin{align*}
f(t) \ge f(T) + (t-T)f'(T) \quad \text{ for } t \ge T. 
\end{align*}
Utilizing these lower bounds and invoking Lemma \ref{LGausTailBnd} once more, one obtains that 
\begin{align}
\int_{t_0+\epsilon}^\infty e^{-f(t)}\dd t &\le e^{-f(t_0)}\int_{t_0+\epsilon}^{\infty} \exp\bigg\{-\frac{a(t-t_0)^2}{4\sqrt{T}}\bigg\}\dd t + e^{-f(T)}\int_{T}^\infty \exp\bigg\{-(t-T)f'(T)\bigg\}\dd t \nonumber \\
&\le \frac{\sqrt{\pi}T^{1/4}}{\sqrt{a}}\exp\bigg\{-f(t_0)-\frac{a\epsilon^2}{4\sqrt{T}}\bigg\} + \frac{\exp\{-f(T)\}}{f'(T)} \nonumber\\
&\le \frac{\sqrt{\pi}T^{1/4}}{\sqrt{a}}\exp\bigg\{-f(t_0)-\frac{a\epsilon^2}{4\sqrt{T}}\bigg\} + \frac{1}{a\sqrt{T}-b\sqrt{s}}\exp\bigg\{-f(t_0)-\frac{a(T-t_0)^2}{4\sqrt{T}}\bigg\} \nonumber\\
&\le \bigg(\frac{\sqrt{\pi}T^{1/4}}{\sqrt{a}}+\frac{1}{a\sqrt{T}-b\sqrt{s}}\bigg)\exp\bigg\{-f(t_0)-\frac{a\epsilon^2}{4\sqrt{T}}\bigg\}. \nonumber
\end{align}
The claimed bound follows upon taking $T = 4t_0$. 
\end{proof}

\subsection{Crossing (comparison) lemma}

The next lemma states a  well-known monotonicity of  planar first- and last-passage percolation.  Different proofs are given in Lemma 6.2  of \cite{rass-cgm-18} and in   Lemma 4.6 of \cite{sepp-cgm-18}  . 

\begin{lem}
\label{LCros}   Let the LPP values $\G_{p,q}(m, n)$ be defined as in \eqref{Elpp} with  arbitrary real weights $\{\w(i,j)\}$.   Then the following inequalities hold. 
\begin{align}
\begin{split}
\G_{i, j}(m+1, n)-\G_{i+1, j}(m+1, n) &\le  \G_{i, j}(m, n)-\G_{i+1, j}(m, n) \\
& \le
 \G_{i, j}(m, n+1)-\G_{i+1, j}(m, n+1) ;   \label{ECr1}
 \end{split} \\[3pt] 
\begin{split}
\G_{i, j}(m, n+1)-\G_{i, j+1}(m, n+1)   &\le  \G_{i, j}(m, n)-\G_{i, j+1}(m, n) \\
&\le  \G_{i, j}(m+1, n)-\G_{i, j+1}(m+1, n) . 
 \label{ECr2}
 \end{split} 
\end{align}
\end{lem}

\bibliographystyle{habbrv}
\bibliography{refs}

\end{document}